\documentclass[11pt, a4paper]{amsart}
\pdfoutput=1

\usepackage{amsmath, amssymb, amsthm}
\usepackage{chngcntr}
\usepackage[utf8]{inputenc}
\usepackage{tikz}
\usepackage{subcaption}
\usepackage[normalem]{ulem}

\usepackage[numbers,sort&compress]{natbib}

\usepackage[margin=0.75in]{geometry}

\usepackage{fancyhdr}
\usepackage{amsaddr}
\usepackage{color}
\usepackage{enumitem}
\usepackage{dsfont}
\usepackage{cleveref}

\usepackage{xargs}

\newtheorem{theorem}{Theorem}[section]
\newtheorem{proposition}[theorem]{Proposition}
\newtheorem{lemma}[theorem]{Lemma}
\newtheorem{corollary}[theorem]{Corollary}

\newtheorem{remark}[theorem]{Remark}
\newtheorem{definition}[theorem]{Definition}
\newtheorem{example}[theorem]{Example}
\newtheorem{main result}[theorem]{Main Result}

\makeatletter
\newtheorem*{rep@theorem}{\rep@title}
\newcommand{\newreptheorem}[2]{%
\newenvironment{rep#1}[1]{%
 \def\rep@title{#2 \ref{##1}}%
 \begin{rep@theorem}}%
 {\end{rep@theorem}}}
\makeatother

\newreptheorem{theorem}{Theorem}
\newreptheorem{definition}{Definition}

\newtheorem*{theorem*}{Theorem}

\counterwithin{equation}{section}
\counterwithin{theorem}{section}

\newcommand{\alert}[1]{{#1}}

\newcommand{\bs}{\boldsymbol}
\newcommand{\mc}{\mathcal}

\newcommand{\R}{\mathbb{R}}

\newcommand{\N}{\mathbb{N}}
\renewcommand{\P}{\mathbb{P}}

\newcommand{\Lp}[1][p]{\ensuremath{L^{#1}}}
\newcommandx{\Lpq}[2][1=p, 2=q, usedefault=@]{\ensuremath{L^{#1, #2}}}
\newcommand{\lp}[1][p]{\ensuremath{\ell^{#1}}}

\newcommandx{\Bsqp}[3][1=\alpha, 2=q, 3=p, usedefault=@]{\ensuremath{B^{#1}_{#3,#2}}}
\newcommandx{\Lipsp}[2][1=\alpha, 2=p, usedefault=@]{\ensuremath{\mathrm{Lip}(#1,#2)}}
\newcommandx{\Wkp}[2][1=\alpha, 2=p, usedefault=@]{\ensuremath{W^{#1, #2}}}

\newcommandx{\XY}[2][1=\theta, 2=q, usedefault=@]{\ensuremath{(X, Y)_{#1,#2}}}

\newcommandx{\tbd}[2][1=b, 2=d, usedefault=@]{\ensuremath{t_{#1,#2}}}
\newcommandx{\Tbd}[2][1=b, 2=d,  usedefault=@]{\ensuremath{T_{#1,#2}}}
\newcommand{\Ib}[1][b]{\ensuremath{I_{#1}}}
\newcommand{\set}[1]{\ensuremath{\left\{#1\right\}}}
\newcommand{\tensor}[1]{\ensuremath{\boldsymbol{#1}}}
\newcommandx{\norm}[2][2=]{\ensuremath{\left\| #1 \right\|_{#2}}}
\newcommandx{\snorm}[2][2=]{\ensuremath{\left| #1 \right|_{#2}}}
\newcommandx{\Vbd}[3][1=b, 2=d, usedefault=@]{\ensuremath{V_{#1,#2,#3}}}
\newcommandx{\Vb}[2][1=b, usedefault=@]{\ensuremath{V_{#1, #2}}}
\renewcommand*\d{\mathop{}\!\mathrm{d}}
\newcommand{\Umin}[1]{\ensuremath{U^{\mathrm{min}}_{#1}}}
\newcommand{\linspan}[1]{\ensuremath{\operatorname{span}{\left\{#1\right\}}}}

\newcommand{\TT}[2]{\ensuremath{\mathcal{TT}_{#2}\left(#1\right)}}

\newcommandx{\res}[3][1=d, 2=\bar{d}]{\ensuremath{\mathcal{R}_{#1\rightarrow #2}#3}}
\newcommandx{\ext}[3][1=d, 2=\bar{d}]{\ensuremath{\mathcal{E}_{#1\rightarrow #2}#3}}

\newcommandx{\E}[3][1=n,3=,usedefault=@]{\ensuremath{E_{#1}\left(#2\right)_{#3}}}
\newcommandx{\tool}[1][1=n, usedefault=@]{\ensuremath{\Phi_{#1}}}
\newcommandx{\Asq}[2][1=\alpha, 2=q, usedefault=@]{\ensuremath{A^{#1}_{#2}}}
\newcommand{\rJ}{\ensuremath{r_{\mathrm{J}}}}
\newcommand{\rB}{\ensuremath{r_{\mathrm{B}}}}
\DeclareMathOperator{\cost}{compl}
\newcommand{\growth}{\ensuremath{\gamma}}
\newcommand{\rmax}{\ensuremath{r_{\max}}}

\newcommand{\indicator}[1]{\ensuremath{\mathds{1}_{#1}}}

\newcommandx{\interpol}[2][1=b, 2=d, usedefault=@]{\ensuremath{\mathcal{I}_{#1,#2}}}
\newcommandx{\li}[4][2=b, 3=l, 4=j, usedefault=@]{\ensuremath{#1_{#2,#3,#4}}}

\newcommand{\diff}{\ensuremath{\Delta}}
\renewcommand{\mod}{\ensuremath{\omega}}

\newcommand{\w}{\ensuremath{\mathrm{w}}}

\newcommand{\nn}{\ensuremath{\Psi}}
\newcommand{\real}{\ensuremath{\mathcal{R}}}

\DeclareMathOperator*{\esssup}{ess\,sup}
\newcommand{\fbar}{\ensuremath{\bar{f}}}
\newcommand{\abar}{\ensuremath{\bar{a}}}
\newcommand{\bbar}{\ensuremath{\bar{b}}}
\newcommand{\xbar}{\ensuremath{\bar{x}}}
\newcommand{\tbar}{\ensuremath{\bar{t}}}
\renewcommand{\hbar}{\ensuremath{\bar{h}}}

\title{Approximation Theory of Tree Tensor Networks:
Tensorized Univariate Functions -- Part I}

\author{Mazen Ali and Anthony Nouy}
\address{Fraunhofer ITWM, 67663 Kaiserslautern, Germany}
\address{Centrale Nantes, Nantes Universit\'e, LMJL UMR CNRS 6629, France}
\email{mazen.ali@itwm.fraunhofer.de}
\email{anthony.nouy@ec-nantes.fr}

\thanks{Acknowledgments: The authors acknowledge AIRBUS Group for the financial support with the project AtRandom.}
\date{\today}


\begin{document}

\begin{abstract}
    We study the approximation of functions
    by tensor networks (TNs).
    We show that Lebesgue $\Lp$-spaces in one dimension
    can be identified with tensor product spaces of
    arbitrary order through tensorization.
    We use this tensor product structure
    to define subsets of $\Lp$
    of rank-structured functions of finite representation
    complexity.
    These subsets are then used to define
    different
    approximation classes of tensor networks, associated with different measures of complexity. These approximation classes 
 are shown to be quasi-normed linear
    spaces. We study some elementary properties
    and relationships of said spaces.
    
    In part II of this work,
    we will show that classical smoothness (Besov) spaces
    are continuously embedded into these approximation classes.
    We will also show that functions in these approximation
    classes do not possess any Besov smoothness, unless
    one restricts the \emph{depth} of the tensor networks.
    
    The results of this work are both an analysis of the
    approximation spaces of TNs and
    a study of the expressivity of a particular type of neural
    networks (NN) -- namely feed-forward sum-product networks with sparse architecture. The input variables of this network
    result from the tensorization step, interpreted as a particular featuring step which can also be implemented  with a neural network with a specific architecture.
    We point out interesting parallels to recent results
    on the expressivity of rectified linear unit (ReLU) networks
    -- currently one of the most popular type of NNs.
\end{abstract}

\maketitle

\section{Introduction}


\subsection{Approximation of Functions}

We present a new perspective and therewith a
new tool for approximating
one-dimensional real-valued functions $f:\Omega\rightarrow\R$
on bounded intervals $\Omega\subset\R$.
We focus on the one-dimensional setting to keep
the presentation comprehensible,
but we intend to address the multi-dimensional
setting in a forthcoming part III.

The approximation of general functions by
simpler ``building blocks" has been a central topic in mathematics
for centuries with many arising methods:
algebraic polynomials,
trigonometric polynomials, splines, wavelets or rational functions
are among some of the by now established tools.
Recently,
more sophisticated tools such as tensor networks (TNs) or
neural networks (NNs) have proven to be powerful techniques.
Approximation methods find application in various
areas: signal processing, data compression,
pattern recognition, statistical learning, differential equations, uncertainty quantification, and so on.
See
\cite{Braess2007,
Bartlett2009, Mallat2009, MichaelMarcellin2013, Boche2015,  Bachmayr2016, nouy:2017_morbook, Orus2019, grelier:2018,Grelier:2019,Michel2020}
for examples.

In the 20th century a deep mathematical theory of approximation
has been established. It is by now well understood
that approximability properties of a function by more standard
tools -- such as polynomials or splines -- are closely related to
its smoothness.
Moreover, functions that can be approximated with
a certain rate can be grouped to form quasi-normed linear spaces.
Varying the approximation rate then generates an entire
scale of spaces that turn out to be so-called interpolation spaces.
See \cite{DeVore93, DeVore98} for more details.

In this work, we address the classical question of one-dimensional function approximation
but with a new set of tools relying on tensorization of functions and the use of rank-structured tensor formats (or tensor networks).
We analyze the resulting approximation classes.
We will show in part II \cite{partII} that many known classical spaces of smoothness
are embedded in these newly defined approximation classes.
On the other hand, we will also show that these classes are, in a sense,
much larger than classical smoothness spaces.


\subsection{Neural Networks}

Our work was partly motivated by current developments in the field of deep learning. Originally developed in \cite{McCulloch1943}, artificial neural networks (NNs) were inspired by models in theoretical neurophysiology
of the nervous system that constitutes a brain. The intention behind NNs was
to construct a mathematical
(and ultimately digital) analogue
of a biological neural network. The increase in computational power in recent decades has led to many successful applications of NNs in various fields
\cite{Schmidhuber2015}. This in turn has sparked interest in a better mathematical foundation for NNs. One flavor of this research is the analysis of the approximation power of feed-forward NNs. In particular, relevant for this work is the recent paper \cite{Gribonval2019}, where the authors analyzed the approximation spaces of
deep ReLU and rectified power unit (RePU) networks. In \Cref{fig:ann}, we sketch an example of a feed-forward NN.
\begin{figure}[ht]
    \center
    \tikzset{every picture/.style={line width=0.75pt}} 

\begin{tikzpicture}[x=0.75pt,y=0.75pt,yscale=-.7,xscale=.7]

\draw    (355.33,27.33) -- (503.73,66.49) ;
\draw [shift={(505.67,67)}, rotate = 194.78] [color={rgb, 255:red, 0; green, 0; blue, 0 }  ][line width=0.75]    (10.93,-3.29) .. controls (6.95,-1.4) and (3.31,-0.3) .. (0,0) .. controls (3.31,0.3) and (6.95,1.4) .. (10.93,3.29)   ;
\draw    (355.33,124.33) -- (503.73,84.85) ;
\draw [shift={(505.67,84.33)}, rotate = 525.1] [color={rgb, 255:red, 0; green, 0; blue, 0 }  ][line width=0.75]    (10.93,-3.29) .. controls (6.95,-1.4) and (3.31,-0.3) .. (0,0) .. controls (3.31,0.3) and (6.95,1.4) .. (10.93,3.29)   ;
\draw    (355.33,233.33) -- (503.75,187.92) ;
\draw [shift={(505.67,187.33)}, rotate = 522.99] [color={rgb, 255:red, 0; green, 0; blue, 0 }  ][line width=0.75]    (10.93,-3.29) .. controls (6.95,-1.4) and (3.31,-0.3) .. (0,0) .. controls (3.31,0.3) and (6.95,1.4) .. (10.93,3.29)   ;
\draw    (355.33,233.33) -- (512.87,86.86) ;
\draw [shift={(514.33,85.5)}, rotate = 497.08] [color={rgb, 255:red, 0; green, 0; blue, 0 }  ][line width=0.75]    (10.93,-3.29) .. controls (6.95,-1.4) and (3.31,-0.3) .. (0,0) .. controls (3.31,0.3) and (6.95,1.4) .. (10.93,3.29)   ;
\draw    (355.33,124.33) -- (503.75,169.42) ;
\draw [shift={(505.67,170)}, rotate = 196.9] [color={rgb, 255:red, 0; green, 0; blue, 0 }  ][line width=0.75]    (10.93,-3.29) .. controls (6.95,-1.4) and (3.31,-0.3) .. (0,0) .. controls (3.31,0.3) and (6.95,1.4) .. (10.93,3.29)   ;
\draw    (231.33,84.33) -- (343.85,33.33) ;
\draw [shift={(345.67,32.5)}, rotate = 515.61] [color={rgb, 255:red, 0; green, 0; blue, 0 }  ][line width=0.75]    (10.93,-3.29) .. controls (6.95,-1.4) and (3.31,-0.3) .. (0,0) .. controls (3.31,0.3) and (6.95,1.4) .. (10.93,3.29)   ;
\draw    (231.33,84.33) -- (342.12,123.66) ;
\draw [shift={(344,124.33)}, rotate = 199.55] [color={rgb, 255:red, 0; green, 0; blue, 0 }  ][line width=0.75]    (10.93,-3.29) .. controls (6.95,-1.4) and (3.31,-0.3) .. (0,0) .. controls (3.31,0.3) and (6.95,1.4) .. (10.93,3.29)   ;
\draw    (231.33,167.33) -- (342.27,232.32) ;
\draw [shift={(344,233.33)}, rotate = 210.36] [color={rgb, 255:red, 0; green, 0; blue, 0 }  ][line width=0.75]    (10.93,-3.29) .. controls (6.95,-1.4) and (3.31,-0.3) .. (0,0) .. controls (3.31,0.3) and (6.95,1.4) .. (10.93,3.29)   ;
\draw    (231.33,84.33) -- (347.38,221.97) ;
\draw [shift={(348.67,223.5)}, rotate = 229.87] [color={rgb, 255:red, 0; green, 0; blue, 0 }  ][line width=0.75]    (10.93,-3.29) .. controls (6.95,-1.4) and (3.31,-0.3) .. (0,0) .. controls (3.31,0.3) and (6.95,1.4) .. (10.93,3.29)   ;
\draw    (231.33,167.33) -- (342.76,131.12) ;
\draw [shift={(344.67,130.5)}, rotate = 522] [color={rgb, 255:red, 0; green, 0; blue, 0 }  ][line width=0.75]    (10.93,-3.29) .. controls (6.95,-1.4) and (3.31,-0.3) .. (0,0) .. controls (3.31,0.3) and (6.95,1.4) .. (10.93,3.29)   ;
\draw    (231.33,167.33) -- (347.33,38.98) ;
\draw [shift={(348.67,37.5)}, rotate = 492.1] [color={rgb, 255:red, 0; green, 0; blue, 0 }  ][line width=0.75]    (10.93,-3.29) .. controls (6.95,-1.4) and (3.31,-0.3) .. (0,0) .. controls (3.31,0.3) and (6.95,1.4) .. (10.93,3.29)   ;
\draw    (80.34,209.67) -- (224.09,96.74) ;
\draw [shift={(225.67,95.5)}, rotate = 501.85] [color={rgb, 255:red, 0; green, 0; blue, 0 }  ][line width=0.75]    (10.93,-3.29) .. controls (6.95,-1.4) and (3.31,-0.3) .. (0,0) .. controls (3.31,0.3) and (6.95,1.4) .. (10.93,3.29)   ;
\draw    (80.34,126.67) -- (217.73,90.02) ;
\draw [shift={(219.67,89.5)}, rotate = 525.06] [color={rgb, 255:red, 0; green, 0; blue, 0 }  ][line width=0.75]    (10.93,-3.29) .. controls (6.95,-1.4) and (3.31,-0.3) .. (0,0) .. controls (3.31,0.3) and (6.95,1.4) .. (10.93,3.29)   ;
\draw    (80.34,47.67) -- (218.07,83.83) ;
\draw [shift={(220,84.33)}, rotate = 194.71] [color={rgb, 255:red, 0; green, 0; blue, 0 }  ][line width=0.75]    (10.93,-3.29) .. controls (6.95,-1.4) and (3.31,-0.3) .. (0,0) .. controls (3.31,0.3) and (6.95,1.4) .. (10.93,3.29)   ;
\draw    (80.34,211.67) -- (219.74,174.02) ;
\draw [shift={(221.67,173.5)}, rotate = 524.89] [color={rgb, 255:red, 0; green, 0; blue, 0 }  ][line width=0.75]    (10.93,-3.29) .. controls (6.95,-1.4) and (3.31,-0.3) .. (0,0) .. controls (3.31,0.3) and (6.95,1.4) .. (10.93,3.29)   ;
\draw    (80.34,126.67) -- (218.08,166.77) ;
\draw [shift={(220,167.33)}, rotate = 196.23] [color={rgb, 255:red, 0; green, 0; blue, 0 }  ][line width=0.75]    (10.93,-3.29) .. controls (6.95,-1.4) and (3.31,-0.3) .. (0,0) .. controls (3.31,0.3) and (6.95,1.4) .. (10.93,3.29)   ;
\draw  [draw opacity=0][fill={rgb, 255:red, 208; green, 2; blue, 27 }  ,fill opacity=1 ] (80.34,40) -- (88.68,47.67) -- (80.34,55.33) -- (72,47.67) -- cycle ;
\draw  [draw opacity=0][fill={rgb, 255:red, 208; green, 2; blue, 27 }  ,fill opacity=1 ] (80.34,119) -- (88.68,126.67) -- (80.34,134.33) -- (72,126.67) -- cycle ;
\draw  [draw opacity=0][fill={rgb, 255:red, 208; green, 2; blue, 27 }  ,fill opacity=1 ] (80.34,204) -- (88.68,211.67) -- (80.34,219.33) -- (72,211.67) -- cycle ;
\draw  [draw opacity=0][fill={rgb, 255:red, 248; green, 231; blue, 28 }  ,fill opacity=1 ] (220,84.33) .. controls (220,78.07) and (225.07,73) .. (231.33,73) .. controls (237.59,73) and (242.67,78.07) .. (242.67,84.33) .. controls (242.67,90.59) and (237.59,95.67) .. (231.33,95.67) .. controls (225.07,95.67) and (220,90.59) .. (220,84.33) -- cycle ;
\draw  [draw opacity=0][fill={rgb, 255:red, 248; green, 231; blue, 28 }  ,fill opacity=1 ] (220,167.33) .. controls (220,161.07) and (225.07,156) .. (231.33,156) .. controls (237.59,156) and (242.67,161.07) .. (242.67,167.33) .. controls (242.67,173.59) and (237.59,178.67) .. (231.33,178.67) .. controls (225.07,178.67) and (220,173.59) .. (220,167.33) -- cycle ;
\draw  [draw opacity=0][fill={rgb, 255:red, 248; green, 231; blue, 28 }  ,fill opacity=1 ] (344,27.33) .. controls (344,21.07) and (349.07,16) .. (355.33,16) .. controls (361.59,16) and (366.67,21.07) .. (366.67,27.33) .. controls (366.67,33.59) and (361.59,38.67) .. (355.33,38.67) .. controls (349.07,38.67) and (344,33.59) .. (344,27.33) -- cycle ;
\draw  [draw opacity=0][fill={rgb, 255:red, 248; green, 231; blue, 28 }  ,fill opacity=1 ] (344,124.33) .. controls (344,118.07) and (349.07,113) .. (355.33,113) .. controls (361.59,113) and (366.67,118.07) .. (366.67,124.33) .. controls (366.67,130.59) and (361.59,135.67) .. (355.33,135.67) .. controls (349.07,135.67) and (344,130.59) .. (344,124.33) -- cycle ;
\draw  [draw opacity=0][fill={rgb, 255:red, 248; green, 231; blue, 28 }  ,fill opacity=1 ] (344,233.33) .. controls (344,227.07) and (349.07,222) .. (355.33,222) .. controls (361.59,222) and (366.67,227.07) .. (366.67,233.33) .. controls (366.67,239.59) and (361.59,244.67) .. (355.33,244.67) .. controls (349.07,244.67) and (344,239.59) .. (344,233.33) -- cycle ;
\draw  [draw opacity=0][fill={rgb, 255:red, 126; green, 211; blue, 33 }  ,fill opacity=1 ] (505.67,67) -- (523,67) -- (523,84.33) -- (505.67,84.33) -- cycle ;
\draw  [draw opacity=0][fill={rgb, 255:red, 126; green, 211; blue, 33 }  ,fill opacity=1 ] (505.67,170) -- (523,170) -- (523,187.33) -- (505.67,187.33) -- cycle ;

\end{tikzpicture}
    \caption{Example of an artificial neural network. On the left we have the \emph{input} nodes marked in red that
represent input data to the neural system.
The yellow nodes are the \emph{neurons} that
perform some simple operations on the input.
The edges between the nodes represent
synapses or \emph{connections} that transfer (after
possibly applying an affine linear transformation) the
output of one node into the input of another.
The final green nodes are the \emph{output} nodes.
In this particular example the
number of \emph{layers} $L$ is three, with two \emph{hidden layers}.}
    \label{fig:ann}
\end{figure}
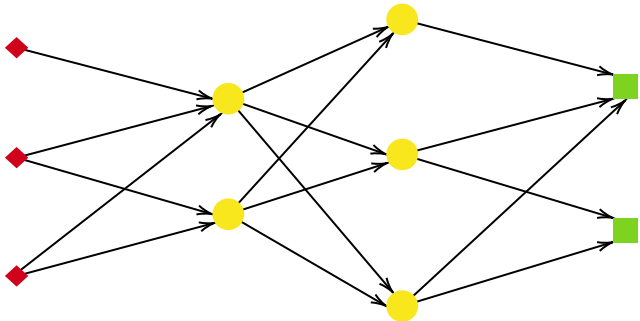


Mathematically, a feed-forward NN can be represented by a tuple
\begin{align*}
    \nn = \left([T_1,\sigma_1],\ldots,[T_L,\sigma_L]\right),
\end{align*}
where $T_l:\R^{N_{l-1}}\rightarrow\R^{N_{l}}$ are
affine maps; 
$N_l$ is the number of neurons in layer $l$, with
$N_0$ being the number of inputs and $N_L$ the
number of outputs;
the functions $\sigma_l:\R^l\rightarrow\R^l$
are typically non-linear and represent
the operations performed on data inside the neurons.
The functions $\sigma_l$ are often implemented
via a component-wise application of a single
non-linear function $\rho:\R\rightarrow\R$
referred to as the \emph{activation} function. 
The realization $\real(\nn)$ of the NN $\nn$ is the function
\begin{align*}
    \real(\nn):\R^{N_0}\rightarrow\R^{N_L},\quad
    \real(\nn):=\sigma_L\circ T_L\circ\cdots\circ\sigma_1\circ T_1.
\end{align*}
Before one can proceed with \emph{training} (i.e., estimating)
a NN, one usually specifies the \emph{architecture} of
the NN: this includes the number of layers $L$, neurons $N_l$, 
connections between the neurons and
the non-linearities $\sigma_l$.
Only after this, one proceeds with training
which entails determining the affine maps
$T_l$. This is typically done by minimizing some \emph{loss}
or distance functional $\mc J(\real(\nn), f)$,
where $f:\R^{N_0}\rightarrow\R^{N_L}$ is the target function
that we want to approximate, in some sense. 
If we set $N_0=N_L=1$ and, e.g.,
$\mc J(\real(\nn), f)=\norm{f-\real(\nn)}[p]$, then we are in
the classical setting of one-dimensional approximation\footnote{Though one
would typically not train a network by directly minimizing the $\Lp$-norm.}
(in the $\Lp$-norm).


\subsection{Tensor Networks}\label{sec:introtns}

Tensor networks have been studied in parallel
in different fields, sometimes under different names: e.g.,
hierarchical tensor formats in numerical analysis,  sum-product networks in machine learning, belief networks in bayesian inference.
Tensor networks are  commonly
applied and studied in condensed matter physics, where understanding phenomena in quantum many-body systems
has proven to be a challenging problem, to a large extent
due to the sheer amount of dependencies that cannot be simulated
even on the most powerful computers (see \cite{Orus2014} for a non-technical introduction).
In all these fields, a common challenging problem is the approximation 
of  functions of a very large number
of  variables. 
This led to the development of tools tailored to
so-called \emph{high-dimensional} problems.

The mathematical principle behind tensor networks is to use tensor
products. 
For approximating a $d$-variate function $f$, 
there are several types of tensor formats.
The simplest is the so-called \emph{$r$-term} or \emph{CP} format, where $f$ is approximated as
\begin{align}\label{eq:cp}
   f(x_1,\hdots,,x_d)\approx  \sum_{k=1}^r v_1^k(x_1)\cdots v^k_d(x_d).
\end{align}
If each factor $v^k_\nu$ is encoded with $N$ parameters, the total number of  parameters  is thus $dNr$, which is linear in the number of variables. 
The approximation format \eqref{eq:cp} is successful in many applications (chemometrics, inverse problems in signal processing...)
 but due to 
 a few unfavorable properties (see 
\cite[Chapter 9]{Hackbusch2012}), different types of tensor formats 
are frequently used in numerical approximation. 
In particular, with the so-called \emph{tensor train} (TT) format or \emph{matrix product state} (MPS), the function $f$ is approximated as 
\begin{align}\label{eq:TTex}
    f(x_1,\hdots,,x_d)\approx 
  \sum_{k_1=1}^{r_1} \hdots 
    \sum_{k_{d-1}=1}^{r_{d-1}} v_1^{k_1}(x_1)v_2^{k_1,k_2}(x_2) \hdots v_{d-1}^{k_{d-2},k_{d-1}}(x_{d-1}) v_d^{k_{d-1}}(x_d).
\end{align}
The numbers $r_\nu$ are referred to as \emph{multi-linear ranks} or \emph{hierarchical ranks}. The rank $r_\nu$ is related to 
the classical notion of rank for bi-variate functions,  by identifying a $d$-variate function as a function of two complementary groups
of variables $(x_1,\hdots,x_\nu)$ and $(x_{\nu+1},\hdots,x_d)$.  It corresponds to the so-called $\beta$-rank $r_\beta$, with $\beta  = \{1,\hdots,\nu\}$.
The format in \eqref{eq:TTex} is a particular case of \emph{tree-based tensor formats}, or \emph{tree tensor networks} \cite{Hackbusch2012,Falco2018SEMA}, 
the TT format being associated with a linear dimension partition tree. 
Numerically, such formats have favorable stability properties
with robust algorithms (see \cite{Grasedyck2010,oseledets2009breaking,Nouy2019,grelier:2018}). Moreover, the corresponding tensor networks and decompositions
have a physical
interpretation in the context of entangled many-body systems,
see \cite{Orus2014, Orus2019}.

For more general tensor networks, we refer to \Cref{fig:tns} for graphical representations. 

\begin{figure}[ht!]
    \begin{subfigure}{0.47\textwidth}
        \center
        \tikzset{every picture/.style={line width=0.75pt}} 

\begin{tikzpicture}[x=0.75pt,y=0.75pt,yscale=-1,xscale=1]

\draw    (330,56.33) -- (330,30.33) ;
\draw    (381.67,30.33) -- (397,30.33) -- (425.33,30.33) ;
\draw    (232.33,30.33) -- (247.67,30.33) -- (276,30.33) ;
\draw    (276,30.33) -- (381.67,30.33) ;
\draw  [color={rgb, 255:red, 255; green, 255; blue, 255 }  ,draw opacity=1 ][fill={rgb, 255:red, 74; green, 144; blue, 226 }  ,fill opacity=1 ] (269,30.33) .. controls (269,26.47) and (272.13,23.33) .. (276,23.33) .. controls (279.87,23.33) and (283,26.47) .. (283,30.33) .. controls (283,34.2) and (279.87,37.33) .. (276,37.33) .. controls (272.13,37.33) and (269,34.2) .. (269,30.33) -- cycle ;
\draw  [color={rgb, 255:red, 255; green, 255; blue, 255 }  ,draw opacity=1 ][fill={rgb, 255:red, 74; green, 144; blue, 226 }  ,fill opacity=1 ] (374.67,30.33) .. controls (374.67,26.47) and (377.8,23.33) .. (381.67,23.33) .. controls (385.53,23.33) and (388.67,26.47) .. (388.67,30.33) .. controls (388.67,34.2) and (385.53,37.33) .. (381.67,37.33) .. controls (377.8,37.33) and (374.67,34.2) .. (374.67,30.33) -- cycle ;
\draw  [color={rgb, 255:red, 255; green, 255; blue, 255 }  ,draw opacity=1 ][fill={rgb, 255:red, 74; green, 144; blue, 226 }  ,fill opacity=1 ] (323,30.33) .. controls (323,26.47) and (326.13,23.33) .. (330,23.33) .. controls (333.87,23.33) and (337,26.47) .. (337,30.33) .. controls (337,34.2) and (333.87,37.33) .. (330,37.33) .. controls (326.13,37.33) and (323,34.2) .. (323,30.33) -- cycle ;

\draw (269,3.4) node [anchor=north west][inner sep=0.75pt]    {$v_{1}^{k_{1}}$};
\draw (374,2.4) node [anchor=north west][inner sep=0.75pt]    {$v_{3}^{k_{2}}$};
\draw (311,1.4) node [anchor=north west][inner sep=0.75pt]    {$v_{2}^{k_{1} ,k_{2}}$};

\end{tikzpicture}
        \caption{Tensor corresponding to \eqref{eq:TTex} with $d=3$.\label{fig:order3}}
    \end{subfigure}
    \begin{subfigure}{0.47\textwidth}
        \center
        \tikzset{every picture/.style={line width=0.75pt}} 

\begin{tikzpicture}[x=0.75pt,y=0.75pt,yscale=-1,xscale=1]

\draw    (183,49.33) -- (183,23.33) ;
\draw    (288.67,49.33) -- (288.67,23.33) ;
\draw    (338,49.33) -- (338,23.33) ;
\draw    (443.67,49.33) -- (443.67,23.33) ;
\draw    (294.33,23.33) -- (309.67,23.33) -- (338,23.33) ;
\draw    (237,49.33) -- (237,23.33) ;
\draw    (288.67,23.33) -- (304,23.33) -- (332.33,23.33) ;
\draw    (183,23.33) -- (288.67,23.33) ;
\draw  [color={rgb, 255:red, 255; green, 255; blue, 255 }  ,draw opacity=1 ][fill={rgb, 255:red, 74; green, 144; blue, 226 }  ,fill opacity=1 ] (176,23.33) .. controls (176,19.47) and (179.13,16.33) .. (183,16.33) .. controls (186.87,16.33) and (190,19.47) .. (190,23.33) .. controls (190,27.2) and (186.87,30.33) .. (183,30.33) .. controls (179.13,30.33) and (176,27.2) .. (176,23.33) -- cycle ;
\draw  [color={rgb, 255:red, 255; green, 255; blue, 255 }  ,draw opacity=1 ][fill={rgb, 255:red, 74; green, 144; blue, 226 }  ,fill opacity=1 ] (281.67,23.33) .. controls (281.67,19.47) and (284.8,16.33) .. (288.67,16.33) .. controls (292.53,16.33) and (295.67,19.47) .. (295.67,23.33) .. controls (295.67,27.2) and (292.53,30.33) .. (288.67,30.33) .. controls (284.8,30.33) and (281.67,27.2) .. (281.67,23.33) -- cycle ;
\draw  [color={rgb, 255:red, 255; green, 255; blue, 255 }  ,draw opacity=1 ][fill={rgb, 255:red, 74; green, 144; blue, 226 }  ,fill opacity=1 ] (230,23.33) .. controls (230,19.47) and (233.13,16.33) .. (237,16.33) .. controls (240.87,16.33) and (244,19.47) .. (244,23.33) .. controls (244,27.2) and (240.87,30.33) .. (237,30.33) .. controls (233.13,30.33) and (230,27.2) .. (230,23.33) -- cycle ;
\draw    (392,49.33) -- (392,23.33) ;
\draw    (338,23.33) -- (443.67,23.33) ;
\draw  [color={rgb, 255:red, 255; green, 255; blue, 255 }  ,draw opacity=1 ][fill={rgb, 255:red, 74; green, 144; blue, 226 }  ,fill opacity=1 ] (331,23.33) .. controls (331,19.47) and (334.13,16.33) .. (338,16.33) .. controls (341.87,16.33) and (345,19.47) .. (345,23.33) .. controls (345,27.2) and (341.87,30.33) .. (338,30.33) .. controls (334.13,30.33) and (331,27.2) .. (331,23.33) -- cycle ;
\draw  [color={rgb, 255:red, 255; green, 255; blue, 255 }  ,draw opacity=1 ][fill={rgb, 255:red, 74; green, 144; blue, 226 }  ,fill opacity=1 ] (436.67,23.33) .. controls (436.67,19.47) and (439.8,16.33) .. (443.67,16.33) .. controls (447.53,16.33) and (450.67,19.47) .. (450.67,23.33) .. controls (450.67,27.2) and (447.53,30.33) .. (443.67,30.33) .. controls (439.8,30.33) and (436.67,27.2) .. (436.67,23.33) -- cycle ;
\draw  [color={rgb, 255:red, 255; green, 255; blue, 255 }  ,draw opacity=1 ][fill={rgb, 255:red, 74; green, 144; blue, 226 }  ,fill opacity=1 ] (385,23.33) .. controls (385,19.47) and (388.13,16.33) .. (392,16.33) .. controls (395.87,16.33) and (399,19.47) .. (399,23.33) .. controls (399,27.2) and (395.87,30.33) .. (392,30.33) .. controls (388.13,30.33) and (385,27.2) .. (385,23.33) -- cycle ;

\end{tikzpicture}
        \caption{General Tensor Train (TT) or Matrix Product State (MPS).\label{fig:mps}}
    \end{subfigure}
    \hfill
    \begin{subfigure}{0.47\textwidth}
        \center
        \tikzset{every picture/.style={line width=0.75pt}} 

\begin{tikzpicture}[x=0.75pt,y=0.75pt,yscale=-1,xscale=1]

\draw    (302.67,168.33) -- (302.67,193.33) ;
\draw    (258.67,165.33) -- (258.67,190.33) ;
\draw    (214.67,119.33) -- (228.67,165.33) ;
\draw    (258.67,165.33) -- (264.67,120.33) ;
\draw    (182.67,163.33) -- (182.67,188.33) ;
\draw    (276,30.33) -- (246.67,72.33) ;
\draw    (304.67,73.33) -- (276,30.33) ;
\draw    (214.67,119.33) -- (246.67,72.33) ;
\draw    (362.67,115.33) -- (304.67,73.33) ;
\draw    (214.67,119.33) -- (182.67,163.33) ;
\draw    (264.67,120.33) -- (302.67,168.33) ;
\draw    (228.67,165.33) -- (228.67,190.33) ;
\draw    (308.67,122.33) -- (308.67,147.33) ;
\draw    (362.67,115.33) -- (362.67,140.33) ;
\draw    (264.67,120.33) -- (246.67,72.33) ;
\draw    (304.67,73.33) -- (308.67,122.33) ;
\draw  [color={rgb, 255:red, 255; green, 255; blue, 255 }  ,draw opacity=1 ][fill={rgb, 255:red, 74; green, 144; blue, 226 }  ,fill opacity=1 ] (269,30.33) .. controls (269,26.47) and (272.13,23.33) .. (276,23.33) .. controls (279.87,23.33) and (283,26.47) .. (283,30.33) .. controls (283,34.2) and (279.87,37.33) .. (276,37.33) .. controls (272.13,37.33) and (269,34.2) .. (269,30.33) -- cycle ;
\draw  [color={rgb, 255:red, 255; green, 255; blue, 255 }  ,draw opacity=1 ][fill={rgb, 255:red, 74; green, 144; blue, 226 }  ,fill opacity=1 ] (239.67,72.33) .. controls (239.67,68.47) and (242.8,65.33) .. (246.67,65.33) .. controls (250.53,65.33) and (253.67,68.47) .. (253.67,72.33) .. controls (253.67,76.2) and (250.53,79.33) .. (246.67,79.33) .. controls (242.8,79.33) and (239.67,76.2) .. (239.67,72.33) -- cycle ;
\draw  [color={rgb, 255:red, 255; green, 255; blue, 255 }  ,draw opacity=1 ][fill={rgb, 255:red, 74; green, 144; blue, 226 }  ,fill opacity=1 ] (297.67,73.33) .. controls (297.67,69.47) and (300.8,66.33) .. (304.67,66.33) .. controls (308.53,66.33) and (311.67,69.47) .. (311.67,73.33) .. controls (311.67,77.2) and (308.53,80.33) .. (304.67,80.33) .. controls (300.8,80.33) and (297.67,77.2) .. (297.67,73.33) -- cycle ;
\draw  [color={rgb, 255:red, 255; green, 255; blue, 255 }  ,draw opacity=1 ][fill={rgb, 255:red, 74; green, 144; blue, 226 }  ,fill opacity=1 ] (207.67,119.33) .. controls (207.67,115.47) and (210.8,112.33) .. (214.67,112.33) .. controls (218.53,112.33) and (221.67,115.47) .. (221.67,119.33) .. controls (221.67,123.2) and (218.53,126.33) .. (214.67,126.33) .. controls (210.8,126.33) and (207.67,123.2) .. (207.67,119.33) -- cycle ;
\draw  [color={rgb, 255:red, 255; green, 255; blue, 255 }  ,draw opacity=1 ][fill={rgb, 255:red, 74; green, 144; blue, 226 }  ,fill opacity=1 ] (257.67,120.33) .. controls (257.67,116.47) and (260.8,113.33) .. (264.67,113.33) .. controls (268.53,113.33) and (271.67,116.47) .. (271.67,120.33) .. controls (271.67,124.2) and (268.53,127.33) .. (264.67,127.33) .. controls (260.8,127.33) and (257.67,124.2) .. (257.67,120.33) -- cycle ;
\draw  [color={rgb, 255:red, 255; green, 255; blue, 255 }  ,draw opacity=1 ][fill={rgb, 255:red, 74; green, 144; blue, 226 }  ,fill opacity=1 ] (301.67,122.33) .. controls (301.67,118.47) and (304.8,115.33) .. (308.67,115.33) .. controls (312.53,115.33) and (315.67,118.47) .. (315.67,122.33) .. controls (315.67,126.2) and (312.53,129.33) .. (308.67,129.33) .. controls (304.8,129.33) and (301.67,126.2) .. (301.67,122.33) -- cycle ;
\draw  [color={rgb, 255:red, 255; green, 255; blue, 255 }  ,draw opacity=1 ][fill={rgb, 255:red, 74; green, 144; blue, 226 }  ,fill opacity=1 ] (355.67,115.33) .. controls (355.67,111.47) and (358.8,108.33) .. (362.67,108.33) .. controls (366.53,108.33) and (369.67,111.47) .. (369.67,115.33) .. controls (369.67,119.2) and (366.53,122.33) .. (362.67,122.33) .. controls (358.8,122.33) and (355.67,119.2) .. (355.67,115.33) -- cycle ;
\draw  [color={rgb, 255:red, 255; green, 255; blue, 255 }  ,draw opacity=1 ][fill={rgb, 255:red, 74; green, 144; blue, 226 }  ,fill opacity=1 ] (175.67,163.33) .. controls (175.67,159.47) and (178.8,156.33) .. (182.67,156.33) .. controls (186.53,156.33) and (189.67,159.47) .. (189.67,163.33) .. controls (189.67,167.2) and (186.53,170.33) .. (182.67,170.33) .. controls (178.8,170.33) and (175.67,167.2) .. (175.67,163.33) -- cycle ;
\draw  [color={rgb, 255:red, 255; green, 255; blue, 255 }  ,draw opacity=1 ][fill={rgb, 255:red, 74; green, 144; blue, 226 }  ,fill opacity=1 ] (221.67,165.33) .. controls (221.67,161.47) and (224.8,158.33) .. (228.67,158.33) .. controls (232.53,158.33) and (235.67,161.47) .. (235.67,165.33) .. controls (235.67,169.2) and (232.53,172.33) .. (228.67,172.33) .. controls (224.8,172.33) and (221.67,169.2) .. (221.67,165.33) -- cycle ;
\draw  [color={rgb, 255:red, 255; green, 255; blue, 255 }  ,draw opacity=1 ][fill={rgb, 255:red, 74; green, 144; blue, 226 }  ,fill opacity=1 ] (251.67,165.33) .. controls (251.67,161.47) and (254.8,158.33) .. (258.67,158.33) .. controls (262.53,158.33) and (265.67,161.47) .. (265.67,165.33) .. controls (265.67,169.2) and (262.53,172.33) .. (258.67,172.33) .. controls (254.8,172.33) and (251.67,169.2) .. (251.67,165.33) -- cycle ;
\draw  [color={rgb, 255:red, 255; green, 255; blue, 255 }  ,draw opacity=1 ][fill={rgb, 255:red, 74; green, 144; blue, 226 }  ,fill opacity=1 ] (295.67,168.33) .. controls (295.67,164.47) and (298.8,161.33) .. (302.67,161.33) .. controls (306.53,161.33) and (309.67,164.47) .. (309.67,168.33) .. controls (309.67,172.2) and (306.53,175.33) .. (302.67,175.33) .. controls (298.8,175.33) and (295.67,172.2) .. (295.67,168.33) -- cycle ;

\end{tikzpicture}
        \caption{Hierarchical Tucker (HT) or a tree-based format.\label{fig:ht}}
    \end{subfigure}
    \begin{subfigure}{0.47\textwidth}
        \center
        \tikzset{every picture/.style={line width=0.75pt}} 

\begin{tikzpicture}[x=0.75pt,y=0.75pt,yscale=-1,xscale=1]

\draw    (203.5,99.42) -- (240,98.75) ;
\draw    (471,79.33) -- (497.5,79.33) ;
\draw    (424,30) -- (424,55.33) ;
\draw    (424,55.33) -- (471,79.33) ;
\draw    (424,105.33) -- (424,130.67) ;
\draw    (471,79.33) -- (424,105.33) ;
\draw    (314.5,98.75) -- (374,105.33) ;
\draw    (374,105.33) -- (374,130.67) ;
\draw   (374,55.33) -- (424,55.33) -- (424,105.33) -- (374,105.33) -- cycle ;
\draw    (307.5,98.75) -- (374,55.33) ;
\draw    (307.5,98.75) -- (322,31.33) ;
\draw    (254.5,45.75) -- (273.75,65) ;
\draw    (273.75,65) -- (322,31.33) ;
\draw   (240,98.75) .. controls (240,80.11) and (255.11,65) .. (273.75,65) .. controls (292.39,65) and (307.5,80.11) .. (307.5,98.75) .. controls (307.5,117.39) and (292.39,132.5) .. (273.75,132.5) .. controls (255.11,132.5) and (240,117.39) .. (240,98.75) -- cycle ;
\draw  [color={rgb, 255:red, 255; green, 255; blue, 255 }  ,draw opacity=1 ][fill={rgb, 255:red, 74; green, 144; blue, 226 }  ,fill opacity=1 ] (315,31.33) .. controls (315,27.47) and (318.13,24.33) .. (322,24.33) .. controls (325.87,24.33) and (329,27.47) .. (329,31.33) .. controls (329,35.2) and (325.87,38.33) .. (322,38.33) .. controls (318.13,38.33) and (315,35.2) .. (315,31.33) -- cycle ;
\draw  [color={rgb, 255:red, 255; green, 255; blue, 255 }  ,draw opacity=1 ][fill={rgb, 255:red, 74; green, 144; blue, 226 }  ,fill opacity=1 ] (233,98.75) .. controls (233,94.88) and (236.13,91.75) .. (240,91.75) .. controls (243.87,91.75) and (247,94.88) .. (247,98.75) .. controls (247,102.62) and (243.87,105.75) .. (240,105.75) .. controls (236.13,105.75) and (233,102.62) .. (233,98.75) -- cycle ;
\draw  [color={rgb, 255:red, 255; green, 255; blue, 255 }  ,draw opacity=1 ][fill={rgb, 255:red, 74; green, 144; blue, 226 }  ,fill opacity=1 ] (367,55.33) .. controls (367,51.47) and (370.13,48.33) .. (374,48.33) .. controls (377.87,48.33) and (381,51.47) .. (381,55.33) .. controls (381,59.2) and (377.87,62.33) .. (374,62.33) .. controls (370.13,62.33) and (367,59.2) .. (367,55.33) -- cycle ;
\draw  [color={rgb, 255:red, 255; green, 255; blue, 255 }  ,draw opacity=1 ][fill={rgb, 255:red, 74; green, 144; blue, 226 }  ,fill opacity=1 ] (266.75,132.5) .. controls (266.75,128.63) and (269.88,125.5) .. (273.75,125.5) .. controls (277.62,125.5) and (280.75,128.63) .. (280.75,132.5) .. controls (280.75,136.37) and (277.62,139.5) .. (273.75,139.5) .. controls (269.88,139.5) and (266.75,136.37) .. (266.75,132.5) -- cycle ;
\draw  [color={rgb, 255:red, 255; green, 255; blue, 255 }  ,draw opacity=1 ][fill={rgb, 255:red, 74; green, 144; blue, 226 }  ,fill opacity=1 ] (300.5,98.75) .. controls (300.5,94.88) and (303.63,91.75) .. (307.5,91.75) .. controls (311.37,91.75) and (314.5,94.88) .. (314.5,98.75) .. controls (314.5,102.62) and (311.37,105.75) .. (307.5,105.75) .. controls (303.63,105.75) and (300.5,102.62) .. (300.5,98.75) -- cycle ;
\draw  [color={rgb, 255:red, 255; green, 255; blue, 255 }  ,draw opacity=1 ][fill={rgb, 255:red, 74; green, 144; blue, 226 }  ,fill opacity=1 ] (266.75,65) .. controls (266.75,61.13) and (269.88,58) .. (273.75,58) .. controls (277.62,58) and (280.75,61.13) .. (280.75,65) .. controls (280.75,68.87) and (277.62,72) .. (273.75,72) .. controls (269.88,72) and (266.75,68.87) .. (266.75,65) -- cycle ;
\draw  [color={rgb, 255:red, 255; green, 255; blue, 255 }  ,draw opacity=1 ][fill={rgb, 255:red, 74; green, 144; blue, 226 }  ,fill opacity=1 ] (367,105.33) .. controls (367,101.47) and (370.13,98.33) .. (374,98.33) .. controls (377.87,98.33) and (381,101.47) .. (381,105.33) .. controls (381,109.2) and (377.87,112.33) .. (374,112.33) .. controls (370.13,112.33) and (367,109.2) .. (367,105.33) -- cycle ;
\draw  [color={rgb, 255:red, 255; green, 255; blue, 255 }  ,draw opacity=1 ][fill={rgb, 255:red, 74; green, 144; blue, 226 }  ,fill opacity=1 ] (417,55.33) .. controls (417,51.47) and (420.13,48.33) .. (424,48.33) .. controls (427.87,48.33) and (431,51.47) .. (431,55.33) .. controls (431,59.2) and (427.87,62.33) .. (424,62.33) .. controls (420.13,62.33) and (417,59.2) .. (417,55.33) -- cycle ;
\draw  [color={rgb, 255:red, 255; green, 255; blue, 255 }  ,draw opacity=1 ][fill={rgb, 255:red, 74; green, 144; blue, 226 }  ,fill opacity=1 ] (417,105.33) .. controls (417,101.47) and (420.13,98.33) .. (424,98.33) .. controls (427.87,98.33) and (431,101.47) .. (431,105.33) .. controls (431,109.2) and (427.87,112.33) .. (424,112.33) .. controls (420.13,112.33) and (417,109.2) .. (417,105.33) -- cycle ;
\draw  [color={rgb, 255:red, 255; green, 255; blue, 255 }  ,draw opacity=1 ][fill={rgb, 255:red, 74; green, 144; blue, 226 }  ,fill opacity=1 ] (464,79.33) .. controls (464,75.47) and (467.13,72.33) .. (471,72.33) .. controls (474.87,72.33) and (478,75.47) .. (478,79.33) .. controls (478,83.2) and (474.87,86.33) .. (471,86.33) .. controls (467.13,86.33) and (464,83.2) .. (464,79.33) -- cycle ;

\end{tikzpicture}
        \caption{General tensor network. Can be seen
        as an instance of Projected Entangled Pair States (PEPS).\label{fig:peps}}
    \end{subfigure}
    \caption{Examples of tensor networks. The vertices in \Cref{fig:tns} represent the low-dimensional functions in
the decomposition, such as $v^1, \hdots,v^d$ in \eqref{eq:TTex}.
The edges between the vertices represent summation over an index
(\emph{contraction)}
between two functions, such as summation over $k_\nu$ in
\eqref{eq:TTex}. The free edges represent input variables $x_1,\hdots,x_d$ in \eqref{eq:TTex}.}
    \label{fig:tns}
\end{figure}
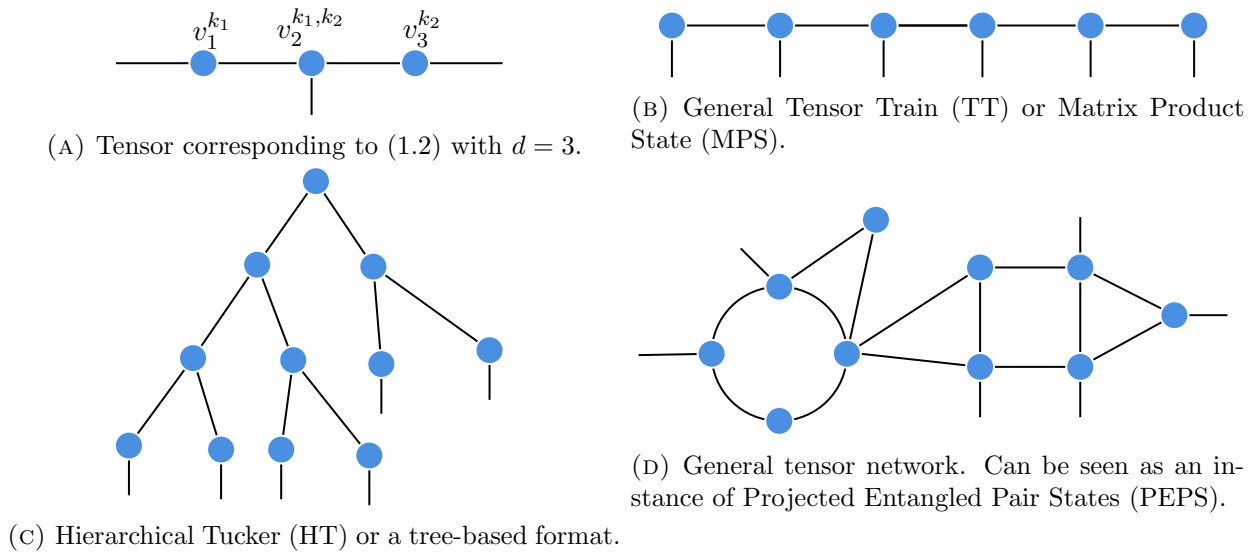

The specific choice of a tensor network is sometimes suggested
by the problem at hand: e.g., in quantum physics
by the entanglement/interaction structure of the
quantum system that $f$ is to model --
see, e.g., \cite{Hastings2007, Ali2019, Schwarz2017, Arad2013}.

At first glance, it seems that tensor networks are a tool
suited only for approximating high-dimensional functions.
However, such formats can be applied in any multi-variate setting
and this multi-variate setting can be enforced even if $d=1$ by a
 ``coarse-graining'' of an interval in $\R$ allowing to identify a one-dimensional function with a multi-variate function (or tensor). This identification is the \emph{tensorization} of functions which is at the core of the approximation tools considered in this work. 
It was originally applied for matrices in
\cite{Oseledets2009} and later coined as
\emph{quantized tensor format} when tensorization is combined with the use of a tensor format. 

In high-dimensional approximation,
keeping the ranks $r_\beta$ small relies
on the correct choice of the tensor network that
``fits'' the interaction structure as hinted above.
For the approximation of tensorized functions
a different type of structure is required.
In \cite{Grasedyck2010techreport, Oseledets2009}, it was shown
that, if $f$ is vector of evaluations of a polynomial on a grid,
then choosing the TT format yields
hierarchical ranks that are bounded by the degree of the polynomial.
Similar statements were shown for trigonometric polynomials and exponential
functions.
In \cite{Kazeev2017}, this fact was utilized to show that a
finite element approximation of two-dimensional functions with singularities,
where the coefficient vector was stored in a quantized tensor format,
automatically recovers an exponential rate of convergence,
analogue to that of $hp$-approximation.

This work can be seen as a consolidation and a deeper analysis of 
approximation of one-dimensional functions using quantized tensor formats.
We first show that Lebesgue spaces of $p$-integrable functions
are isometric to tensor product spaces of any order and analyze
some basic properties of this identification.
We then define and analyze the approximation classes of $\Lp$ functions
that can be approximated by rank-structured functions
in the TT format with a certain rate. In Part II \cite{partII},
we will show direct and (lack of) inverse embeddings.


\subsection{Tensor vs.\ Neural Networks}
 
Recently multiple connections between TNs and NNs
have been discovered.
In \cite{Beny2013}, the author exhibits an analogy between
the Renormalization Group (RG) -- the fundamental physical concept
behind TNs -- and deep learning,
with the scale in RG being akin to depth in NNs.
In \cite{Levine2017}, it was observed that in fact
tree tensor networks can be viewed as a specific type of feed-forward NNs
with multi-linear functions $\sigma_l$, namely sum-product NNs (or arithmetic circuits)  \cite{poon2011sum}.
This connection offers intriguing perspectives
and it can be exploited both ways:
we can use our fundamental
understanding of quantum entanglement\footnote{That is not to say
that we have understood quantum entanglement.
But an argument can be made that our understanding
of entanglement and thus tensor networks offers a different perspective on neural networks.}
to measure and design NNs, see \cite{Levine2017, Cohen2016}.
Or we can use NNs
to \emph{learn} entanglement and augment
tensor networks to better represent
highly entangled systems, see \cite{Carleo2017, Levine2019}.
Tensor networks also
offer a big choice of well studied and robust
numerical algorithms.

In this spirit, our work in part II \cite{partII} can be seen as a result
on the approximation power of a particular type
of NN,
where the TT format is a
feed-forward sum-product NN and a recurrent neural network architecture.
When compared to the results
of \cite{Gribonval2019} on approximation classes
of
RePU networks, we will observe in Part II \cite{partII} that
both tools achieve optimal approximation order for Sobolev spaces.
We also show that TNs (using the TT format) achieve optimal approximation
order for Besov spaces on the embedding line (corresponding to
non-linear approximation). These statements hold for
a tensor network of fixed polynomial degree and \emph{any} smoothness
order of the Sobolev/Besov spaces.

On the other hand,
TNs are much easier to handle -- both analytically and
numerically. Moreover,
it seems the much simpler
architecture of TNs does not sacrifice anything in terms
of representing functions of classical smoothness when compared
to RePU networks.
In particular, both tools are able to recover optimal
or close to optimal approximation rates -- without
being ``adapted'' to the particular space in question.
In other words, all smoothness spaces are included in the
same approximation class.
This is to be contrasted with more standard approximation
tools, such as splines or wavelets, where the approximation class
(and thus the approximation method) has to be adapted to
the smoothness notion in question.
Moreover, both tools will frequently perform
better than predicted on specific instances of functions 
that possess structural features that are not
captured by classical smoothness
theory\footnote{A fundamental theory of these structures
remains an open question for both tools.}.

Of course, this is simply to say that both tools
do a good job when it comes to classical notions of smoothness.
We still expect that NN approximation classes are very different
than those of TNs, in an appropriate sense.
We also show in Part II \cite{partII} that TNs approximation
classes (using the TT format) are not embedded into any Besov space --
as was shown in \cite{Gribonval2019} for RePU networks.


\subsection{Main Results}\label{sec:mainresults}

First, we show that any $\Lp$-function $f$ defined on the interval $[0,1)$ can be identified with a tensor. 
For a given $b \in \N$ (the base) and $d\in \N$ (the level), we first note that any $x \in [0,1)$ can be uniquely decomposed as 
$$
x = \sum_{k=1}^d i_k b^{-k} + b^{-d} y := t_{b,d}(i_1,\hdots,i_d,y),
$$
where $(i_1,...,i_d)$ is the representation of $\lfloor b^d x \rfloor$ in base $b$ and $y = b^d x - \lfloor b^d x \rfloor$. 
This allows to  identify a function  with a tensor (or multivariate function) 
$$\tensor{f}(i_1,\hdots,i_d,y) = f(t_{b,d}(i_1,\hdots,i_d,y)) := T_{b,d} f(i_1,\hdots,i_d,y),$$
and to define different notions of ranks for a univariate function. A function $f$ can be tensorized at different levels $d\in \N$. We analyze the relation between tensorization maps at different levels, and the relation between the ranks of the corresponding tensors of different orders.  
When looking at $T_{b,d}$ as a map on $\Lp{([0,1))}$, a main result is given by \Cref{thm:tensorizationmap} and \Cref{lemma:reasonablecross}.
\begin{main result}
    For any $0<p\leq\infty$,  $b\in\N$ ($b\geq 2$) and $d\in\N$,  
     the map $T_{b,d}$ is a linear isometry from  $\Lp{([0,1))}$ to the algebraic tensor space
    $
        \mathbf{V}_{b,d,\Lp}:=(\R^b)^{\otimes d}\otimes \Lp([0,1)),
    $
    where $\mathbf{V}_{b,d,\Lp}$ is equipped with a reasonable crossnorm.
\end{main result}

For later use in approximation, we introduce the tensor subspace
\begin{align*}
    \mathbf{V}_{b, d, S}:=(\R^b)^{\otimes d}\otimes S,
\end{align*}
where $S\subset\Lp([0,1))$ is some finite-dimensional subspace. 
Then,
we can identify $\mathbf{V}_{b, d, S}$ with a finite-dimensional
subspace of $\Lp$ as
\begin{align*}
    \Vbd{S}:=\Tbd^{-1}\left(\mathbf{V}_{b, d, S}\right)\subset\Lp([0,1)).
\end{align*}
We introduce the crucial assumption that  $S$ is closed under $b$-adic dilation, i.e., for any $f \in S$ and any $k \in \{0,\hdots,b-1\}$, $f(b^{-1}(\cdot + k)) \in S.$  Under this assumption, which is reminiscent of multi-resolution analysis (MRA), we obtain bounds for multilinear ranks  that are related to the dimension of $S$. Also, under this assumption on $S$, we obtain the main results given by \Cref{inclusion_Vbd,lemma:vbvectorspace} and \Cref{thm:dense}.
\begin{main result}
The spaces $\Vbd{S}$ form a hierarchy of $\Lp$-subspaces, i.e. 
 \begin{align*}
        S:=\Vbd[@][0]{S}\subset\Vbd[@][1]{S}\subset\Vbd[@][2]{S}\subset\ldots,
    \end{align*}
    and $\Vb{S}:=\bigcup_{d\in\N}\Vbd{S}$ is a linear space.    
  If we further assume that  $S$ contains the constant function one,
    $\Vb{S}$ is dense in $\Lp$
    for $1\leq p<\infty$. 
\end{main result}
For the approximation of multivariate functions (or tensors), we use the set of tensors in the tensor train (TT) format
 $\TT{\mathbf{V}_{b, d, S}}{\bs r}$, $\bs r=(r_\nu)_{\nu=1}^d$,
cf. \Cref{sec:introtns} and \Cref{fig:mps}.
Given a basis $\{\varphi_k\}_{k=1}^{\dim S}$ of $S$, a tensor $\tensor{f}$ in $\TT{\mathbf{V}_{b, d, S}}{\bs r}$
 can be written as
\begin{align*}
    \tensor{f}(i_1,\ldots,i_d, y)&=\sum_{k_1=1}^{r_1}\cdots\sum_{k_d=1}^{r_d} \sum_{k_{d+1}}^{\dim S} v_1^{k_1}(i_1) v_2^{k_1,k_2}(i_2)
    v_3^{k_2,k_3}(i_3)\cdots v_d^{k_{d-1},k_d}(i_d) v_{d+1}^{k_d,k_{d+1}} \varphi_k(y) ,
\end{align*}
where  the parameters $\mathbf{v}:= (v_1 ,\hdots,v_{d+1}) $ form a tensor network (a collection of low-order tensors) with 
$$
 \mathbf{v} := (v_1 ,\hdots,v_{d+1}) \in \R^{b\times r_1} \times \R^{b\times  r_1 \times r_2} \times \hdots  \times \R^{b\times  r_{d-1} \times r_d} \times \R^{r_d \times \dim S} := \mathcal{P}_{b,d,S,\bs r}. 
$$
With this we define 
	$$\Phi_{b,d,S,\bs r}= T_{b,d}^{-1} ( \TT{\mathbf{V}_{b,d,S}}{\bs r}) = \{ \mathcal{R}_{b,d,S,\bs r}(\mathbf{v}) 
	: \mathbf{v} \in  \mathcal{P}_{b,d,S,\bs r} \},$$
	where $\mathcal{R}_{b,d,S,\bs r}(\mathbf{v})$ is the map which associates to a tensor network $\mathbf{v}$ the function $f = T_{b,d} \tensor f$ with $\tensor f$ defined as above. Then 
 our approximation tool for univariate functions is defined as
 $$
\tool[] := (\Phi_n)_{n\in \N}, \quad  \Phi_n = \{\varphi \in \Phi_{b,d,S,\bs r} : d\in \N, \bs r \in \N^d , \cost(\varphi) \le n\},
 $$
 where $\cost(\varphi)$ is some measure of complexity of a function $\varphi$, defined as 
 $$
 \cost(\varphi) := \min \{\cost(\mathbf{v}) :  \mathcal{R}_{b,d,S,\bs r}(\mathbf{v}) = \varphi, d \in \N , \bs r\in \N^d\},
 $$ 
 where the infimum is taken over all tensor networks $\mathbf{v}$ whose realization  is  the function  $\varphi$. 
We introduce three different measures
of complexity 
\begin{align*}
  \cost_{\mathcal{N}}(\mathbf{v})  &:=  
  \sum_{\nu=1}^{d} r_\nu,\\  
   \cost_\mathcal{C}(\mathbf{v})  &:=
  br_1+b\sum_{k=2}^{d}r_{k-1}  r_k+
    r_d \dim S,\notag \\    
    \cost_\mathcal{S}(\mathbf{v}) &:=  \sum_{\nu=1}^{d+1} \Vert v_\nu \Vert_{\ell_0} \notag,
\end{align*}
where $\Vert v_\nu\Vert_{\ell_0}$ is the number of non-zero entries in the tensor $v_\nu$.
Consequently, this defines three types of subsets
\begin{align*}
    \tool^{\mc N}&:=\set{\varphi\in\tool[]:\;\cost_{\mc N}(\varphi)\leq n},\\
        \tool^{\mc C}&:=\set{\varphi\in\tool[]:\;\cost_{\mc C}(\varphi)\leq n},\notag\\
            \tool^{\mc S}&:=\set{\varphi\in\tool[]:\;\cost_{\mc S}(\varphi)\leq n}.\notag
\end{align*}
Complexity measures $\cost_\mathcal{C}$ and $\cost_\mathcal{N}$ are related to the TT-ranks of the tensor $ T_{b,d(\varphi)} \varphi$, where $d(\varphi)$ is the minimal $d$ such that $\varphi \in V_{b,d,S}$. The function
$\cost_\mathcal{C}$ is a natural measure of complexity which corresponds to the dimension of the parameter space.
The function $\cost_\mathcal{S}$  is also a natural measure of complexity which counts the number of non-zero parameters. 
When interpreting a tensor network $\mathbf{v}$ as a sum-product neural network,  $\cost_\mathcal{N}(\mathbf{v})$ corresponds to the number of neurons, $\cost_\mathcal{C}(\mathbf{v})$ to the number of weights, and $\cost_\mathcal{S}(\mathbf{v})$ the number of non-zero weights (or connections). 

We use $\tool\in\set{\tool^{\mc N}, \tool^{\mc C}, \tool^{\mc S}}$
and the corresponding best approximation error
\begin{align*}
    \E{f}[p]:=\inf_{\varphi\in\tool}\norm{f-\varphi}[p]
\end{align*} 
for functions $f$ in $\Lp([0,1))$ 
to define approximation classes 
\begin{align*}
       \Asq :=  \Asq(\Lp, (\tool)_{n\in \N}) ):=\set{f\in\Lp([0,1)):\; \norm{f}[\Asq]<\infty},
    \end{align*}
for $\alpha>0$ and $0<q\le \infty$, where 
     \begin{align*}
        \norm{f}[\Asq]:=
        \begin{cases}
            \left(\sum_{n=1}^\infty[n^\alpha\E[n-1]{f}]^q\frac{1}{n}\right)^{1/q},&\quad 0<q<\infty,\\
            \sup_{n\geq 1}[n^\alpha\E[n-1]{f}],&\quad q=\infty.
        \end{cases}
    \end{align*}  
For the three approximation classes
\begin{align*}
N_q^\alpha(X) &:= A_q^\alpha(X , (\tool^{\mathcal{N}})_{n\in \N}),\\
C_q^\alpha(X) &:= A_q^\alpha(X ,  (\tool^{\mathcal{C}})_{n\in \N}),\notag\\
S_q^\alpha(X) &:= A_q^\alpha(X ,  (\tool^{\mathcal{S}})_{n\in \N}),
\end{align*}
we obtain the main result of this part I given by \Cref{p1-p6,comparing-spaces}.

\begin{main result}
For any $\alpha>0$, $0<p\leq\infty$ and $0<q \le \infty$,
the classes $N_q^\alpha(\Lp)$, $C_q^\alpha(\Lp)$
and $S_q^\alpha(\Lp)$ are quasi-normed vector spaces
and satisfy the continuous embeddings
\begin{align*}
&C^{\alpha}_q(L^p) \hookrightarrow S^{\alpha}_q(L^p) \hookrightarrow N^{\alpha}_q(L^p)\hookrightarrow C^{\alpha/2}_q(L^p).
\end{align*}
\end{main result}


\subsection{Outline}
In \Cref{sec:tensorization},
we discuss how one-dimensional functions can be identified with
tensors and analyze some basic properties of this identification.
In \Cref{sec:tensorapp},
we introduce our approximation tool, briefly review general results from approximation theory,
and analyze several approximation classes of rank-structured functions.
In particular, we show that these classes are quasi-normed linear spaces.
We conclude in \Cref{sec:discussion} by a brief discussion on how
tensorization can be viewed as a particular featuring step and tensor networks as a particular neural network with features as input variables. 
\section{Tensorization of Functions}\label{sec:tensorization}

We begin by introducing how one-dimensional functions can be identified with tensors of arbitrary dimension. We then introduce finite-dimensional subspaces
of tensorized functions and show that these form
a hierarchy of subspaces that are dense in $\Lp$.
This will be the basis for our approximation tool
in \Cref{sec:tensorapp}.

\subsection{The Tensorization Map}

Consider one-dimensional functions on the unit interval
\begin{equation*}
    f:[0,1)\rightarrow\R.
\end{equation*}
We tensorize such functions by encoding the input variable $x\in[0,1)$ as follows. 
Let $b\in\N$ be the base and $d\in\N$ the level.  We introduce a uniform partition of $[0,1)$ with $ b^d$ intervals $[x_i,x_{i+1})$ with $x_i = b^{-d} i$, $0 \le i \le b^d$. An integer $i\in \{0,\hdots,b^d-1\}$ admits a representation $(i_1,\hdots,i_d)$ in base $b$ such that 
$$
i = \sum_{k=1}^d i_k b^{d-k},
$$
where $i_k \in \set{0,\ldots,b-1} := \Ib.$
We define a \emph{conversion map} $\tbd$ from $\Ib^d\times[0,1)$ to $[0,1)$ defined by 
\begin{align*}
t_{b,d}(i_1,\ldots,i_d,y) = \sum_{k=1}^di_kb^{-k}+b^{-d}y.
\end{align*}
For any $x\in [0,1)$, there exists a unique $(i_1,\hdots,i_d,y) \in \Ib^d \times [0,1)$ such that $t_{b,d}(i_1,\ldots,i_d,y) = x$, where 
$(i_1,\hdots,i_d)$ is the representation of $\lfloor b^d x \rfloor$ in base $b$ and $y = b^d x -  \lfloor b^d x \rfloor$. We therefore deduce the following property. 
\begin{lemma}\label{property_tbd}
The conversion map $\tbd$ defines a linear bijection from the set $\Ib^d\times[0,1)$ to the interval
  $[0,1)$, with inverse defined for $x\in [0,1)$  by
  $$
  \tbd^{-1}(x) = (\lfloor b x \rfloor  , \lfloor b^2 x \rfloor \text{ mod } b, \hdots , \lfloor b^d x \rfloor \text{ mod } b , b^d x- \lfloor b^d x \rfloor). 
  $$
\end{lemma}	
\begin{definition}[Tensorization Map]\label{def-tensorization-map}
We define the 
 \emph{tensorization map} 
	\begin{align*}
	    \Tbd:\R^{[0,1)}\rightarrow\R^{\Ib^d\times[0,1)}, \quad f \mapsto f \circ t_{b,d} := \tensor{f}
	\end{align*}
which associates to a function 
$f \in \R^{[0,1)}$ the multivariate function 
	$	   \tensor{f}\in \R^{\Ib^d\times[0,1)} $ such that 
	\begin{align*}
	    \tensor{f}(i_1,\ldots,i_d,y):=f(\tbd(i_1,\ldots,i_d,y)).
	\end{align*}	
\end{definition}
From \Cref{property_tbd}, we directly deduce the following property of $\Tbd$.
\begin{proposition}\label{Tbd-bijection}
The tensorization map $\Tbd$ is a linear bijection from $\R^{[0,1)}$ to  $\R^{\Ib^d\times[0,1)}$, with inverse  given for $\tensor{f} \in \R^{\Ib^d \times [0,1)}$ by $\Tbd^{-1} \tensor{f} = \tensor{f} \circ t_{b,d}^{-1}$.
  \end{proposition}
The space $\R^{\Ib^d\times[0,1)}$ can be identified with the algebraic tensor space $$ \mathbf{V}_{b,d} := \R^{\Ib^d} \otimes \R^{[0,1)}=  \underbrace{\R^{\Ib} \otimes \hdots \otimes \R^{\Ib}}_{\text{$d$ times }}\otimes \R^{[0,1)} =: (\R^{\Ib})^{\otimes d} \otimes \R^{[0,1)} ,$$ which is the set of functions $\tensor{f}$ defined on $\Ib^d \times [0,1)$ that admit 
a representation
\begin{equation}
\tensor{f}(i_1,\hdots,i_d,y) = \sum_{k=1}^r v_1^k(i_1) \hdots v_d^k(i_d) g^k(y) :=\sum_{k=1}^r (v_1^k \otimes \hdots \otimes v_d^k \otimes g^k)(i_1,\hdots,i_d,y)    \label{algebraic-tensor}
\end{equation}
for some $r \in \N$ and for some functions $v_\nu^k \in \R^{\Ib}$ and $g^k \in \R^{[0,1)}$, $1\le k \le r$, $1\le \nu\le d$. Letting $\{\delta_{j_\nu} : j_\nu \in \Ib\}$ be the canonical basis of  $\R^{\Ib}$, defined by $\delta_{j_\nu}(i_\nu) = \delta_{i_\nu,j_\nu}$, a function $\tensor{f}\in \R^{\Ib^d\times[0,1)}$ admits the particular representation
\begin{equation}
\tensor{f}   = \sum_{j_1 \in \Ib} \hdots \sum_{j_d \in \Ib} \delta_{j_1} \otimes \hdots  \otimes \delta_{j_d} \otimes  \tensor{f}(j_1,\hdots,j_d,\cdot).
\label{canonical-representation}
\end{equation} 
The following result provides an interpretation of the above representation.
\begin{lemma}\label{elementary-tensor-localized-function}
Let $f\in \R^{[0,1)}$ and $\tensor{f} = T_{b,d} f \in \mathbf{V}_{b,d}$. For $(j_1,\hdots,j_d) \in \Ib^d$ and $j= \sum_{k=1}^{d} j_k b^{d-k}$, it holds
\begin{equation}
 T_{b,d}( f \indicator{[b^{-d} j ,b^{-d} (j+1))}) = \delta_{j_1}\otimes \hdots \otimes \delta_{j_d} \otimes \tensor{f}(j_1,\hdots,j_d,\cdot) ,
\label{j-piece}
\end{equation}
and  
\begin{align}
\tensor{f}(j_1,\hdots,j_d,\cdot) = f(b^{-d}(j+\cdot)),\label{partial-eval-dilation}
\end{align}
where $f(b^{-d}(j+\cdot))$ is the restriction of $f$ to the interval $[b^{-d} j ,b^{-d} (j+1))$ rescaled to $[0,1)$. 
\end{lemma}
\begin{proof}
For $x = t_{b,d}(i_1,\hdots,i_d,y)$, 
\begin{align*}f(x) \indicator{[b^{-d} j ,b^{-d} (j+1))}(x) &= \delta_{j_1}(i_1) \hdots \delta_{j_d}(i_d) f(t_{b,d}(i_1,\hdots,i_d,y)) \\
&=  \delta_{j_1}(i_1) \hdots \delta_{j_d}(i_d) f(t_{b,d}(j_1,\hdots,j_d,y)) \\
&= \delta_{j_1}(i_1) \hdots \delta_{j_d}(i_d) \tensor{f}(j_1,\hdots,j_d,y).
\end{align*}
The property \eqref{partial-eval-dilation} simply results from the definition of $\tensor{f}$.   
\end{proof}

From \Cref{elementary-tensor-localized-function}, we deduce that the representation \eqref{canonical-representation}
corresponds to the decomposition of $f = \Tbd^{-1}(\tensor{f})$ as a superposition of functions with disjoint supports, 
\begin{equation}
f(x) = \sum_{j =0}^{b^d-1} f_{j}(x), \quad f_{j}(x) = \indicator{[b^{-d} j ,b^{-d} (j+1))}(x) f(x), \label{piecewise-representation}
\end{equation}
where $f_j$ is the function 
supported  on  the interval  $[b^{-d} j ,b^{-d} (j+1))$ and equal to $f$ on this interval. Also, \Cref{elementary-tensor-localized-function}  yields the following result. 
\begin{corollary}\label{tensorization-dilation} 
A function $f \in \R^{[0,1)}$ defined by 
$$
f(x) = 
  \begin{cases} g(b^d x-j) & \text{for} \quad x\in [ b^{-d}j , b^{-d}(j+1))  \\
    0 & \text{elsewhere}
    \end{cases}$$
    with $g\in \R^{[0,1)}$ and $0\le j <b^d$ admits a tensorization $T_{b,d} f = \delta_{j_1}\otimes \hdots \otimes \delta_{j_d} \otimes g$, which is an elementary tensor. 
    
\end{corollary}

We now provide a useful result on compositions of tensorization maps for changing the representation level of a function.  
\begin{lemma}\label{tbdbard}
Let $\bar d, d \in \N$ such that $\bar d >d$. For any $(i_1,\hdots,i_{\bar d},y)\in \Ib^{\bar d} \times [0,1)$, it holds 
$$
t_{b,\bar d}(i_1,\hdots,i_{\bar d},y) = t_{b,d}(i_1,\hdots,i_{d},t_{b,\bar d-d}(i_{d+1},\hdots,i_{\bar d},y)),
$$
and 
 the operator $T_{b,\bar d}\circ T_{b,d}^{-1}$ from $\mathbf{V}_{b,d}$ to $\mathbf{V}_{b,\bar d}$ is such that 
$$
T_{b,\bar d}\circ T_{b,d}^{-1} = id_{\{1,\hdots,d\}} \otimes T_{b,\bar d - d}
$$
where $id_{\{1,\hdots,d\}}$ is the identity operator on $\R^{\Ib^d}$ and $T_{b,\bar d -d}$ is the tensorization map from $\R^{[0,1)}$ to $\mathbf{V}_{b,\bar d - d}$. Also, the operator $T_{b,  d}\circ T_{b,\bar d}^{-1}$ from $\mathbf{V}_{b,\bar d}$ to $\mathbf{V}_{b,d}$ is such that 
$$
T_{b, d}\circ T_{b,\bar d}^{-1} = id_{\{1,\hdots,d\}} \otimes T_{b,\bar d - d}^{-1}.
$$
\end{lemma}
\begin{proof}
See \Cref{proof:lemma26}.
\end{proof}

For $d=0$, we adopt the conventions that $t_{b,0} $ is the identity on $[0,1)$, $T_{b,0}$ is the identity operator on $\R^{[0,1)}$, and $\mathbf{V}_{b,0} = \R^{[0,1)}$.

\subsection{Ranks and Minimal Subspaces}
The minimal integer $r$ such that $\tensor{f} \in \mathbf{V}_{b,d} $ admits a representation of the form \eqref{algebraic-tensor} is the \emph{canonical tensor rank} of $\tensor{f}$ denoted $r(\tensor{f}).$  We deduce from the representation \eqref{canonical-representation} that  $$r(\tensor{f}) \le b^d.$$ 
 
 Other notions of ranks can be defined from the classical notion of rank  by identifying a tensor with a tensor of order two (through unfolding). Letting $V_\nu := \R^{\Ib}$ for $1\le \nu\le d$, and $V_{d+1}:= \R^{[0,1)}$, we have $$\mathbf{V}_{b,d} = \bigotimes_{\nu=1}^{d+1} V_\nu.$$
 Then for any $\beta \subset \{1,\hdots,d+1\}$ and its complementary set $\beta^c =\{1,\hdots,d+1\}\setminus \beta$, a tensor 
 $\tensor{f} \in \mathbf{V}_{b,d} $ can be identified with an order-two tensor in $\mathbf{V}_\beta \otimes \mathbf{V}_{\beta^c}$, where $\mathbf{V}_{\gamma} = \bigotimes_{\nu\in \gamma} V_\nu$, called the
 $\beta$-unfolding of $\tensor{f}.$
This allows us to define the notion of $\beta$-rank.
\begin{definition}[$\beta$-rank]\label{def:betarank}
For $\beta \subset \{1,\hdots,d+1\}$, the  $\beta$-rank of $\tensor{f} \in \mathbf{V}_{b,d}$, denoted $r_\beta(\tensor{f})$, is the minimal integer  such that $\tensor{f}$ admits a representation of the form 
 \begin{align}\label{eq:tensorfrep}
        \tensor{f}=\sum_{k=1}^{r_\beta(\tensor{f})}
        \bs{v}^k_\beta\otimes\bs{v}^k_{\beta^c},
    \end{align}
    where $\bs{v}_\beta^k \in \mathbf{V}_\beta$ and 
     $\bs{v}_{\beta^c}^k \in \mathbf{V}_{\beta^c}$.  
      \end{definition}
    Since $\mathbf{V}_{b,d}$ is an algebraic tensor space, the $\beta$-rank is finite
and we have $r_\beta(\tensor{f})\leq r(\tensor{f})$ (though the $\beta$-rank can be much smaller).
Moreover, we have the following straightforward property
$$
r_\beta(\tensor{f}) = r_{\beta^c}(\tensor{f}),
$$
and the bound
\begin{align}\label{rank-bound-general-format}
    r_\beta(\tensor{f})\leq\min\left\{\prod_{\nu\in\beta}\dim V_\nu,\prod_{\nu\in\beta^c}\dim V_\nu\right\},
\end{align}
which can be useful for small $b$ and either very small or very large $\#\beta$.      

Representation \eqref{eq:tensorfrep} is not unique but
    the space spanned by the $\bs{v}_k^\beta$ is unique  and corresponds to the $\beta$-\emph{minimal subspace} of $\tensor{f}$.    
\begin{definition}[$\beta$-minimal subspace]\label{def:minsub}
    For $\beta \subset \{1,\hdots,d+1\}$, the $\beta$-minimal subspace of $\tensor{f}$, denoted $\Umin{\beta}(\tensor{f})$, is  the smallest subspace $\mathbf{U}_\beta \subset \mathbf{V}_\beta$ such that $\tensor{f} \in \mathbf{U}_\beta\otimes \mathbf{V}_{\beta^c}$, and its dimension is 
    $$
\dim(\Umin{\beta}(\tensor{f})) = r_\beta(\tensor{f}).
$$   
\end{definition}
We have the following useful characterization of minimal subspaces from partial evaluations of a tensor. 
\begin{lemma}\label{ranks-partial-evaluations}
For $\beta \subset \{1,\hdots,d\}$ and $\tensor{f} \in \mathbf{V}_{b,d}$,  
$$
\Umin{\beta^c}(\tensor{f}) = \mathrm{span}\{\tensor{f}(j_{\beta} , \cdot) : j_\beta \in \Ib^{\#\beta}\} \subset \mathbf{V}_{b,d-\#\beta},
$$
where $\tensor{f}(j_{\beta} , \cdot) \in  \mathbf{V}_{\beta^c}  = \mathbf{V}_{b,d-\#\beta}$ is a partial evaluation of $\tensor{f}$ along dimensions  $\nu \in \beta$.
\end{lemma}
\begin{proof}
    See \Cref{proof:partialeval}.
\end{proof}

Next we define a notion of $(\beta,d)$-rank for univariate functions in $\R^{[0,1)}$. 
\begin{definition}[$(\beta,d)$-rank]\label{def:betadrank}
For a function $f \in \R^{[0,1)}$, $d\in \N$ and $\beta \subset \{1,\hdots,d+1\} $, we define the $(\beta,d)$-rank of $f$, denoted $r_{\beta,d}(f)$, as the $\beta$-rank of its tensorization in $\mathbf{V}_{b,d}$, 
$$
r_{\beta,d}(f) = r_\beta(\Tbd f).
$$
\end{definition}
In the rest of this work, we will essentially consider subsets $\beta$  of the form $\{1,\hdots,\nu\}$ or $\{\nu+1,\hdots,d+1\}$ for some $\nu\in \{1,\hdots,d\}$. For the corresponding $\beta$-ranks, we will use the shorthand notations
$$
r_{\nu}(\tensor{f}) := r_{\{1,\hdots,\nu\}}(\tensor{f}), \quad r_{\nu,d}(f) =  r_{\{1,\hdots,\nu\},d}(f).
$$
Note that $r_\nu(\tensor{f})$ should not be confused with $r_{\{\nu\}}(\tensor{f}).$   
The ranks $(r_\nu(\tensor{f}))_{1\le \nu\le d}$ of a tensor $\tensor{f} \in \mathbf{V}_{b,d}$ have to satisfy some relations, as seen in the next lemma. 
\begin{lemma}[Ranks Admissibility Conditions]\label{lemma:rankgrowth}
    Let $\tensor{f} = \mathbf{V}_{b,d}$. For any set $\beta \subset \{1,\hdots,d+1\}$ and any partition $\beta = \gamma \cup \alpha$, we have 
    $$
    r_{\beta}(\tensor{f})  \le r_{\gamma}(\tensor{f}) r_{\alpha}(\tensor{f}) 
    $$
    and in particular
    \begin{align}\label{eq:rankgrowth}
        r_{\nu+1}(\tensor{f})&\leq b r_{\nu}(\tensor{f})\quad \text{and} \quad r_{\nu}(\tensor{f})\leq b r_{\nu+1}(\tensor{f}),\quad
        1\leq \nu\leq d-1,
    \end{align}
\end{lemma}
\begin{proof}
    See \Cref{proof:rankadmiss}.
\end{proof}

A function $f$ admits infinitely many tensorizations of different levels. The following result provides a relation between  minimal subspaces. 
\begin{lemma}\label{link-minimal-subspaces}
Consider a function $f \in \R^{[0,1)}$ and its tensorization $\tensor{f}^{d}= \Tbd f$ at level $d$. 
For any $1\leq \nu \leq d$,
$$
T_{b,d-\nu}^{-1}(\Umin{\{\nu+1,\hdots,d+1\}}(\tensor{f}^{d}))  =  \linspan{\tensor{f}^{\nu}(j_1,\ldots,j_\nu,\cdot):\;(j_1,\ldots,j_\nu)\in\Ib^\nu} = \Umin{\{\nu+1\}}(\tensor{f}^{\nu}) ,
$$
where $\tensor{f}^{\nu} = T_{b,\nu} f$ is the tensorization of $f$ at level $\nu$.\end{lemma}
\begin{proof}
    See \Cref{proof:liniminsub}.
\end{proof}

For $j=\sum_{k=1}^{\nu} j_k b^{\nu-k},$ since $\tensor{f}^{\nu}(j_1,\ldots,j_\nu,\cdot) =  f(b^{-\nu}(j + \cdot))$ is the restriction of $f$ to the interval $[b^{-\nu}j,b^{-\nu}(j+1))$  rescaled to 
$[0, 1)$,  \Cref{link-minimal-subspaces}
 provides a simple interpretation of minimal subspace $\Umin{\{\nu+1\}}(\tensor{f}^{\nu})$
 as the linear span of contiguous pieces of $f$ rescaled to $[0,1)$, see the illustration in 
 \Cref{fig:unfold}.

\begin{figure}[h]
    \centering
    \begin{subfigure}[h]{.6\textwidth}
        \centering
        \includegraphics[scale=.4]{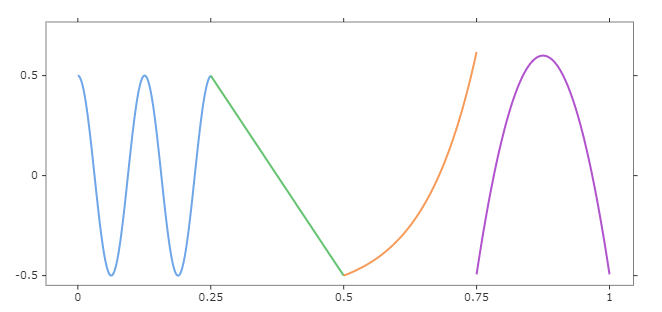}
        \subcaption{Function $f:[0,1)\rightarrow\R$}\label{fig11}
    \end{subfigure}
    \\
    \begin{subfigure}[h]{.6 \textwidth}
        \centering
        \includegraphics[scale=.4]{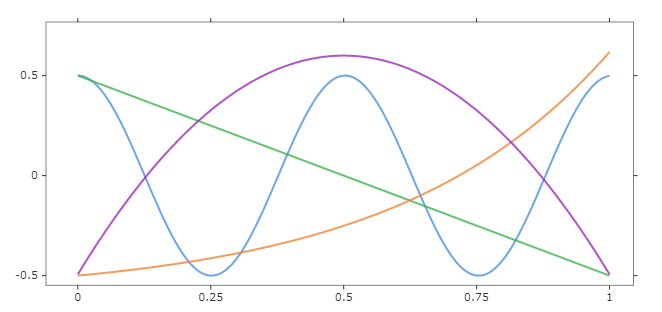}        
        \subcaption{Partial evaluations $\tensor{f}^\nu(j_1,j_2,\cdot)$ for $(j_1,j_2) \in \{0,1\}^2$.}\label{fig12}
    \end{subfigure}
    \caption{A function $f:[0,1)\rightarrow\R$  and partial evaluations of $\tensor{f}^\nu \in \mathbf{V}_{b,d}$ for $b=d=2$.}
    \label{fig:unfold}
\end{figure} 
 
\begin{corollary}\label{link-ranks-d-nu}
Let $f \in \R^{[0,1)}$ and $d\in \N$. For any $1\le \nu\le d,$
$$
r_{\nu,d}(f) = r_{\nu,\nu}(f)$$
and 
$$
r_{\nu,\nu}(f) = \dim\mathrm{span}\{f( b^{-\nu}(j+\cdot)) : 0\le j \le b^\nu-1\} .
$$
\end{corollary}	
\begin{proof}
It holds 
$$r_{\nu,d}(f) = r_{\{1,\hdots,\nu\}}(\tensor{f}^d) = r_{\{\nu+1,\hdots,d+1\}}(\tensor{f}^d) = \dim  \Umin{\{\nu+1,\hdots,d+1\}}(\tensor{f}^d)$$ and  $$r_{\nu,\nu}(f) =  r_{\{1,\hdots,\nu\}}(\tensor{f}^\nu) =  r_{\nu+1}(\tensor{f}^\nu) = \dim \Umin{\{\nu+1\}}(\tensor{f}^\nu).$$ \Cref{link-minimal-subspaces} then implies that $r_{\nu,d}(f) = r_{\nu,\nu}(f)$ and provides the characterization from the linear span of  
$\tensor{f}^{\nu}(j_1,\ldots,j_\nu,\cdot) =  f( b^{-\nu}(j+\cdot)),$ with $j=\sum_{k=1}^{\nu} j_k b^{\nu-k},$ which is linearly identified with the restriction 
$f_{\mid [b^{-\nu}j,b^{-\nu}(j+1))}$ shifted and rescaled to $[0,1)$.
\end{proof}

%

\subsection{Measures, Lebesgue and Smoothness Spaces}
We now look at $T_{b,d}$ as a linear map between spaces of measurable functions, by equipping the interval $[0,1)$ with the  Lebesgue measure.
\begin{proposition}\label{bijection-measurable}
The Lebesgue measure $\lambda$ on 
$[0,1)$ is the push-forward measure  through the map $t_{b,d}$ of the product measure $\mu_{b,d} := \mu_{b}^{\otimes d} \otimes \lambda$, where $\mu_b$ is the uniform probability measure on $\Ib$. Then 
the tensorization map $\Tbd$ defines a linear isomorphism from the space of measurable functions $\R^{[0,1)}$ to the space of measurable functions $\R^{\Ib^d\times[0,1)}$, where $[0,1)$ is equipped with the Lebesgue measure and 
  $\Ib^d\times[0,1)$ is equipped with the product measure   $\mu_{b,d}$.
  \end{proposition}
\begin{proof}
    See \Cref{proof:bijection}.
\end{proof}

\subsubsection{Lebesgue Spaces}
For $0< p \le \infty$, we consider the Lebesgue space $\Lp([0,1))$ of functions defined on $[0,1)$ equipped with its standard (quasi-)norm. 
Then we  consider the algebraic tensor space $$\mathbf{V}_{b,d,L^p} := {\R^{\Ib}}^{\otimes d} \otimes L^p([0,1)) \subset \mathbf{V}_{b,d},$$ which is the space of multivariate functions $\tensor{f}$ on $\Ib^d \times [0,1)$ with partial evaluations $\tensor{f}(j_1,\hdots,j_d,\cdot) \in L^p([0,1))$.  From hereon we frequently abbreviate $\Lp:=\Lp({[0,1)})$.

\begin{theorem}[Tensorization is an $\Lp$-Isometry]\label{thm:tensorizationmap}
For any $0 < p \le \infty$,  
 $\Tbd$ is a linear isometry from $\Lp([0,1))$ to  $\mathbf{V}_{b,d,L^p}$ equipped with the (quasi-)norm 
  $\Vert \cdot\Vert_p$ defined by 
 $$\norm{\tensor{f}}[p]^p = 
 \sum_{j_1\in \Ib} \hdots \sum_{j_d\in \Ib} b^{-d} \Vert  \tensor{f}(j_1,\hdots,j_d,\cdot) \Vert_{p}^p
 $$
 for $p<\infty$,  
or  
$$
\norm{\tensor{f}}[\infty] = \max_{j_1\in \Ib} \hdots \max_{j_d \in \Ib} \Vert \tensor{f}(j_1,\hdots,j_d,\cdot) \Vert_\infty.
$$
\end{theorem}
 \begin{proof}
  The results follows from \Cref{bijection-measurable} and by noting that 
for $\tensor{f} = \Tbd f = f \circ t_{b,d}$, 
\begin{align*}
\Vert f \Vert_p^p   = \int_{[0,1)} \vert f(x) \vert^p d\lambda(x) =  \int_{\Ib^d\times[0,1)} \vert \tensor{f} (j_1,\hdots,j_d,y) \vert^p d\mu_{b,d}(j_1,\hdots,j_d,y)  =\norm{\tensor{f}}[p]^p 
\end{align*}
for $p<\infty$, and 
$
\Vert f \Vert_\infty   = \esssup_{x  } \vert f(x) \vert = \esssup_{(j_1,\hdots,j_d,y) } \vert \tensor{f}(j_1,\hdots ,j_d,y) \vert =  \norm{\tensor{f}}[\infty].
$
 \end{proof}
 
We  denote by $\ell^p(\Ib)$ the space $\R^{\Ib}$ equipped with the (quasi-)norm $\Vert \cdot \Vert_{\ell^p}$ defined for $v = (v_k)_{k\in \Ib}$ by  
$$\norm{v}[\ell^p]^p:=b^{-1}\sum_{k=0}^{b-1}|v_k|^p \quad  (p<\infty), \quad \norm{v}[\ell^\infty]:= \max_{0\le k \le b-1} |v_k|.$$
The space $\mathbf{V}_{b,d,L^p}$ can  then be identified with the algebraic tensor space
$$
(\ell^p(\Ib))^{\otimes d} \otimes L^p([0,1)).
$$
and $\Vert \cdot\Vert_p$ is a crossnorm, i.e., satisfying for an elementary tensor $v^1\otimes \hdots \otimes v^{d+1} \in \mathbf{V}_{b,d,L^p}$, 
$$
\Vert v^1\otimes \hdots \otimes v^{d+1}  \Vert_p = \Vert v^1 \Vert_{\ell^p} \hdots \Vert v^d \Vert_{\ell^p}  \Vert v^{d+1} \Vert_p.
$$ 
and even a reasonable crossnorm for $1\le p \le \infty$ (see \Cref{lemma:reasonablecross} in the appendix).
We let $\{e^p_k\}_{k\in \Ib}$ denote the normalized canonical basis of $\ell^p(\Ib)$, defined by 
\begin{align}
e_k^p = b^{1/p} \delta_k \text{ for }  0<p<\infty, \quad \text{and } \quad 
e^\infty_{k} = \delta_k \; \text{ for }  p=\infty. \label{ekp}
\end{align}
The tensorization 
$\boldsymbol{f} = T_{b,d}f$  of  a  function $f \in L^p([0,1))$ 
admits a representation
\begin{align}
\tensor{f} =   \sum_{j_1 \in \Ib} \hdots \sum_{j_d \in \Ib} e^p_{j_1} \otimes \hdots \otimes e^{p}_{j_d} \otimes f^p_{j_1,\hdots,j_d},
\label{representation-normalized}
\end{align}
with $f^p_{j_1,\hdots,j_d} = b^{-d/p} \tensor{f}(j_1,\hdots,j_d,\cdot)$ for $p<\infty$ and $f^\infty_{j_1,\hdots,j_d} =   \tensor{f}(j_1,\hdots,j_d,\cdot)$. The crossnorm property implies that 
$$
\Vert e^p_{j_1} \otimes \hdots \otimes e^{p}_{j_d} \otimes f^p_{j_1,\hdots,j_d} \Vert_p = \Vert  f^p_{j_1,\hdots,j_d} \Vert_p,
$$
so that \Cref{thm:tensorizationmap} implies 
$$
\Vert f \Vert_p = \Big(\sum_{(j_1,\hdots,j_d) \in \Ib^d} \Vert  f^p_{j_1,\hdots,j_d} \Vert_p^p\Big)^{1/p}\quad (p<\infty), \quad \Vert f \Vert_\infty = \max_{(j_1,\hdots,j_d) \in \Ib^d} \Vert  f^\infty_{j_1,\hdots,j_d} \Vert_p.
$$
 
\subsubsection{Sobolev Spaces}\label{sec:sobolev}
Now consider functions $f $ in the Sobolev space $W^{k,p} := W^{k,p}([0,1))$, equipped with the (quasi-)norm 
$$
\Vert f \Vert_{W^{k,p}} = (\Vert f \Vert_p^p + \vert f \vert_{W^{k,p}}^p)^{1/p} \quad (p<\infty), \quad \Vert f \Vert_{W^{k,\infty}} = \max\set{\Vert f \Vert_p , \vert f \vert_{W^{k,\infty}}},
$$
where $\vert f \vert_{W^{k,p}}$ is a (quasi-)semi-norm defined by
$$
\vert  f \vert_{W^{k,p}} = \Vert D^k f \Vert_p,
$$
with $D^k f := f^{(k)}$ the $k$-th weak derivative of $f$. Since $f$ and  its tensorization $\tensor{f} = T_{b,d} f$ are such that 
$
f(x) = \tensor{f}(j_1,\hdots,j_d,b^d x- j)
$ 
for $x\in [b^d j,b^d(j+1))$ and 
 $j= \sum_{k=1}^d b^{d-k} j_k$, we deduce that 
$$
D^k f(x) = b^{kd} \frac{\partial^k}{\partial y^k} \tensor{f}(j_1,\hdots,j_d, b^d x - j)
$$
for $x\in [b^d j,b^d(j+1))$, 
that means that $D^k$ can be identified with a rank-one operator over $\mathbf{V}_{b,d,\Lp}$,
$$
T_{b,d} \circ D^k \circ T_{b,d}^{-1}  =  id_{\{1,\hdots,d\}} \otimes (b^{kd} D^k) .
$$
We deduce from \Cref{thm:tensorizationmap} that for $f\in W^{k,p}$,
$$
\vert f \vert_{W^{k,p}} = \Vert D^k f \Vert_p = \Vert T_{b,d} (D^k f)\Vert_p = 
b^{kd} \Vert (id_{\{1,\hdots,d\}} \otimes  D^k ) \tensor{f} \Vert_p  ,
$$
with 
$$
 (id_{\{1,\hdots,d\}} \otimes  D^k ) \tensor{f}  =\sum_{j\in \Ib^d} \delta_{j_1}\otimes \hdots \otimes \delta_{j_d} \otimes D^k \tensor{f}(j_1,\hdots,j_d,\cdot) . 
$$
Then we deduce that if  
 $f \in W^{k,p}$, $\tensor{f} = T_{b,d} f$ is in the algebraic tensor space $$\mathbf{V}_{b,d,W^{k,p}} := (\R^{\Ib})^{\otimes d} \otimes W^{k,p},$$ and   
$$
\vert  f \vert_{W^{k,p}} =  b^{kd} \Vert \sum_{j_1\in \Ib} \hdots \sum_{j_d\in \Ib}  \delta_{j_1} \otimes \hdots \otimes \delta_{j_d} \otimes 
 D^k\tensor{f}(j_1,\hdots,j_d,\cdot)  \Vert_p.
$$
This implies that $T_{b,d} W^{k,p} \subset \mathbf{V}_{b,d,W^{k,p}}$ but  $T_{b,d}^{-1}( \mathbf{V}_{b,d,W^{k,p}}) \not\subset W^{k,p}.$ In fact, $T_{b,d}^{-1}( \mathbf{V}_{b,d,W^{k,p}}) = W^{k,p}(\mathcal{P}_{b,d})$,   the broken Sobolev space associated with the partition
 $\mathcal{P}_{b,d} = \{[b^d j , b^d(j+1)) : 0\le j\le b^d-1\}$. 
 From the above considerations, we deduce 
 \begin{theorem}\label{isometryWkp}
For any $0<p\le \infty$ and $k\in \N_0$,   $T_{b,d}$ is a linear isometry from the broken Sobolev space $W^{k,p}(\mathcal{P}_{b,d})$ to 
 $\mathbf{V}_{b,d,W^{k,p}} $ equipped with the (quasi-)norm 
$$
 \Vert \tensor{f} \Vert_{W^{k,p}} = (\Vert \tensor{f}\Vert_p^p + \vert \tensor{f}\vert_{W^{k,p}}^p)^{1/p} \quad (p<\infty), \quad \Vert \tensor{f} \Vert_{W^{k,\infty}} = \max\set{\Vert \tensor{f} \Vert_\infty , \vert \tensor{f} \vert_{W^{k,\infty}}},
$$
 where $\vert \cdot \vert_{W^{k,p}}$ is a (quasi-)semi-norm defined by
$$\vert \tensor{f} \vert_{W^{k,p}}^p := b^{d(kp-1)} \sum_{(j_1,\hdots,j_d)\in \Ib^d} \vert  \tensor{f}(j_1,\hdots,j_d,\cdot) \vert_{W^{k,p}}^p$$
for $p<\infty$, and 
$$\vert \tensor{f} \vert_{W^{k,\infty}}^\infty := b^{dk} \max_{(j_1,\hdots,j_d)\in \Ib^d} \vert  \tensor{f}(j_1,\hdots,j_d,\cdot) \vert_{W^{k,\infty}}.
$$
 \end{theorem}

\subsubsection{Besov Spaces}\label{sec:besov}
Let $f\in\Lp$, $0<p\leq\infty$ and consider the difference operator
\begin{align*}
    \diff_h&:\Lp{([0,1))}\rightarrow\Lp{([0,1-h))},\\
    \diff_h[f](\cdot)&:=f(\cdot+h)-f(\cdot).
\end{align*}
For $r=2,3,\ldots$, the $r$-th difference is defined as
\begin{align*}
    \diff_h^r:=\diff_h \circ \diff_h^{r-1},
\end{align*}
with $\diff_h^1:=\diff_h$.
\emph{The $r$-th modulus of smoothness} is defined as
\begin{align}\label{eq:modsmooth}
    \mod_r(f,t)_p:=\sup_{0<h\leq t}\norm{\diff_h^r[f]}[p],\quad t>0.
\end{align}

\begin{definition}[Besov Spaces]\label{def:besov}
    For parameters $\alpha>0$ and $0<p,q\leq\infty$, define
    $r:=\lfloor\alpha\rfloor+1$ and the Besov  (quasi-)semi-norm as
    \begin{align*}
        \snorm{f}[\Bsqp]:=
        \begin{cases}
            \left({\int_0^{1}}[t^{-\alpha}\mod_r(f,t)_p]^q\frac{\d t}{t}\right)^{1/q},
            &\quad 0<q<\infty,\\
            {\sup_{0<t\le 1}}t^{-\alpha}\mod_r(f,t)_p,
            &\quad q=\infty.            
        \end{cases}
    \end{align*}
    The Besov (quasi-)norm is defined as
    \begin{align*}
        \norm{f}[\Bsqp]:=\norm{f}[p]+\snorm{f}[\Bsqp].
    \end{align*}
    The \emph{Besov space} is defined as
    \begin{align*}
        \Bsqp:=\set{f\in\Lp:\;\norm{f}[\Bsqp]<\infty}.
    \end{align*}
\end{definition}

As in \Cref{sec:sobolev}, we would like to compare the Besov
space $\Bsqp$ with the algebraic tensor space
\begin{align*}
    \mathbf{V}_{b,d,\Bsqp} := (\R^{\Ib})^{\otimes d} \otimes\Bsqp.
\end{align*}
First, we briefly elaborate how the Besov (quasi-)semi-norm scales under
affine transformations of the interval. I.e.,
suppose we are given a function $f:[a,b)\rightarrow\R$ with
$-\infty<a<b<\infty$ and a transformed $\fbar$ such that
\begin{align*}
    \fbar:[\abar,\bbar)\rightarrow\R,\quad
    \xbar\mapsto x:=\frac{b-a}{\bbar-\abar}(\xbar-\abar)+a
    \mapsto f(x)=\fbar(\xbar),
\end{align*}
for $-\infty<\abar<\bbar<\infty$.
Then,
\begin{alignat*}{1}
    \diff_{\hbar}^r[\fbar]&:[\abar,\bbar-r\hbar)\rightarrow\R,\\
    \diff_{\hbar}^r[\fbar](\xbar)&=\sum_{k=0}^r{r\choose k}(-1)^{r-k}\fbar(\xbar+k\hbar)=
    \sum_{k=0}^r{r\choose k}(-1)^{r-k}f(x+kh)
    =\diff_{h}^r[f](x)
\end{alignat*}
with
\begin{align*}
    h:=\frac{b-a}{\bbar-\abar}\hbar.
\end{align*}
For $0<p<\infty$, we obtain for the modulus of smoothness
\begin{align*}
    \mod_r(\fbar, \tbar)^p_p=\frac{\bbar-\abar}{b-a}\mod_r(f,t)_p^p,
    \quad t:=\frac{b-a}{\bbar-\abar}\tbar,
\end{align*}
and for $p=\infty$, $\mod_r(\fbar, \tbar)_\infty=\mod_r(f, t)_\infty$.
Finally, for the Besov (quasi-)semi-norm this implies
\begin{alignat*}{2}
    \snorm{\fbar}[\Bsqp]&=
    \left(\frac{\bbar-\abar}{b-a}\right)^{1/p-\alpha}\snorm{f}[\Bsqp],
    &&\quad 0<q\leq\infty,\;0<p<\infty,\\
    \snorm{\fbar}[\Bsqp]&=
    \left(\frac{\bbar-\abar}{b-a}\right)^{-\alpha}\snorm{f}[\Bsqp],&&\quad
    0<q\leq\infty,\;p=\infty\notag.
\end{alignat*}

With this scaling at hand, {for $p<\infty$}, what remains is ``adding up''
Besov (quasi-)norms of partial evaluations
$\tensor{f}(j_1,\hdots,j_d,\cdot)$.
The modulus of smoothness from \eqref{eq:modsmooth}
is not suitable for this task.
Instead, we can use an equivalent measure of
smoothness via the \emph{averaged modulus of smoothness}
(see \cite[§5 of Chapter 6 and §5 of Chapter 12]{DeVore93})
\begin{align*}
    \w_r(f,t)^p_p:=\frac{1}{t}\int_0^t\norm{\diff_h^r[f]}[p]^p\d h,\quad 0<p<\infty.
\end{align*}
With this {definition of Besov norm}, we can {define a  (quasi-)semi-norm
\begin{align*}
    \snorm{f}[\Bsqp]  {=}
   \left( \int_0^1 [t^{-\alpha } \w_r(f,t)_p]^q \frac{\d t}{t} \right)^{1/q},\quad 0<q<\infty,
\end{align*}
which is equivalent to the former one and therefore results in the same Besov space $\Bsqp$.}
Expanding the right-hand-side and interchanging the order
of integration
allows us to add up the contributions
to $\snorm{f}[\Bsqp]$ over the intervals $[b^d j , b^d(j+1))$,
provided that $q=p$. 
However, note that this is not the same
as summing over $\snorm{\tensor{f}(j_1,\hdots,j_d,\cdot)}[\Bsqp]$,
since the latter necessarily omits the contributions of
$|\diff_h^r[f]|$ across the right boundaries of the intervals
$[b^d j , b^d(j+1))$.

\begin{example}
    Consider the function
    \begin{align*}
        f(x):=
        \begin{cases}
            1,\quad& 0\leq x\leq 1/2,\\
            0,\quad&\text{otherwise}.
        \end{cases}
    \end{align*}
    Take $0<\alpha<1$ and $r=1$ in \Cref{def:besov}.
    The first difference is then
    \begin{align*}
        \diff_h[f](x)=
        \begin{cases}
            1,\quad &1/2-h<x\leq 1/2,\\
            0,\quad &\text{otherwise}.
        \end{cases}
    \end{align*}
    For $0<p<\infty$,
    \begin{align*}
        \norm{\diff_h[f]}[p]^p=h,
    \end{align*}
    and for $p=\infty$,
    \begin{align*}
        \norm{\diff_h[f]}[\infty]=1.
    \end{align*}
    Thus, for the ordinary modulus of smoothness we obtain
    \begin{alignat*}{2}
        \mod_r(f,t)_p&=t^{1/p},&&\quad 0<p<\infty,\\
        \mod_r(f,t)_\infty&=1.
    \end{alignat*}
    Inserting this into \Cref{def:besov}, we see that
    $f\in\Bsqp$ if and only if $p\neq\infty$ and $0<\alpha<1/p$.
    In this case $0<\snorm{f}[{\Bsqp}]<\infty$.
    
    On the other hand, for $b=2$ and $d=1$, the partial evaluations of
    the tensorization $\Tbd[2][1] f$ are the constant functions $0$ and $1$.
    Thus, any Besov semi-norm of these partial evaluations is $0$ and
    consequently the sum as well. We see that, unlike in
    \Cref{isometryWkp}, even if a function $f$ has Besov regularity,
    the Besov norm of $f$ is in general not equivalent to the sum
    of the Besov norms of partial evaluations.
\end{example}

\begin{proposition}\label{isometryBesov}
Let $0<p=q\le \infty$ and $\alpha>0$. Let $\Bsqp[@][p][p]$ be equipped with the (quasi-)norm associated with the modulus of smoothness when $p=\infty$ or the averaged modulus of smoothness when $p<\infty$. Then, 
we equip the tensor space 
 $\mathbf{V}_{b,d,\Bsqp[@][p][p]} $ with the (quasi-)norm 
$$
 \Vert \tensor{f} \Vert_{{\Bsqp[@][p][p]}} := (\Vert \tensor{f}\Vert_p^p + \vert \tensor{f}\vert_{{\Bsqp[@][p][p]}}^p)^{1/p}\quad
 (p<\infty),\quad
 \Vert \tensor{f} \Vert_{{\Bsqp[@][\infty][\infty]}}
 :=\max\{\Vert \tensor{f}\Vert_\infty, \vert \tensor{f}\vert_{{\Bsqp[@][\infty][\infty]}}\},
$$
 where $\vert \cdot \vert_{{\Bsqp[@][p][p]}}$ is a (quasi-)semi-norm defined by
 \begin{align*}
    \vert \tensor{f} \vert_{{\Bsqp[@][p][p]}}^p :=b^{d(\alpha p-1)}
    \sum_{(j_1,\hdots,j_d) \in \Ib^d}
    \snorm{\tensor{f}(j_1,\hdots,j_d,\cdot)}[{\Bsqp[@][p][p]}]^p,
 \end{align*}
 for $p<\infty$, and
 \begin{align*}
    \vert \tensor{f} \vert_{{\Bsqp[@][\infty][\infty]}}:=
       b^{d\alpha} \max_{(j_1,\hdots,j_d) \in \Ib^d}
    \snorm{\tensor{f}(j_1,\hdots,j_d,\cdot)}[{\Bsqp[@][\infty][\infty]}]
 \end{align*}
 Then, $\Tbd(\Bsqp[@][p][p])\hookrightarrow\mathbf{V}_{b,d,\Bsqp[@][p][p]}$
 with
 \begin{align*}
     \snorm{f}[{\Bsqp[@][p][p]}] \geq
     \snorm{\Tbd(f)}[{\Bsqp[@][p][p]}].
 \end{align*}
 \end{proposition}

\subsection{Tensor Subspaces and Corresponding Function Spaces}

For a linear space of functions   $S\subset \R^{[0,1)}$, we define the tensor subspace
\begin{align*}
 \mathbf{V}_{b,d,S} :=  (\R^{\Ib})^{\otimes d} \otimes S \subset  \mathbf{V}_{b,d}, 
\end{align*}
and the corresponding linear subspace of functions in $\R^{[0,1)}$,
$$
\Vbd{S} = \Tbd^{-1}(\mathbf{V}_{b,d,S}).
$$ 
In the majority of this work we will be using finite-dimensional subspaces $S$ for approximation.
In particular, we will frequently use $S=\P_m$ where $\P_m$ is the space of polynomials of degree up to $m\in\N_0$. In this case we use the shorthand notation
\begin{align*}
    \Vbd{m}:=\Vbd{\P_m}.
\end{align*}
The tensorization   
 $\tensor{f}=\Tbd(f)$ of a function $f\in\Vbd{S}$ admits a representation \eqref{canonical-representation} with functions $  \tensor{f}(j_1,\hdots,j_d,\cdot) := f_{j_1,\hdots,j_d}$ in $S$. For $x = t_{b,d}(j_1,\hdots,j_d,y)$ in the interval  $[x_j,x_{j+1}),$ with $j = \sum_{k=1}^{b} j_k b^{d-k} $, we have $f(x) = f_{j_1,\hdots,j_d}(y) = f_{j_1,\hdots,j_d}(b^d x - j)$. Therefore, the   functions $f\in\Vbd{S}$ have restrictions on intervals $[x_j,x_{j+1})$
that are obtained by shifting and scaling functions in $S$. In particular, the space $ \Vbd{m}$ corresponds to the space of piecewise polynomials of degree $m$ over the uniform partition of $[0,1)$ with $b^d$ intervals. 

For considering functions with variable levels $d\in\N$, we introduce the set 
\begin{align*}
    \Vb{S}:=\bigcup_{d\in\N}\Vbd{S}.
\end{align*}
It is straight-forward to see that, in general,
\begin{align*}
    V_{b, d,S} \not\subset V_{b,\bar d,S}
\end{align*}
for $\bar{d}<d$.
E.g., take $S=\P_m$ and let $f\in\Vbd{m}$ be a piece-wise polynomial but discontinuous function. Then, clearly
$f$ does not have to be a polynomial over the intervals
\begin{align*}
    [kb^{-\bar{d}}, (k+1)b^{-\bar{d}}),\quad 0\leq k\leq b^{-\bar{d}}-1,
\end{align*}
for $\bar{d}<d$.
The same holds for the other inclusion, as the following example demonstrates.
\begin{example}\label{ex:cosines}
    Consider the one-dimensional subspace
    $
            S:=\linspan{\cos(2\pi\cdot)}.
   $
    A function $0\neq f\in\Vbd{S}$ is thus a piece-wise cosine.
    Take for simplicity $b=2$, $d=0$ (i.e., $\Vbd[2][0]{S}=S$) and $\bar{d}=1$.
    Then, $f\not\in\Vbd[2][1]{S}$ due to
    $
        \linspan{\cos(2\pi\cdot)}\not\supset\linspan{\cos(\pi\cdot),\,\cos(\pi+\pi\cdot)},
    $ 
    since cosines of different frequencies are linearly independent.
    The same reasoning can be applied to any $b\geq 2$ and $d,\,\bar{d}\in\N$ with $d<\bar{d}$.        
\end{example}
This motivates the following definition that is reminiscent
of multi-resolution analysis (MRA).
\begin{definition}[Closed under $b$-adic dilation]\label{def:badicdil}
We say that a linear space $S$ is closed under $b$-adic dilation if for any $f \in S$ and any $k \in \{0,\hdots,b-1\}$, 
$$
f(b^{-1}(\cdot + k)) \in S.
$$
\end{definition}
\begin{lemma}\label{b-adic-stability}
If $S$ is closed under $b$-adic dilation, then for all $f\in S$, 
$$
f(b^{-d}(\cdot + k)) \in S
$$
for all $d\in \N$ and $k \in \{0,\hdots,b^d-1\}$.
\end{lemma}	
\begin{proof}
    See \Cref{proof:badicstab}.
\end{proof}
Important examples of spaces $S$ that satisfy the above property include spaces of polynomials and MRAs.
The closedness of $S$ under $b$-adic dilation implies a hierarchy between spaces $\Vbd{S}$ with different levels, and provides $V_{b,S}$ with a linear space structure.  
\begin{proposition}\label{inclusion_Vbd}
If  $S$ is closed under $b$-adic dilation, then 
$$
S := V_{b,0,S} \subset   V_{b,1,S} \subset \hdots \subset  V_{b,d,S} \subset \hdots .
$$
\end{proposition}
\begin{proof}
    See \Cref{proof:inclusion}.
\end{proof}

\begin{proposition}[$\Vb{S}$ is a linear space]\label{lemma:vbvectorspace}
    If $S$ is closed under $b$-adic dilation, then $\Vb{S}$ is a linear space.
\end{proposition}
\begin{proof}
    See \Cref{proof:vectorspace}.
\end{proof}

If $S \subset \Lp({[0,1)})$, then $\Vb{S} $ is clearly a subspace of $\Lp({[0,1)}).$ However, 
it is not difficult to see that, in general, $\Vb{S}$ is not a closed subspace of $\Lp({[0,1)})$. On the other hand, we have the following density result.
\begin{theorem}[$\Vb{S}$ dense in $\Lp$]\label{thm:dense}
    Let $1\leq p< \infty$. If $S \subset \Lp([0,1)) $ and $S$ contains the constant function one, then $\Vb{S}=\bigcup_{d\in\N}\Vbd{S}$ is dense in $\Lp([0,1)).$
\end{theorem}
\begin{proof}
    See \Cref{proof:dense}.
\end{proof}

Now we provide bounds for ranks of functions in ${V}_{b,S}$, directly deduced from \eqref{rank-bound-general-format}.
\begin{lemma}\label{rank-bound-VbdS}
For $f \in \Vbd{S}$ and any $\beta \subset \{1,\hdots,d\},$
$$
r_{\beta,d}(f) \le \min\{b^{\#\beta} , b^{d-\#\beta} \dim S\}.
$$
In particular, for all $1\le \nu\le d$,
$$
r_{\nu,d}(f) \le \min\{b^\nu , b^{d-\nu} \dim S\}.
$$
\end{lemma}
In the case where $S$ is closed under $b$-adic dilation, we can obtain sharper bounds for ranks. 
\begin{lemma}\label{rank-bound-VbdS-closed-dilation} 
Let $S$ be closed under $b$-adic dilation.
\begin{enumerate}
\item[(i)] If $f \in S$, then for any $d\in \N$, $f\in V_{b,d,S}$ and we have 
$$
r_{\nu,d}(f) \le \min\{b^\nu , \dim S\}, \quad 1\le \nu\le d.
$$
\item[(ii)] If $f\in V_{b,d,S}$, then for any $\bar d \ge d$, $f \in V_{b,\bar d,S}$ and we have 
      \begin{align}
       & r_{\nu,\bar d}(f) = r_{\nu,d}(f) \leq \min\set{b^\nu, b^{d-\nu} \dim S},\quad  1\le \nu \le d,\notag\\
      &  r_{\nu,\bar d}(f)\leq \min\set{b^{\nu}  ,\, \dim S}, \quad d < \nu \le \bar d\label{eq:rjbig}.
    \end{align}
    \end{enumerate}
\end{lemma}
\begin{proof}
    See \Cref{proof:rankbound}.
\end{proof}

\begin{remark}
\Cref{rank-bound-VbdS-closed-dilation} shows that the ranks are independent of the representation level of a function $\varphi$, so that we will frequently suppress this dependence and simply note $r_{\nu,d}(\varphi) = r_{\nu}(\varphi)$ for any $d$ such that $\varphi \in V_{b,d,S}$.
\end{remark}

We end this section by introducing projection operators based on local projection. Let $\mathcal{I}_S$ be a linear projection operator from $\Lp([0,1))$ to a finite-dimensional space $S$. Then, define the linear operator $\mathcal{I}_{b,d,S} $ on $\Lp([0,1))$ defined for $f \in \Lp([0,1))$ by 
\begin{equation}
	(\mathcal{I}_{b,d,S} f) (b^{-d} (j +\cdot))=  \mathcal{I}_{S} (f(b^{-d} (j+\cdot ))), \quad 0\le j < b^d.\label{local-projection}
\end{equation}
\begin{lemma}[Local projection]\label{local-projection-structure}
The operator $\mathcal{I}_{b,d,S}$ is a linear operator from $\Lp([0,1))$ to $V_{b,d,S}$ and satisfies 
\begin{align}
T_{b,d} \circ \mathcal{I}_{b,d,S} \circ T_{b,d}^{-1} = id_{\{1,\hdots,d\}} \otimes \mathcal{I}_S. \label{tensor-structure-local-projection}
\end{align}
\end{lemma}
\begin{proof}
    See \Cref{proof:localproj}.
\end{proof}

We now provide a result on the ranks of projections. 
\begin{lemma}[Local projection ranks]\label{local-projection-ranks}
For any $f\in \Lp$, $\mathcal{I}_{b,d,S} f \in V_{b,d,S}$ satisfies 
$$
r_{\nu,d}(\mathcal{I}_{b,d,S} f) \le r_{\nu,d}(f) ,\quad 1\le \nu\le d.
$$
\end{lemma}
\begin{proof}
\Cref{local-projection-structure} implies that $T_{b,d} \circ \mathcal{I}_{b,d,S} \circ T_{b,d}^{-1}$ is a rank one operator. 
Since a rank-one operator can not increase $\beta$-ranks, we have for all $1\le \nu \le d$
     $$r_{\nu,d}(\mathcal{I}_{b,d,S}(f))  = r_{\nu}(T_{b,d}\circ \mathcal{I}_{b,d,S} \circ T_{b,d}^{-1} \boldsymbol{f}) = r_{\nu}((id_{\{1,\hdots,d\}} \otimes \mathcal{I}_S) \boldsymbol{f}) \le r_\nu(\boldsymbol{f}) = r_{\nu,d}(f).
     $$
     \end{proof}


%
%
%
%

%

\section{Tensor Networks and Their Approximation spaces}\label{sec:tensorapp}

In this section, we begin by describing particular tensor formats, namely tree tensor networks that will constitute our  approximation tool.
We then briefly review classical approximation spaces (see \cite{DeVore93}).
We conclude by introducing different measures of complexity of tree tensor networks and analyze the resulting approximation classes.

\subsection{Tree Tensor Networks and The Tensor Train Format}

Let $S$ be a finite-dimensional subspace of $\R^{[0,1)}.$
A tensor format in the tensor space $\mathbf{V}_{b,d,S} = (\R^{\Ib})^{\otimes d} \otimes S$ is defined as a set of tensors with $\beta$-ranks bounded by some integers $r_\beta$, for a certain collection $A$ of subsets $\beta \subset \{1,\hdots,d+1\},$
$$
\mathcal{T}^{A}_{\boldsymbol{r}}(\mathbf{V}_{b,d,S}) = \{\tensor{f} \in \mathbf{V}_{b,d,S} : r_\beta(\tensor{f}) \le r_\beta, \beta\in A\}.
$$ 
When $A$ is a dimension partition tree (or a subset of such a tree), the resulting format is called a hierarchical or tree-based tensor format \cite{hackbusch2009newscheme,Falco2018SEMA}. 
A tensor $\tensor{f}\in \mathcal{T}^{A}_{\boldsymbol{r}}(\mathbf{V}_{b,d,S}) $ in a tree-based tensor format admits a parametrization in terms of a collection of low-order tensors $v_\beta$, $\beta \in A$. Hence, the interpretation as a tree tensor network (see \cite[Section 4]{Nouy2019}). 
\begin{remark}\label{rem:equiv-NN}
Tree tensor networks are convolutional feedforward neural networks with non-linear feature maps, product pooling, a number of layers equal to the depth of the dimension partition tree and a number of neurons equal to the sum of ranks $r_\beta$.
 \end{remark}
 For the most part we will work  with the tensor train format with the exception of a few remarks. This format considers the collection of subsets $A = \{\{1\},\{1,2\}, \hdots,\{1,\hdots,d\}\}$, which is a subset of a linear dimension partition tree.

\begin{definition}[Tensor Train Format]\label{def:ttfuncs}
	The set\footnote{It is in fact a manifold, see \cite{Holtz2011,Falco2015,Falco2019,falco2020geometry}.} of tensors in $\mathbf{V}_{b,d}$ in tensor train (TT) format 
	with ranks at most $\bs r:=(r_\nu)_{\nu=1}^d$  is defined as
	\begin{align*}
	    \TT{\mathbf{V}_{b,d,S}}{\bs r}:=\set{\tensor{f} \in\mathbf{V}_{b,d,S}:\;r_\nu(\tensor{f})\leq r_\nu , \; 1 \le \nu\le d},
	\end{align*}
	where we have used the shorthand notation
	$
	r_\nu(\tensor{f}) := r_{\{1,\hdots,\nu\}}(\tensor{f}).
	$
	This defines a set of univariate functions 
	$$\Phi_{b,d,S,\bs r}= T_{b,d}^{-1} ( \TT{\mathbf{V}_{b,d,S}}{\bs r}) = \{f \in V_{b,d,S} : r_{\nu}(f) \le r_\nu , 1\le \nu \le d\},$$
	where $r_\nu(f) := r_{\nu,d}(f)$, that we further call the tensor train format for univariate functions.
\end{definition}    

Letting $\{\varphi_k\}_{k=1}^{\dim S}$ be a basis of $S$, a tensor $\tensor{f} \in\TT{\mathbf{V}_{b,d,S}}{\bs r}$ admits a representation
\begin{align}
    \tensor{f}(i_1,\ldots,i_d, y)&=\sum_{k_1=1}^{r_1}\cdots\sum_{k_d=1}^{r_d}\sum_{k=1}^{\dim S}v_1^{k_1}(i_1) v_2^{k_1,k_2}(i_2)
    \cdots v_d^{k_{d-1},k_d}(i_d) v^{k_d,k}_{d+1} \varphi_{k}(y), \label{repres-tt}
    \end{align}
with parameters $v_1 \in  \R^{b\times r_1}$, $v_\nu\in \R^{b\times r_{\nu-1} \times r_\nu}$, $2\leq \nu\leq d$,  and $v_{d+1} \in \R^{r_d \times \dim S}$ forming a tree tensor network $$\mathbf{v} = (v_1 ,\hdots,v_{d+1}) \in \mathcal{P}_{b,d,S,\bs r} :=   \R^{b\times r_1} \times \R^{b\times  r_1 \times r_2} \times \hdots  \times \R^{b\times  r_{d-1} \times r_d} \times \R^{r_d \times \dim S} .$$  
The  format $\TT{\mathbf{V}_{b,d,S}}{\bs r}$ then corresponds to the image of the space of tree tensor networks   
$\mathcal{P}_{b,d,S,\bs r}$  
through the map  $${R}_{b,d,S,\bs r} : \mathcal{P}_{b,d,S,\bs r} \to \TT{\mathbf{V}_{b,d,S}}{\bs r} \subset \mathbf{V}_{b,d,S}$$ such that for $\mathbf{v}  = (v_1,\hdots,v_{d+1}) \in \mathcal{P}_{b,d,S,\bs r}$, the tensor $\tensor{f}={R}_{b,d,S,\bs r}(\mathbf{v})$ is defined by \eqref{repres-tt}. The set of functions  $\Phi_{b,d,S,\bs r}$ in the tensor train format can be parametrized as follows: 
$$
\Phi_{b,d,S,\bs r}=  \{\varphi = \mathcal{R}_{b,d,S,\bs r}(\mathbf{v}) : \mathbf{v} \in \mathcal{P}_{b,d,S,\bs r}\}, \quad \mathcal{R}_{b,d,S,\bs r}(\mathbf{v}) = T_{b,d}^{-1} \circ  {R}_{b,d,S,\bs r}. 
$$
With an abuse of terminology, we call tensor networks both the set of tensors $\mathbf{v}$ and the corresponding function 
$\varphi = \mathcal{R}_{b,d,S,\bs r}(\mathbf{v})$.
The  representation complexity
of $\tensor{f} ={R}_{b,d,S,\bs r}(\mathbf{v}) \in\TT{\Vbd{S}}{\bs r}$ is
\begin{align}\label{eq:dof}
\mathcal{C}(b,d,S,\bs r) :=  \dim(\mathcal{P}_{b,d,S,\bs r}) = br_1+b\sum_{\nu=2}^dr_{\nu-1}r_\nu+r_d \dim S.
\end{align}
%



\begin{remark}[Re-Ordering Variables in the TT Format]\label{remark:varordering}
    We chose in \Cref{def-tensorization-map} to order the input variables of the tensorized function 
    $\tensor{f}$ such that $y\in[0,1)$ is in the last position.
     This specific choice allows the interpretation of partial evaluations of $\{1,\hdots,\nu\}$-unfoldings as contiguous pieces of $f=\Tbd^{-1}(\tensor{f})$ (see \Cref{link-minimal-subspaces} and the discussion thereafter).
    Alternatively, we could have chosen the ordering 
   $
        (y, i_1,\ldots, i_d)\mapsto\tensor{f}(y, i_1,\ldots,i_d),
    $
    and defined the TT-format and TT-ranks correspondingly.
    Essentially this is the same as considering a different tensor format, see
    discussion above.
    Many of the results of part II \cite{partII} remain the same.
    In particular, the order of magnitude of the rank bounds and therefore
    the resulting direct and inverse estimates would not change.
    However, this re-ordering may lead to slightly smaller rank bounds
    as in \Cref{remark:extranks} or slightly larger rank bounds as in
    \cite[Remark 4.6]{partII}.
\end{remark}

\subsection{General Approximation Spaces}

Approximation spaces have been extensively studied
in the second half of the last century.
They provide a systematic way of classifying functions that can be approximated with a certain rate.
Moreover, they have intriguing connections to smoothness and interpolation spaces (see \cite{DeVore98}) and
thus provide a complete characterization of approximation properties.
We briefly review some fundamentals that we require for the rest of this work.
For details we refer to \cite{DeVore93, DeVore98}.

Let $X$ be a quasi-normed linear space, $\tool\subset X$ subsets of $X$
for $n\in\N_0$ and $\tool[]:= (\tool)_{n\in \N_0}$ an approximation tool.
Define the best approximation error
\begin{align*}
E_n(f):= E(f,\Phi_n):=\inf_{\varphi\in\tool}\norm{f-\varphi}[X].
\end{align*}
With this we define approximation classes as

\begin{definition}[Approximation Classes]\label{def:asq}
    For any $f\in X$ and $\alpha>0$, define the quantity
    \begin{align*}
        \norm{f}[\Asq]:=
        \begin{cases}
            \left(\sum_{n=1}^\infty[n^\alpha E_{n-1}({f})]^q\frac{1}{n}\right)^{1/q},&\quad 0<q<\infty,\\
            \sup_{n\geq 1}[n^\alpha E_{n-1}({f})],&\quad q=\infty.
        \end{cases}
    \end{align*}
    The approximation classes $\Asq$ of $\tool[] = (\tool)_{n\in \N_0}$ are defined by  
    \begin{align*}
        \Asq:=\Asq(X):=\Asq(X,\tool[]):=\set{f\in X:\; \norm{f}[\Asq]<\infty}.
    \end{align*}
\end{definition}

These approximation classes have many useful properties if we further assume that $\tool$ satisfy the following criteria for any $n\in \N_0$.
\begin{enumerate}[label=(P\arabic*)]
    \item\label{P1}    $0\in\tool$, $\tool[0]=\set{0}$.
    \item\label{P2}     $\tool[n]\subset\tool[n+1]$.
    \item\label{P3}     $a\tool=\tool$ for any $a\in\R\setminus\set{0}$.
    \item\label{P4}     $\tool+\tool\subset\tool[cn]$ for some $c:=c(\tool[])$.
    \item\label{P5}     $\bigcup_{n\in \N_0} \tool$ is dense in $X$.
    \item\label{P6}     $\tool$ is proximinal in $X$, i.e. each $f\in X$ has a best approximation in $\tool$.
\end{enumerate}
Additionally, properties \ref{P1} -- \ref{P6} will be frequently
combined with the so-called \emph{direct} or \emph{Jackson}
inequality
\begin{align}\label{eq:jackson}
    E_{n}({f})\leq Cn^{-\rJ}\snorm{f}[Y],\quad\forall f\in Y,
\end{align}
for a semi-normed vector space $Y$ and some parameter $\rJ>0$,
and the \emph{inverse} or \emph{Bernstein} inequality
\begin{align}\label{eq:bernstein}
    \snorm{\varphi}[Y]\leq Cn^{\rB}\norm{\varphi}[X],\quad\forall\varphi\in\tool,
\end{align}
for some parameter $\rB>0$.

The implications of \ref{P1} -- \ref{P6} about the properties of $\Asq$ are as follows
\begin{itemize}
    \item    \ref{P1}+\ref{P3}+\ref{P4} $\Rightarrow$ $\Asq$ is a linear space with a quasi-norm.
    \item    \ref{P1}+\ref{P3}+\ref{P4} $\Rightarrow$ $\Asq$ satisfies the \emph{direct} or \emph{Jackson} inequality
                \begin{align*}
                    E_{n}({f})\leq Cn^{-\rJ}\norm{f}[\Asq],\quad\forall f\in\Asq,
                \end{align*}
                for $\rJ=\alpha$.
    \item    \ref{P1}+\ref{P2}+\ref{P3}+\ref{P4} $\Rightarrow$ $\Asq$ satisfies the \emph{inverse} or \emph{Bernstein} inequality
                \begin{align*}
                    \norm{\varphi}[\Asq]\leq Cn^{\rB}\norm{\varphi}[X],\quad\forall\varphi\in\tool,
                \end{align*}
                for $\rB=\alpha$.
\end{itemize}
The other properties \ref{P5} and \ref{P6} are required for characterizing approximation spaces by
interpolation spaces, see \cite{DeVore93}. Specifically,
\ref{P1} -- \ref{P4} together with a Jackson estimate as in \eqref{eq:jackson}
are required to prove so-called \emph{direct embeddings}:
a range of smoothness spaces is continuously embedded into $\Asq$.
While \ref{P1} -- \ref{P6} together with a Bernstein estimate
\eqref{eq:bernstein}
are required for \emph{inverse embeddings}:
$\Asq$ is continuously embedded into smoothness spaces.
We will see in Part II \cite{partII} that, in general, for approximation spaces of
tensor networks
no inverse estimates\footnote{The same was shown for RePU networks
in \cite{Gribonval2019}.}
are
possible,
since these spaces are ``too large''. Therefore,
properties
\ref{P5} -- \ref{P6} are not essential, while \ref{P5}
is typically true for any type of reasonable
approximation tool\footnote{Think of
\emph{universality theorems} for neural networks which hold
for tensor networks as well as we will see shortly.}.

We have the continuous embeddings
\begin{align*}
    \Asq\hookrightarrow\Asq[\beta][\bar{q}],\quad\text{if }\alpha>\beta\quad\text{or if }\alpha=\beta\text{ and }
    q \leq \bar{q}.
\end{align*}

We will see that, while most properties are easy to satisfy, property \ref{P4} will be the most critical one.
In essence \ref{P4} is a restriction on the non-linearity of the sets $\tool$,
with $c(\tool[])=1$ being satisfied by linear subspaces.


\subsection{Measures of Complexity}\label{sec:appclasses}
We consider as an approximation tool $\Phi$ the collection of tensor networks $\Phi_{b,d,S,\bs r}$ associated with different levels and ranks,
$$
\tool[] := (\Phi_{b,d,S,\bs r})_{d\in \N , \bs r \in \N^d},
$$
and define the sets of functions $\tool$ as
\begin{align}\label{eq:deftool}
    \tool:=\set{\varphi   \in \Phi_{b,d,S,\bs r} :   d\in \N , \bs r \in \N^d, \cost(\varphi)\leq n},
\end{align}
where $\cost(\varphi)$ is some measure of complexity of a function $\varphi.$ The approximation classes of tensor networks depend on the chosen measure of complexity.  We propose different measures of complexity and discuss the critical property \ref{P4}. 
A function $\varphi \in \tool[] $ may admit representations at different levels. We set $$d(\varphi)=\min\{d : \varphi \in \Vbd{S}\}$$ to be the minimal representation level of $\varphi$, and $\bs r(\varphi) = (r_\nu(\varphi))_{\nu=1}^{d(\varphi)}$ be the corresponding ranks. 
Measures of complexity may be based on a measure of complexity $\cost(\mathbf{v})$ of tensor networks $\mathbf{v}$ such that $\varphi = \mathcal{R}_{b,d,S,\bs r}(\mathbf{v})$. Then, we would define  
\begin{align}\label{eq:deftool-network}
    \tool:=\set{\varphi = \mathcal{R}_{b,d,S,\bs r}(\mathbf{v}) \in \Phi_{b,d,S,\bs r} :  \mathbf{v} \in \mathcal{P}_{b,d,S,\bs r} , d\in \N , \bs r \in \N^d, \cost(\mathbf{v})\leq n},
\end{align}
which is equivalent to the definition \eqref{eq:deftool} if we let 
\begin{align}\label{mincost}
\cost(\varphi) := \min \{\cost(\mathbf{v}) :  \mathcal{R}_{b,d,S,\bs r}(\mathbf{v}) = \varphi, d \in \N , \bs r\in \N^d\},
\end{align}
where the minimum is taken over all possible representations of $\varphi$.

\subsubsection{Complexity Measure: Maximum Rank}

In many high-dimensional approximation problems it is common to consider the maximum rank as
an indicator of complexity (see, e.g., \cite{Bachmayr2015}). By this analogy we consider for $\varphi\in \tool[]$,
\begin{align}\label{eq:rmaxdef}
    \cost(\varphi)&:=bd \rmax^2(\varphi) +\rmax(\varphi) \dim S, \quad \rmax(\varphi) = \max\set{r_\nu(\varphi) :\;1\leq \nu\leq d(\varphi)}.
\end{align}
This complexity measure does not satisfy \ref{P4}.

\begin{proposition}[\ref{P4} not satisfied by the complexity measure based on $\rmax$]\label{prop:maxrank}
    Let $S$ be closed under $b$-adic dilation and assume $\dim S<\infty$.
    Then, with $\tool$ as defined in \eqref{eq:deftool} with the measure of complexity \eqref{eq:rmaxdef}, 
    \begin{enumerate}[label=(\roman*)]
        \item\label{prop:rmaxi}    There exists no constant $c\in\R$ such that
                    \begin{align*}
                        \tool+\tool\subset\tool[cn].
                    \end{align*}
        \item    There exists a constant $c>1$ such that
                    \begin{align*}
                        \tool+\tool\subset\tool[cn^2].
                    \end{align*}
    \end{enumerate}
\end{proposition}
\begin{proof}
    See \Cref{proof:maxrank}.
\end{proof}

\subsubsection{Complexity Measure: Sum of Ranks}

For a neural network, a natural measure of complexity is the number of neurons.  By analogy (see \Cref{rem:equiv-NN}), we can define a complexity measure equal to the sum of ranks 
\begin{align}\label{eq:neuroncost}
  \cost(\varphi): = \cost_{\mathcal{N}}(\varphi) :=  
  \sum_{\nu=1}^{d(\varphi)} r_\nu(\varphi),
\end{align}
and the corresponding set 
\begin{align}\label{eq:deftool-N}
    \tool^{\mathcal{N}}:=\set{\varphi   \in \Phi_{b,d,S,\bs r} :   d\in \N , \bs r \in \N^d, \cost_{\mathcal{N}}(\varphi)\leq n}.
\end{align}
This complexity measure can be equivalently defined by \eqref{mincost} with $\cost_{\mathcal{N}}(\mathbf{v}) = \sum_{\nu=1}^{d} r_\nu $ for $\mathbf{v} \in \mathcal{P}_{b,d,S,\bs r}$.
\begin{lemma}[$\tool^{\mathcal{N}}$ satisfies \ref{P4}]\label{lemma:p4-N}
    Let $S$ be closed under $b$-adic dilation and $\dim S<\infty$.
    Then, the set $\tool^{\mathcal{N}}$ as defined in \eqref{eq:deftool-N} satisfies
    \ref{P4} with $c=2+\dim S$.
\end{lemma}
\begin{proof}
    See \Cref{proof:p4N}.
\end{proof}

\subsubsection{Complexity Measure: Representation Complexity}

A straight-forward choice for the complexity measure is
the number of parameters required for representing $\varphi$ as in 
\eqref{eq:dof}, i.e., 
\begin{align}\label{eq:dofcost}
  \cost(\varphi) := \cost_\mathcal{C}(\varphi) := \mathcal{C}(b,d(\varphi),S,\bs r(\varphi))  = 
  br_1(\varphi)+b\sum_{k=2}^{d(\varphi)}r_{k-1}(\varphi)  r_k(\varphi)+
    r_d(\varphi) \dim S,
\end{align}
and  the corresponding set is defined as
\begin{align}\label{eq:deftool-C}
    \tool^{\mathcal{C}}:=\set{\varphi   \in \Phi_{b,d,S,\bs r} :   d\in \N , \bs r \in \N^d, \cost_{\mathcal{C}}(\varphi)\leq n}.
\end{align}
This complexity measure can be equivalently defined by \eqref{mincost} with $\cost_{\mathcal{C}}(\mathbf{v}) = \mathcal{C}(b,d,S,\bs r)$ for $\mathbf{v} \in \mathcal{P}_{b,d,S,\bs r}$.
\begin{remark}
When interpreting tensor networks as neural networks (see \Cref{rem:equiv-NN}), the complexity measure $\cost_\mathcal{C}$ is equivalent to the number of weights for a fully connected neural network with $r_\nu$ neurons  in layer $\nu$.
\end{remark}
We can show the set $\tool^{\mathcal{C}}$ satisfies \ref{P4} with the help of \Cref{rank-bound-VbdS-closed-dilation,lemma:rankgrowth}.
\begin{lemma}[$\tool^{\mathcal{C}}$ satisfies \ref{P4}]\label{lemma:p4-C}
    Let $S$ be closed under $b$-adic dilation and $\dim S<\infty$.
    Then, the set $\tool^{\mathcal{C}}$ as defined in \eqref{eq:deftool-C}  satisfies
    \ref{P4} with $c=c(b, \dim S)>1$.
\end{lemma}
\begin{proof}
    See \Cref{proof:p4C}.
\end{proof}

\begin{remark}[Re-Ordering Input Variables]\label{remark:extranks}
    In the proof, we have used the property \eqref{eq:rjbig} from \Cref{lemma:rankgrowth}.
    As mentioned in Remark \ref{remark:varordering}, we could consider
    a different ordering of the input variables
    $        (y,i_1,\ldots,i_d)\mapsto\tensor{f}(y,i_1,\ldots,i_d),
    $
    and the corresponding TT-format.
    This would change \eqref{eq:rjbig} to
    \begin{align*}
        r_{\nu,\bar{d}}(f)=1,\quad d+1\leq \nu \leq \bar{d}.
    \end{align*}
    We still require $S$ to be closed under $b$-adic dilation to ensure
    $f\in\Vbd[@][\bar{d}]{S}$.
\end{remark}

\begin{remark}[{$\lp[2]$}-norm of Ranks]
    We also considered defining the complexity measure as a $\lp[2]$-norm of the tuple of ranks
    \begin{align*}
        \cost(\varphi):=b\sum_{k=1}^{d(\varphi) }r_k(\varphi)^2 +
        r_{d}(\varphi)\dim S.
    \end{align*}
    This definition satisfies \ref{P4} as well with
    analogous results as for the complexity measure $\cost_{\mathcal{C}}$ for direct and inverse embeddings.
    The $\lp[2]$-norm of ranks is less sensitive
    to rank-anisotropy than the representation complexity $\cost_{\mathcal{C}}(\varphi)$.
    Note that both complexity measures
    reflect the cost of representing a function with tensor networks, 
    not the cost of performing arithmetic operations,
    where frequently an additional power of $r$ is required (e.g., $\sim r^3$
    or higher).
\end{remark}

\subsubsection{Complexity Measure: Sparse Representation Complexity}
Finally, for a function $\varphi = \mathcal{R}_{b,d,S,\bs r}(\mathbf{v}) \in \Phi_{b,d,S,\bs r}$, we consider a complexity measure that takes into account the sparsity of the tensors $\mathbf{v} = (v_1,\hdots,v_{d+1})$,
\begin{align}\label{eq:sparsecost}
  \cost(\mathbf{v}) = \cost_\mathcal{S}(\mathbf{v}) :=  \sum_{\nu=1}^{d+1} \Vert v_\nu \Vert_{\ell_0},
\end{align}
where $\Vert v_\nu\Vert_{\ell_0}$ is the number of non-zero entries in the tensor $v_\nu$. 
By analogy with neural networks (see \Cref{rem:equiv-NN}), this corresponds to the number of non-zero weights for sparsely connected neural networks. We define the corresponding set as
\begin{align}\label{eq:deftool-S}
    \tool^{\mathcal{S}}:=\set{\varphi   \in \Phi_{b,d,S,\bs r} :   d\in \N , \bs r \in \N^d, \cost_{\mathcal{S}}(\varphi)\leq n}.
\end{align}
We can show the set $\tool^{\mathcal{S}}$ satisfies \ref{P4}. For that, we need the following two lemmas.
\begin{lemma}\label{ext-sparse}
Assume $S$ is closed under $b$-adic dilation and $\dim S<\infty$. Let $\varphi = \mathcal{R}_{b,d,S,\bs r}(\mathbf{v})\in \Phi_{b,d,S,\bs r}$ with $\bs r = (r_\nu)_{\nu=1}^d$. For $\bar d > d$,  there exists a representation $\varphi = \mathcal{R}_{b,\bar d,S, \overline{\bs r}}(\overline{ \mathbf{v}}) \in \Phi_{b,\bar d,S,\overline{\bs r}}$ with $\overline{\bs r} = (\bar r_\nu)_{\nu=1}^{\bar d}$ such that $\bar r_\nu = r_\nu$ for $1\le \nu \le d$ and $\bar r_\nu \le \max\{\dim S,b\} \dim S$ for $d<\nu \le \bar d$, and 
$$
\cost_{\mathcal{S}}(\overline{ \mathbf{v}}) \le b\cost_{\mathcal{S}}(  \mathbf{v}) + (\bar d-d) b^2(\dim S)^3.
$$
\end{lemma}
\begin{proof}
    See \Cref{proof:extsparse}.
\end{proof}

\begin{lemma}[Sum of Sparse Representations]\label{sum-functions-sparse}
Let $\varphi_A = \mathcal{R}_{b,d,S,\bs r^A}(\mathbf{v}_A) \in \Phi_{b,d,S,\bs r^A}$ and $\varphi_B = \mathcal{R}_{b,d,S,\bs r^B}(\mathbf{v}_B) \in \Phi_{b,d,S,\bs r^B}$.
Then, $\varphi_A  + \varphi_B$ admits a representation $ \varphi_A  + \varphi_B =  \mathcal{R}_{b,d,S,\bs r}(\mathbf{v}) \in  \Phi_{b,d,S,\bs r}$ with $ r_\nu= r_\nu^A +  r_\nu^B$ for $1\le \nu\le d$, and  
$$
\cost_{\mathcal{S}}(\varphi_A + \varphi_B) \le \cost_{\mathcal{S}}(\mathbf{v}) \le \cost_{\mathcal{S}}(\mathbf{v}_A) + \cost_{\mathcal{S}}(\mathbf{v}_B).
$$
\end{lemma}
\begin{proof}
    See \Cref{proof:sumsparse}.
\end{proof}


\begin{lemma}[$\tool^{\mathcal{S}}$ satisfies \ref{P4}]\label{lemma:p4-S}
    Let $S$ be closed under $b$-adic dilation and $\dim S<\infty$.
    Then, the set $\tool^{\mathcal{S}}$ as defined in \eqref{eq:deftool-S}  satisfies
    \ref{P4} with $c=\alert{b+1+b^2(\dim S)^3}$.
\end{lemma}
\begin{proof}
    Let $\varphi_A,\;\varphi_B\in\tool^{\mathcal{S}}$ with 
    $\varphi_A = \mathcal{R}_{b,d_A,S,\bs r^A}(\mathbf{v}_A) \in \Phi_{b,d_A,S,\bs r^A}$ and $\varphi_B = \mathcal{R}_{b,d_B,S,\bs r^B}(\mathbf{v}_B) \in \Phi_{b,d_B,S,\bs r^B}$
    and w.l.o.g. $d_A\leq d_B$. From \Cref{ext-sparse,sum-functions-sparse}, we know that 
    $\varphi_A+\varphi_B $ admits a representation $\varphi_A+\varphi_B = \mathcal{R}_{b,d_B,S,\bs r}(\mathbf{v})$ at level $d_B$ with 
    \begin{align*}
    \cost_{\mathcal{S}}(\mathbf{v}) &\le \alert{b}\cost_{\mathcal{S}}(\mathbf{v}_A) + \cost_{\mathcal{S}}(\mathbf{v}_B) + \alert{(d_B-d_A)b^2(\dim S)^3}\\
    &\le \alert{(b+1+b^2(\dim S)^3) n},
    \end{align*}
    which ends the proof.  \end{proof}
 
\subsubsection{Necessity of \ref{P4}}

We could consider replacing $n$ with $n^2$ in \ref{P4}, i.e.,
\begin{align}\label{eq:modp4}
    \tool+\tool\subset\tool[cn^2].
\end{align}
This implies that $\Asq$
as defined in Definition \ref{def:asq} is no longer a vector space.
The statements about Jackson and Bernstein inequalities as well as the relation
to interpolation and smoothness spaces is no longer valid as well.

One could try to recover the linearity of $\Asq$ by modifying Definition \ref{def:asq}.
In Definition \ref{def:asq} we measure \emph{algebraic} decay of $\E{f}$.
Algebraic decay is compatible with \ref{P4} that in turn ensures $\Asq$ is a vector space. We could reverse this by asking: what type of decay behavior is ``compatible'' with \eqref{eq:modp4}
in the sense that the corresponding approximation class would be a linear space?
We can introduce a growth function  $\gamma:\N\rightarrow\R^+$ with $\lim_n\gamma(n)=\infty$ and define an approximation class $A^\gamma_\infty$ of $\tool[] = (\tool)_{n\in \N_0}$  as 
$$
A^\gamma_\infty:=\set{f\in X:\; \sup_{n\geq 1} \, \gamma(n)\E[n-1]{f} <\infty}.
$$
With some elementary computations one can deduce that
if the growth function is of the form 
\begin{align*}
    \growth(n):=1+\ln(n),
\end{align*}
then \eqref{eq:modp4} implies $A^\gamma_\infty$ is closed under addition. 
However, functions in $A^\gamma_\infty$
have too slowly decaying errors
for any practical purposes such that we do not intend to analyze this space further.

We could instead ask what form of \ref{P4} would be compatible with a growth function such as
\begin{align*}
    \growth(n):=\exp(a n^\alpha),
\end{align*}
for some $a>0$ and $\alpha>0$, i.e., classes of functions with exponentially decaying
errors. In this case we would have to require $c=1$ in \ref{P4}, i.e.,
\begin{align*}
    \tool+\tool\subset\tool,
\end{align*}
in other words, $\tool$ is a linear space.

These considerations suggest that preserving \ref{P4} in its original form
is necessary to exploit the full potential of classical approximation theory
while preserving some flexibility in defining $\tool$.
Thus, we only consider definitions of $\cost(\cdot)$ that preserve \ref{P4}.

\subsection{Approximation Spaces of Tensor Networks}\label{sec:appspacestensors}

We denote by $\tool^{\mathcal{N}},$ $\tool^{\mathcal{C}}$ and  $\tool^{\mathcal{S}}$ the approximation set $\Phi_n$ asscociated with the measures of complexity $\cost_{\mathcal{N}}$, $\cost_{\mathcal{C}}$ and $\cost_{\mathcal{S}}$ respectively. 
Then, we define three  different families of approximation classes
 \begin{align}
N_q^\alpha(X) &:= A_q^\alpha(X , (\tool^{\mathcal{N}})_{n\in \N}),\\
C_q^\alpha(X) &:= A_q^\alpha(X ,  (\tool^{\mathcal{C}})_{n\in \N}),\\
S_q^\alpha(X) &:= A_q^\alpha(X ,  (\tool^{\mathcal{S}})_{n\in \N}),
 \end{align}
 with $\alpha>0$ and $0<q \le \infty$. 
 Below, we will show that these approximation classes are in fact approximation spaces and we will then compare  these spaces.

\subsubsection{Approximation Classes are Approximation Spaces}

We proceed with checking if 
$\tool^{\mathcal{N}},$ $\tool^{\mathcal{C}}$ and  $\tool^{\mathcal{S}}$ 
satisfy properties \ref{P1}--\ref{P6}. In particular, satisfying \ref{P1}--\ref{P4}
will imply that the corresponding approximation classes are quasi-normed Banach spaces. 
The only property -- other than \ref{P4} -- that is not obvious, is \ref{P6}.
This is addressed in the following Lemma for $\tool^{\mathcal{N}}$ and $\tool^{\mathcal{C}}$.

\begin{lemma}[$\tool^{\mathcal{N}}$ and $\tool^{\mathcal{C}}$ satisfy \ref{P6}]\label{lemma:p6}
    Let $1<p<\infty$ and let $S\subset\Lp$ be a closed subspace. Then,
    $\tool^{\mathcal{N}}$ and $\tool^{\mathcal{C}}$ are proximinal in $\Lp$ for any $n\in\N$.
    Moreover, if $S$ is finite-dimensional, the above holds for
    $1\leq p\leq\infty$.
\end{lemma}
\begin{proof}
    See \Cref{proof:p6}.
\end{proof}

As the following example shows, we cannot in general
guarantee \ref{P6} for $\tool^{\mc S}$.

\begin{example}
    Suppose $b\geq 3$ and $\dim S\geq 3$. Take two linearly
    independent vectors $v,w\in\R^b$ and $f,g\in S$.
    For any $N\in\N$, set
    \begin{align*}
        \varphi_N:=(w+Nv)\otimes(v+\frac{1}{N}w)\otimes f+
        v\otimes \alert{v}\otimes(g-Nf),
    \end{align*}
    and
    \begin{align*}
        \varphi:=v\otimes v\otimes g+v\otimes w\otimes f+w\otimes v\otimes f.
    \end{align*}
    Then, the following holds
    (see \cite[Proposition 9.10 and Remark 12.4]{Hackbusch2012}).
    \begin{enumerate}[label=(\roman*)]
        \item    For the canonical tensor rank, we have
                    $r(\varphi_N)=2$ for any $N\in\N$ and
                    $r(\varphi)=3$.
        \item    As we will see in \Cref{lemma:canonical},
                    $\varphi_N\in\tool[6b+2\dim S]^{\mc S}$
                    for any $N\in\N$ and
                    $\varphi\in\tool[9b+3\dim S]^{\mc S}$. Moreover,
                    this complexity is minimal for both functions.
        \item     For $N\rightarrow\infty$, $\varphi_N\rightarrow\varphi$
                    in any norm.
    \end{enumerate}
    In other words, $E^{\mc S}_{6b+2\dim S}(\varphi)=0$,
    even though $\varphi\not\in\tool[6b+2\dim S]^{\mc S}$.
\end{example}

\begin{remark}
    \ref{P6} is required for showing that the approximation spaces
    $\Asq$ are continuously embedded into interpolation spaces,
    see \cite[Chapter 7, Theorem 9.3]{DeVore93}.
    Since inverse embeddings for $N^\alpha_q$, $C^\alpha_q$ and
    $S_q^\alpha$ hold only in
    very restricted cases,
    property \ref{P6} is not essential for our work
    and the majority of our results.
    As a side note, \ref{P6} does not
    hold for ReLU or RePU networks as was discussed in
    \cite{Gribonval2019}.
\end{remark}

We now derive the main result of this section.
\begin{theorem}[Properties of $\tool^{\mathcal{N}}$, $\tool^{\mathcal{C}}$ and $\tool^{\mathcal{S}}$]\label{p1-p6}
    Let $0<p\leq\infty$, $S\subset\Lp$ be a closed subspace that
    is also closed under $b$-adic dilation and $\dim S<\infty$.  Then,
    \begin{enumerate}[label=(\roman*)]
        \item    $\tool^{\mathcal{N}}$ and $\tool^{\mathcal{C}}$ and $\tool^{\mathcal{S}}$ satisfy \ref{P1} -- \ref{P4}.
        \item    If $1\leq p\leq\infty$, then $\tool^{\mathcal{N}}$, 
        $\tool^{\mathcal{C}}$
        additionally satisfy \ref{P6}.
        \item    If $1\leq p<\infty$ and if $S$ contains the constant function one,
                    $\tool^{\mathcal{N}}$, $\tool^{\mathcal{C}}$ and $\tool^{\mathcal{S}}$ additionally satisfy \ref{P5}.
    \end{enumerate}
\end{theorem}
\begin{proof}
    \ref{P1} -- \ref{P3} are obvious and \ref{P4} follows from  \Cref{lemma:p4-N,lemma:p4-C,lemma:p4-S}, that yields (i). 
    (iii) follows from the fact that 
    $$
\bigcup_{n\in \N} \tool[n] = \bigcup_{d \in \N}  \bigcup_{\bs r \in \N^d}  \tool[b,d ,S,\bs r] = 
\bigcup_{d \in \N} {V}_{b,d,S} =  {V}_{b,S},
$$
and from \Cref{thm:dense}. Finally,  
    (ii) follows from  \Cref{lemma:p6}.
\end{proof}
\Cref{p1-p6} (i)  implies that the approximation classes $N_q^\alpha(L^p) $, 
$C_q^\alpha(L^p) $ and 
$S_q^\alpha(X) $ 
are quasi-normed vector spaces that satisfy the Jackson and Bernstein inequalities.

\subsubsection{Comparing Approximation Spaces}
For comparing approximation spaces $N_q^\alpha(L^p) $, 
$C_q^\alpha(L^p) $ and 
$S_q^\alpha(L^p) $, we first provide some relations between the sets $\tool^{\mathcal{N}}$, $\tool^{\mathcal{C}}$ and $\tool^{\mathcal{S}}$.
\begin{proposition}\label{comparing-complexities}
For any $n\in \N$, 
\begin{align*}
&\Phi_n^{\mathcal{C}} \subset \Phi_n^{\mathcal{S}} \subset \Phi_n^{\mathcal{N}} \subset \Phi_{b\dim S + b n^2}^{\mathcal{C}}.
\end{align*}
\end{proposition}
\begin{proof}
    See \Cref{proof:cfcomplex}.
\end{proof}

From   \Cref{comparing-complexities}, we obtain\footnote{Compare to similar
results obtained for RePU networks in \cite[Section 3.4]{Gribonval2019}.}	
\begin{theorem}\label{comparing-spaces}
For any $\alpha>0$, $0<p\leq\infty$ and $0<q \le \infty$,
the classes $N_q^\alpha(\Lp)$, $C_q^\alpha(\Lp)$
and $S_q^\alpha(\Lp)$ satisfy the continuous embeddings
\begin{align*}
&C^{\alpha}_q(L^p) \hookrightarrow S^{\alpha}_q(L^p) \hookrightarrow N^{\alpha}_q(L^p)\hookrightarrow C^{\alpha/2}_q(L^p).
\end{align*}
\end{theorem}

\subsection{About The Canonical Tensor Format}
We conclude by comparing tensor networks with the canonical tensor format
$$
\mathcal{T}_{r}(V_{b,d,S}) = \{\tensor{f} \in \mathbf{V}_{b,d,S}  : r(\tensor{f}) \le r\},
$$
which is the set of tensors that admit a representation
$$
\tensor{f}(i_1,\hdots,i_d,y) = \sum_{k=1}^r w_1^{k}(i_1)\hdots w_d^k(i_d) g_{d+1}^k(y), \quad 
g_{d+1}^k(y) = \sum_{q=1}^{\dim S} w^{q,k}_{d+1} \varphi_q(y),
$$
with $w_\nu \in \R^{b\times r}$ for $1\le \nu \le d$ and $w_{d+1} \in \R^{\dim S\times r }.$  
The canonical tensor format can be interpreted as a shallow
sum-product neural network (or arithmetic circuit), see \cite{Cohen2016}.

We let ${R}_{b,d,S,r} $ be the map from $ (\R^{b\times r})^{d} \times  \R^{\dim S\times r } := \mathcal{P}_{b,d,S,r}$ to $\mathbf{V}_{b,d,S}$  which associates to a set of tensors $(w_1,\hdots,w_{d+1})$ the tensor 
$\tensor{f} = {R}_{b,d,S,r} (w_1,\hdots,w_{d+1})$ as defined above. 
We introduce the sets of functions  
$$
\Phi_{b,d,S,r} = T_{b,d}^{-1} \mathcal{T}_{r}(V_{b,d,S}), 
$$
which can be parametrized as follows:
$$
\Phi_{b,d,S,r} = \{\varphi = \mathcal{R}_{b,d,S,r}(\mathbf{w}) : \mathbf{w} \in \mathcal{P}_{b,d,S,r}\}, \quad \mathcal{R}_{b,d,S,r} = T_{b,d}^{-1} \circ {R}_{b,d,S,r}.
$$
For $\varphi \in V_{b,S} $, we let 
$$r(\varphi) = \min\{r(\tensor{f}) : \tensor{f} \in \mathbf{V}_{b,d(\varphi),S}, \varphi = T_{b,d}^{-1}(\tensor{f})\}.$$
We introduce as a natural complexity measure the representation complexity 
$$
\cost_{\mathcal{R}}({\varphi})= bd(\varphi) r(\varphi) +  r(\varphi) \dim S,
$$
 define the sets 
$$
\Phi^{\mathcal{R}}_{n} = \{ \varphi \in \Phi_{b,d,S,r}  : d\in \N , r \in \N, \cost_{\mathcal{R}}(\varphi) \le n\},
$$ 
and consider the corresponding approximation classes 
$$
R^{\alpha}_q(L^p) = A^{\alpha}_q(L^p,(\Phi^{\mathcal{R}}_n)_{n\in \N}), 
$$
with $\alpha>0$ and $0< q\le \infty .$ 
We start by showing that $\Phi^{\mathcal{R}}_{n}$ satisfies \ref{P1}-\ref{P3} and $\ref{P5}$ (under some assumptions), 
but not \ref{P4}.
\begin{lemma}[$\Phi^{\mathcal{R}}_{n}$ satisfies  \ref{P1}-\ref{P3} and $\ref{P5}$]
Let $1\le p \le \infty$ and $S\subset L^p$ be a finite-dimensional space. Then $\Phi^{\mathcal{R}}_{n}$ satisfies \ref{P1}-\ref{P3}. Moreover, if $S$ contains the constant function one, $\Phi^{\mathcal{R}}_{n}$ satisfies \ref{P5} for $1\le p<\infty$.
\end{lemma}
\begin{proof}
\ref{P1}-\ref{P3} are obvious. \ref{P5} follows from the fact that 
    $$
\bigcup_{n\in \N} \tool^{\mathcal{R}} = \bigcup_{d \in \N}  \bigcup_{ r \in \N}  \tool[b,d ,S, r] = 
\bigcup_{d \in \N} {V}_{b,d,S} =  {V}_{b,S},
$$
and from Theorem \ref{thm:dense}. 
\end{proof}
\begin{lemma}[$\Phi^{\mathcal{R}}_{n}$ does not satisfy   \ref{P4}]\label{lemma:cpnop4}
Let $1\le p \le \infty$ and $S\subset L^p$ be a finite-dimensional subspace which is closed under $b$-adic dilation and such that $r(T_{b,d} (\varphi))=1$ for any $\varphi\in S$ and $d\in \N$. Then, $\Phi^{\mathcal{R}}_{n}$ satisfies 
\begin{itemize}
\item[(i)] $\Phi^{\mathcal{R}}_{n} + \Phi^{\mathcal{R}}_{n} \subset \Phi^{\mathcal{R}}_{3 n^2}$,
\alert{\item[(ii)] there exists no constant $c>1$ such that $\Phi^{\mathcal{R}}_{n} + \Phi^{\mathcal{R}}_{n} \subset \Phi^{\mathcal{R}}_{cn}.$}
\end{itemize}
\end{lemma}
\begin{proof}
    See \Cref{proof:cpnop4}.
\end{proof}

\begin{lemma}\label{lemma:canonical}
For any $n\in \N$,  it holds
$\Phi^{\mathcal{R}}_n \subset \Phi^{\mathcal{S}}_n. $\end{lemma}
\begin{proof}
    See \Cref{proof:canonical}.
\end{proof}
	
\begin{corollary}
For any $\alpha>0$ and $0<q \le \infty$,
$$R_q^\alpha(L^p) \subset S_q^\alpha(L^p).$$
\end{corollary}
	
\section{Tensor Networks as Neural Networks -- The Role of Tensorization}\label{sec:discussion}

The tensorization of functions is a milestone allowing the use of tensor networks for the approximation of multivariate functions.
In this section, we interpret tensorization as a non-standard and powerful featuring step which can be 
 encoded in a neural network with a non-classical architecture. Then, we discuss the role of this particular featuring. 

\subsection{Tensorization as Featuring.}

When applying $t_{b,d}^{-1}$ to the input variable $x$, we create $d+1$ new variables $(i_1,...,i_d,y)$ defined by
$$i_\nu = \sigma(b^\nu x), \quad  \sigma(t) = \lfloor t\rfloor \, mod \, b ,$$
$1\le \nu\le d,$ and 
$$y = \tilde \sigma(b^d x), \quad \tilde \sigma(t) = t - \lfloor t\rfloor,$$
see  \Cref{fig:sigma-sigmatilde} for a graphical representation of functions $\sigma$ and $\tilde \sigma$.
\begin{figure}[h]
    \begin{subfigure}{0.47\textwidth}
        \center
        \includegraphics[scale=.2]{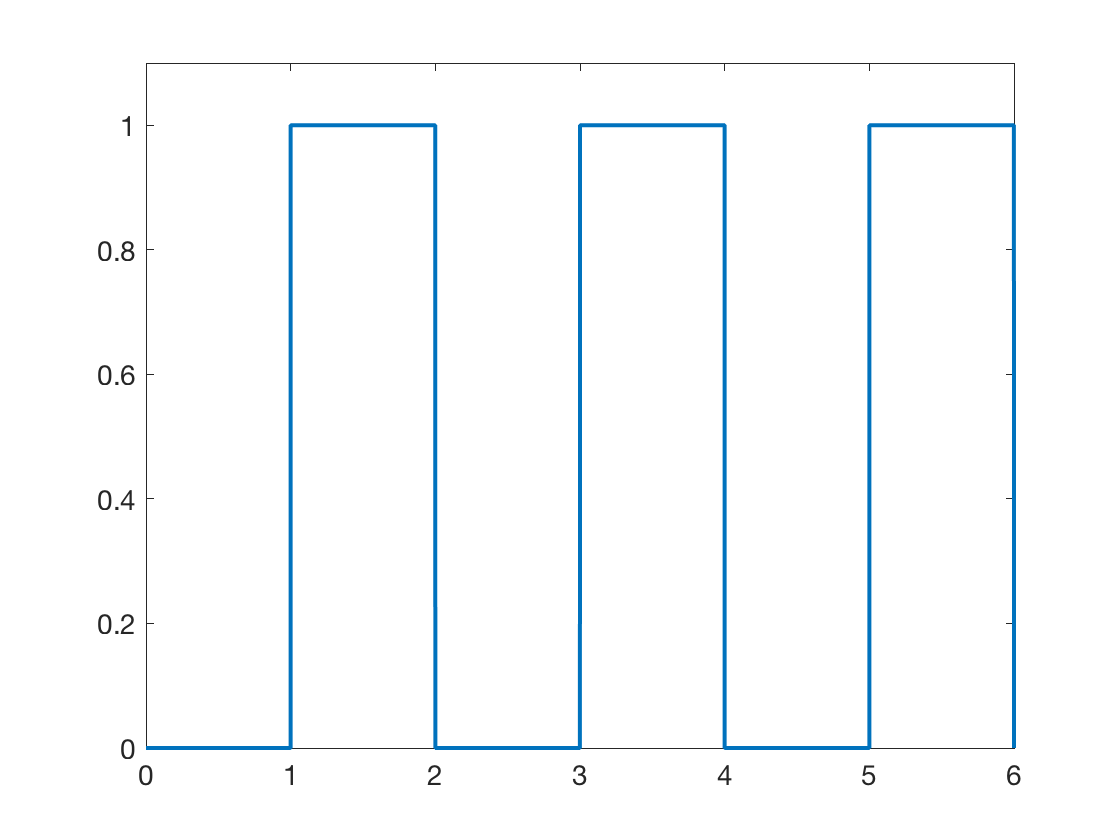}
        \caption{$\sigma$}\label{fig:sigma}
    \end{subfigure}
    \begin{subfigure}{0.47\textwidth}
        \center
        \includegraphics[scale=.2]{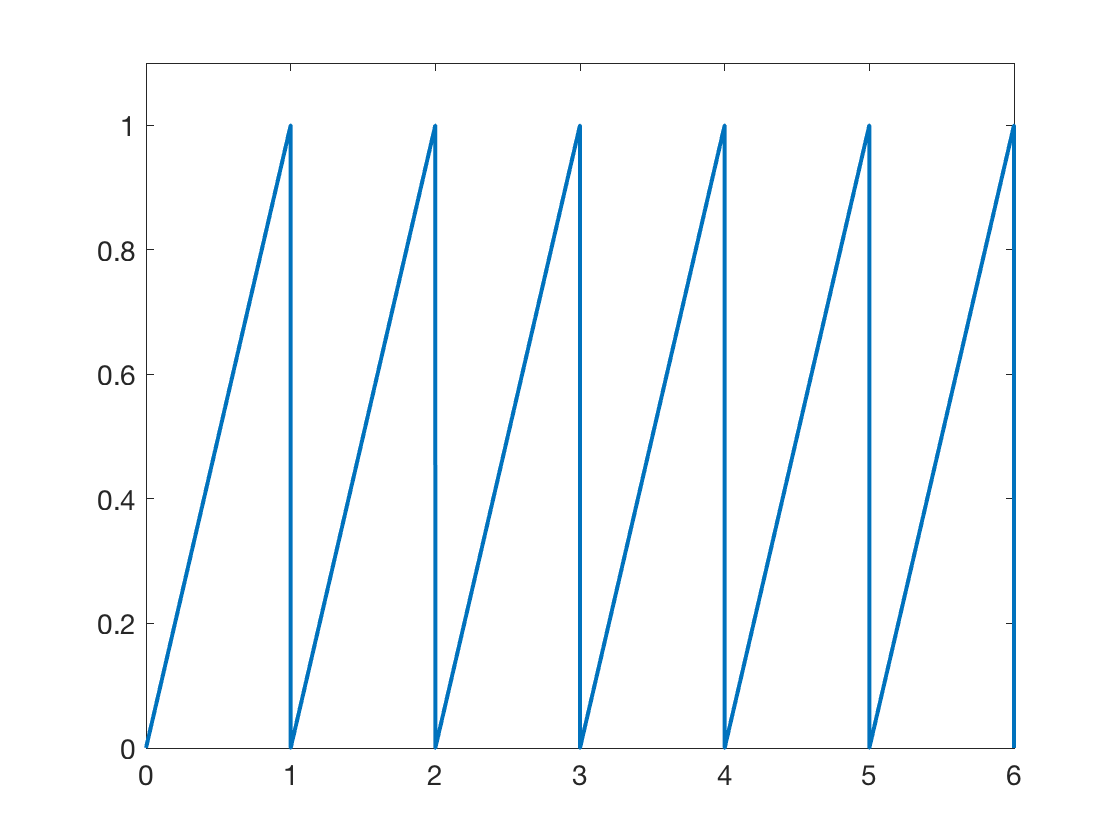}
        \caption{$\tilde \sigma$}\label{fig:sigmatilde}
    \end{subfigure}
    \caption{Functions $\sigma$ and $\tilde \sigma$}
    \label{fig:sigma-sigmatilde}
\end{figure}
Then for each $1\le \nu \le d$, we create $b$ features $\delta_{j_\nu}(i_\nu)$, $0\le j_\nu \le b-1$, and we also create  $m+1$ features  $\varphi_{k}(y) = y^{k}$ from the variable $y$ (or other features for $S$ different from $\P_m$).\footnote{For $m=0$, the extra variable $y$ is not exploited. For $m=1$, we only consider the variable $y$ and for $ m>1$, we exploit more from this variable.}
\Cref{fig:features} provides an illustration of these features and of products of these features. 
\begin{figure}[h]
    \begin{subfigure}{0.47\textwidth}
        \center
        \includegraphics[scale=.2]{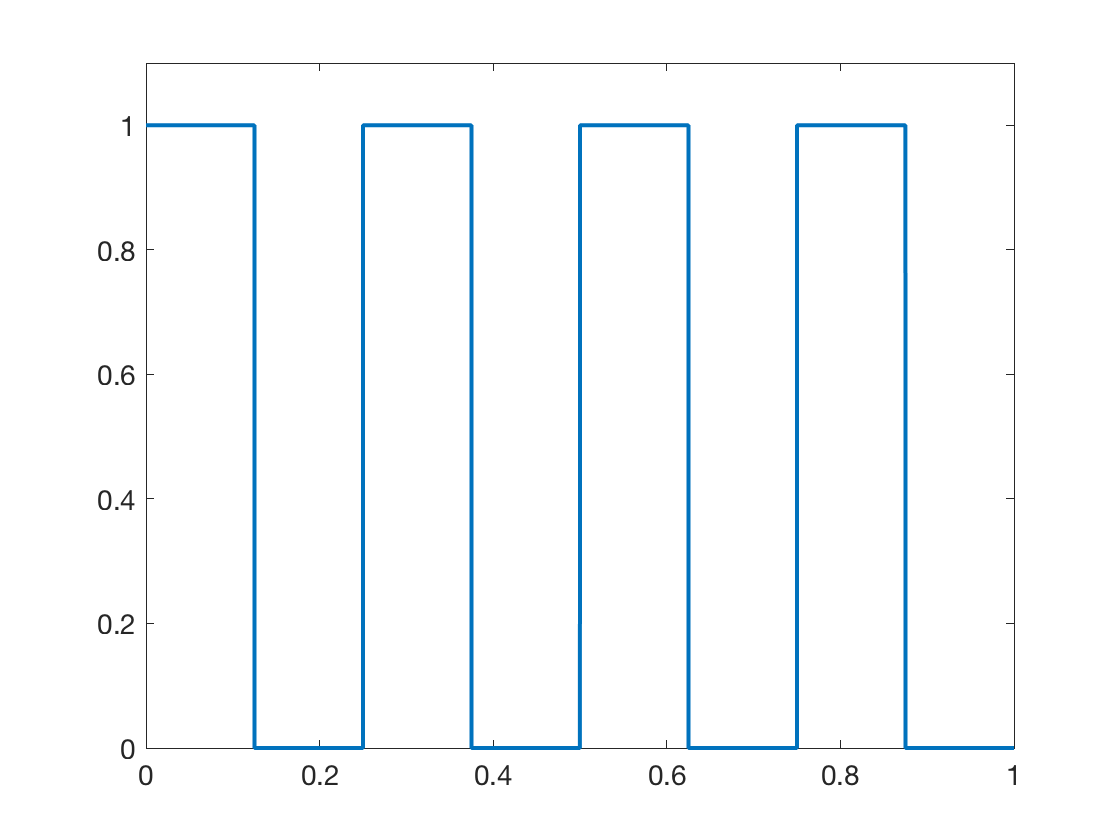}
        \caption{$\delta_0(\sigma(b^3 x))$}\label{fig:feat1}
    \end{subfigure}
    \begin{subfigure}{0.47\textwidth}
        \center
        \includegraphics[scale=.2]{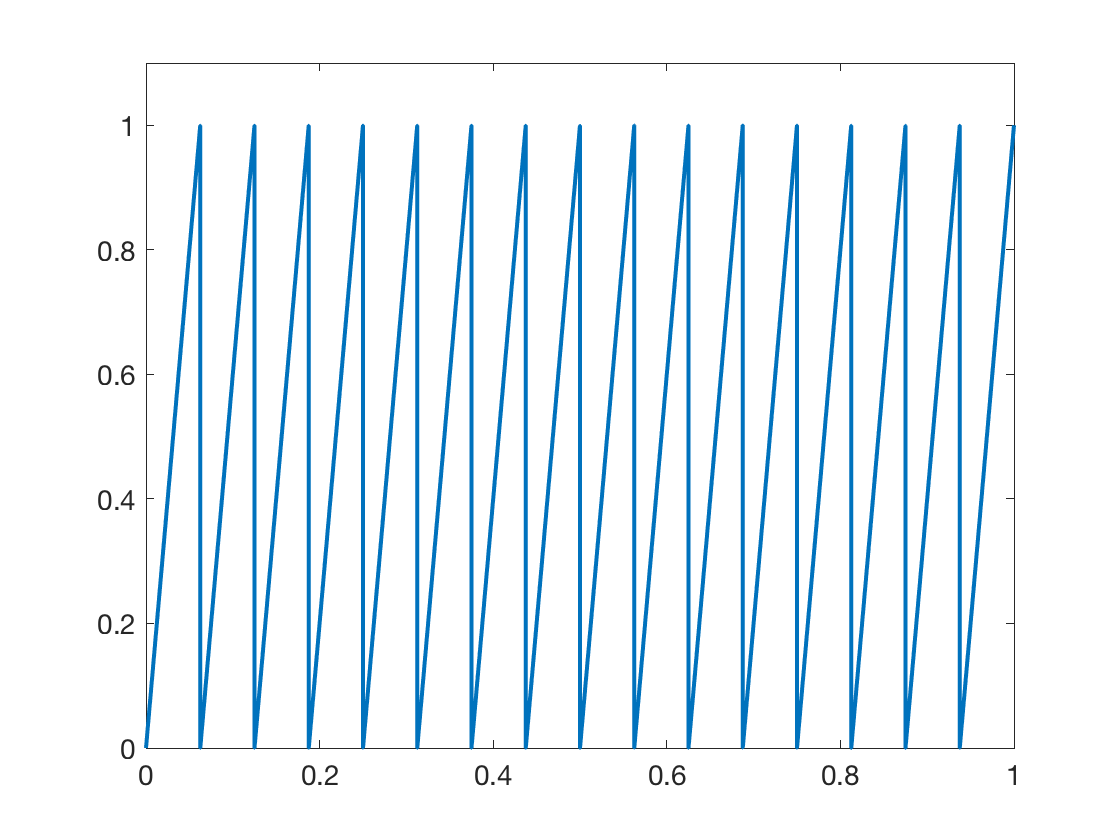}
        \caption{$\tilde \sigma(b^4 x)$}\label{fig:feat2}
            \end{subfigure}
                \begin{subfigure}{0.47\textwidth}
        \center
        \includegraphics[scale=.2]{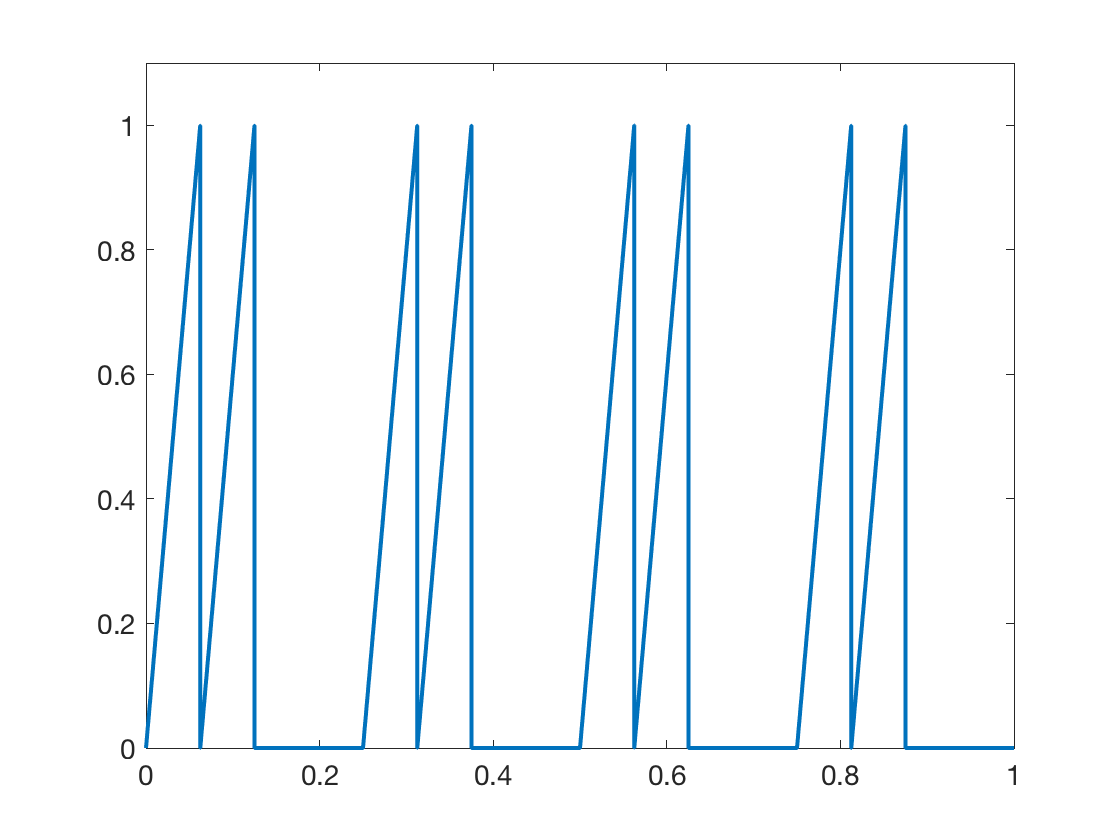}
        \caption{$\delta_0(b^3x) \tilde \sigma(b^4 x)$}\label{fig:feat3}             
    \end{subfigure}
    \begin{subfigure}{0.47\textwidth}
        \center
        \includegraphics[scale=.2]{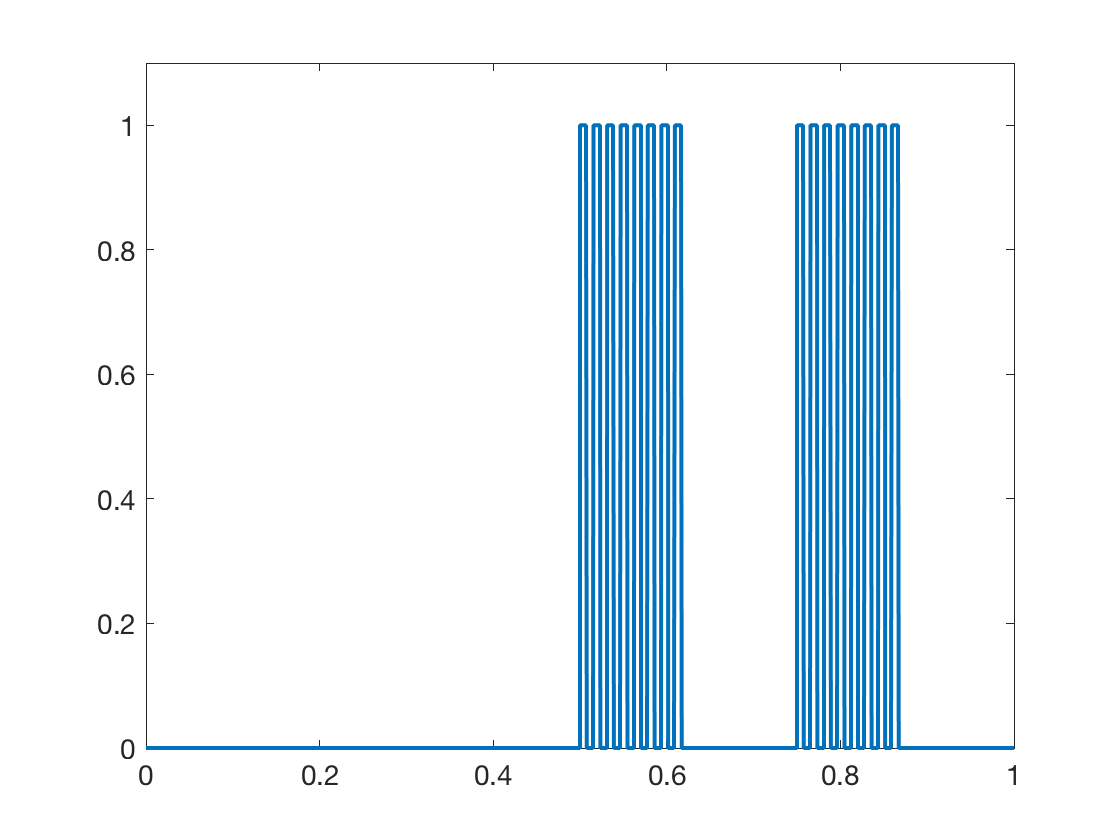}
        \caption{$\delta_1(\sigma(b x)) \delta_0(\sigma(b^3 x)) \delta_0(b^7x)$}\label{fig:feat4}
    \end{subfigure}
    \caption{Representation of some features and their products for $b=2$.}
    \label{fig:features}
\end{figure}
Finally, tensorization can be seen as a featuring step with a featuring map $$\Phi : [0,1) \to \R^{b^d (m+1)}$$
which maps $x\in [0,1)$ to a $(d+1)$-order tensor 
$$
\Phi(x)_{j_1,\hdots,j_{d+1}} = \delta_{j_1}(\sigma(b x)) \hdots \delta_{j_d}(\sigma(b^d x))  \tilde \sigma(b^d x)^{j_{d+1}}.
$$
A function $\varphi \in V_{b,d,m}$ is then represented by $\varphi(x) = \sum_j \Phi(x)_j a_j $, where $a$ is a $(d+1)$-order tensor with entries associated with the $b^d(m+1)$ features.  
When considering for $a$ a full tensor (not rank-structured), it results in a linear approximation tool which is equivalent to spline approximation. Note that functions represented on \Cref{fig:feat3,fig:feat4} are obtained by summing many features $\Phi(x)_{j_1,\hdots,j_{d+1}}$. However, these functions, which have rank-one tensorizations, can be represented with a rank-one tensor $a$ in the feature tensor space, and thus can be encoded with very low complexity within our nonlinear approximation tool.
 
  Increasing $d$ means considering more and more features, and is equivalent to refining the discretisation.
At this point, tensorization is an interpretation of a univariate function as a multivariate function, but it is also an alternative way to look at discretization.

\subsection{Encoding Tensorization with a Neural Network.}
The tensorization step can be encoded as a three-layer feedforward neural network with a single intput $x$ and $b^d(m+1)$ outputs
corresponding  to the entries of $ \Phi(x)$. The variables 
 $(i_1,...,i_d,y)$ can be seen as the output of a first layer with $d+1$ neurons, which corresponds to the application of a linear map $x\mapsto (b,b^2,\hdots,b^d , b^d) x $ to the input variable $x$, followed by a component-wise application of the activation function $\sigma$ (for the first $d$ neurons) or $\tilde \sigma$ (for the last neuron).  
Noting that $\delta_{j_\nu}(i_\nu) = \indicator{[0,1)}(i_\nu - j_\nu)$, the variable $\delta_{j_\nu}(i_\nu) $ can be seen as the output of a neuron which applies the activation function $t\mapsto \indicator{[0,1)}(t)$ to a shift of the input $i_\nu$. The variables $y^k$ correspond to the ouputs of a classical power unit with input $y$. Therefore, 
the resulting  variables $\delta_{j_\nu}(i_\nu)$ ($1\le \nu \le  d$, $0\le j_\nu \le b-1$) and $y^k$   ($0\le k\le m$)
 are the outputs of a second layer with $b d + (m+1)$ neurons, using the activation function $t\mapsto \indicator{[0,1)}(t)$ (for the first $bd$ neurons) or classical power activation functions $y\mapsto y^k$ for the last $m+1$ neurons. 
The third  layer corresponds to a product pooling layer which creates new variables  that are products of $d+1$ variables 
 $\{ \delta_{j_1}(i_1) ,\hdots, \delta_{j_d}(i_d) , \varphi_{j_{d+1}}(y)\}$.
 
The approximation tool considered in this work then corresponds to neural networks with three first layers implementing the particular featuring map $\Phi$ followed by a tensor network or sum-product network (with recurrent network architecture).
Note that, if instead of using a sum-product network we take a simple linear combination of the outputs of the third layer, 
we end up with a neural network implementing classical splines with degree $m$ and $b^d$ knots.

\subsection{The Role of Tensorization}

Another featuring (which is rather straight-forward) would have consisted in taking new variables (or features) $x_{j,k} = (b^d x-j)^k \indicator{I_j}(x)=x^k \indicator{[0,1)}(b^d x-j)$, $0\le k\le m$, $0\le j< b^d$, where $I_j$ is the interval $[b^d j , b^d(j+1))$. This also leads to $(m+1)b^d$ features. This featuring step can be encoded with a two-layer neural network, the first layer with $b^d$ neurons implementing affine transformations $x \mapsto b^d x-j$ followed by the application of the activation function $t\mapsto t \indicator{[0,1)}(t)$, and a second layer with $b^d(m+1)$ neurons that apply classical power activation functions to the outputs of the first layer.  When considering a simple linear combination of the outputs of this two-layer neural network, we also end up with   
classical fixed knot spline approximation. 

Both featuring (or tensorization) methods can be encoded with feed-forward neural networks, and both methods lead to a linear feature space corresponding to classical  spline approximation. One may ask what is the interest of using the very specific feature map $\Phi$? In fact, the use of the particular feature map $\Phi$, which is related to multi-resolution analysis, allows to further exploit sparsity or low-rankness of the tensor when  approximating functions from smoothness spaces, and probably other classes of functions. 
It is well known that the approximation class of splines of degree $m \ge r-1$ is the Sobolev space $W^{r,p}.$ Therefore, whatever the featuring used, the approximation class of the resulting linear approximation tool (taking linear combinations of the features) is the Sobolev space $W^{m+1,p}.$ 
We will see in Part II \cite{partII} that (near-)optimal performance is achieved by the proposed approximation tool  for a large range of smoothness spaces for any fixed $m$ (including $m=0$), at the price of letting $d$ grow (or equivalently the depth of the tensor networks) to capture higher regularity of functions. When working with a fixed $m$, exploiting low-rank structures will then be crucial. 

This reveals that the power of the approximation tool considered in this work comes from the combination of  a particular featuring step (the tensorization step) and the use of tensor networks. 

\bibliographystyle{abbrv}
\bibliography{literature}

\appendix

\newpage
\section{Proofs for \Cref{sec:tensorization}}

\begin{proof}[Proof of \Cref{tbdbard}]
\label{proof:lemma26}
We have 
\begin{align*}
t_{b,\bar d}(i_1,\hdots,i_{\bar d},y) &=  \sum_{k=1}^{\bar d} i_k b^{-k} +b^{-\bar d} y = 
\sum_{k=1}^{d} i_k b^{-k} +  \sum_{k=1}^{\bar d -d} i_{k+d} b^{-k-d} + b^{-\bar d} y = \sum_{k=1}^{d} i_k b^{-k} + b^{-d} z,
\end{align*}
with $z = \sum_{k=1}^{\bar d -d} i_{k+d} b^{-k} + b^{-(\bar d-d)} y = t_{b,\bar d- d}(i_{d+1},\hdots,i_{\bar d},y)$, which proves the first statement. Then consider an elementary tensor  $\tensor{v} = v_1 \otimes  \hdots \otimes v_d \otimes g \in \mathbf{V}_{b,d}$, with $v_k\in  \R^{\Ib}$ and $g\in \R^{[0,1)}$. 
We have \begin{align*}
T_{b,\bar d}\circ T_{b,d}^{-1} \tensor{v}(i_1,\hdots,i_{\bar d},y) &= 
\tensor{v} ( t_{b,d}^{-1} \circ t_{b,\bar d}(i_1,\hdots,i_{\bar d},y) ) \\
&= \tensor{v}(i_1,\hdots,i_d,t_{b,\bar d -d}(i_{d+1},\hdots,i_{\bar d},y))\\
&= v_1(i_1) \hdots v_d(i_d) g(t_{b,\bar d -d}(i_{d+1},\hdots,i_{\bar d},y)) 
\\
&=  v_1(i_1) \hdots v_d(i_d) T_{b,\bar d-d} g(i_{d+1},\hdots,i_{\bar d},y)\\
&=(v_1\otimes \hdots \otimes v_d \otimes (T_{b,\bar d-d} g))(i_{d+1},\hdots,i_{\bar d},y),
\end{align*}
which proves the second property. The last property simply follows from $T_{b, d}\circ T_{b,\bar d}^{-1} = (T_{b,\bar d}\circ T_{b,d}^{-1} )^{-1} =  (id_{\{1,\hdots,d\}} \otimes T_{b,\bar d - d})^{-1} = id_{\{1,\hdots,d\}} \otimes T_{b,\bar d - d}^{-1}.$
\end{proof}

\begin{proof}[Proof of \Cref{ranks-partial-evaluations}]
\label{proof:partialeval}
$\tensor{f}$ is identified with a tensor in $\mathbf{V}_\beta \otimes\mathbf{V}_{\beta^c} $ with $\mathbf{V}_\beta \in (\R^{\Ib})^{\otimes \#\beta}$ and $\mathbf{V}_{\beta^c} = \mathbf{V}_{b,d-\#\beta}.$
We have 
$
\Umin{\beta^c}(\tensor{f}) =\{ ( \boldsymbol{\varphi}_{\beta} \otimes id_{\beta^c} )\tensor{f} : \boldsymbol{\varphi}_{\beta} \in (\mathbf{V}_{\beta})'\}
$
with $(\mathbf{V}_{\beta})'$ the algebraic dual of $\mathbf{V}_{\beta}$ (see \cite[Corollary 2.19]{Falco:2012uq}). Then for any basis $\{\boldsymbol{\varphi}_{\beta}^{j_\beta} : j_\beta \in \Ib^{\#\beta}\}$ of $(\mathbf{V}_{\beta})'$, we have
$
\Umin{\beta^c}(\tensor{f}) = \mathrm{span}\{ ( \boldsymbol{\varphi}_{\beta}^{j_\beta}\otimes id_{\beta^c})\tensor{f}  :j_\beta \in \Ib^{\#\beta} \}.
$
We conclude by introducing the particular basis $ \boldsymbol{\varphi}_{\beta}^{j_\beta} = \delta_{j_\beta}$, with $\delta_{j_\beta} = \otimes_{\nu \in \beta} \delta_{j_\nu} \in \R^{\Ib^{\#\beta}},$ and by noting that $(\delta_{j_\beta}\otimes id_{\beta^c}) \tensor{f} = \tensor{f}(j_\beta,\cdot) \in \mathbf{V}_{\beta^c}$. 
   \end{proof}

\begin{proof}[Proof of \Cref{lemma:rankgrowth}]
\label{proof:rankadmiss}
For any set $\beta \subset \{1,\hdots,d+1\}$ and any partition $\beta = \gamma \cup \alpha$, 
   the minimal subspaces from Definition \ref{def:minsub} satisfy the
    hierarchy property (see \cite[Corollary 6.18]{Hackbusch2012}) $\Umin{\beta}(\tensor{f}) \subset \Umin{\gamma}(\tensor{f}) \otimes \Umin{\alpha}(\tensor{f})$, from which we deduce that $r_{\beta}(\tensor{f})  \le r_{\gamma}(\tensor{f}) r_{\alpha}(\tensor{f}) $.  Then for $1\le \nu\le d -1$, by considering $\gamma = \{1,\hdots,\nu\}$ and $\alpha = \{\nu+1\}$, we obtain $r_{\nu+1}(\tensor{f}) \le r_{\nu}(\tensor{f}) r_{\{\nu+1\}}(\tensor{f})$, where $r_{\{\nu+1\}}(\tensor{f}) = \dim \Umin{\{\nu+1\}}(\tensor{f}) \le b$, which yields the first inequality. By considering $\gamma = \{\nu+1\}$ and $\alpha = \{\nu+2,\hdots,d+1\}$, we obtain $r_{\nu}(\tensor{f}) = r_{\{\nu+1,\hdots,d+1\}}(\tensor{f}) \le r_{\{\nu+1\}}(\tensor{f})r_{\{\nu+2,\hdots,d+1\}}(\tensor{f}) = 
  r_{\{\nu+1\}}(\tensor{f}) r_{\nu+1}(\tensor{f}) \le b r_{\nu+1}(\tensor{f}) , $ that is the second inequality.
    \end{proof}

\begin{proof}[Proof of \Cref{link-minimal-subspaces}]
\label{proof:liniminsub}
 We have  from \Cref{ranks-partial-evaluations} that 
$$
T_{b,d-\nu}^{-1}(\Umin{\{\nu+1,\hdots,d+1\}}(\tensor{f}^{d})) = \mathrm{span}\{T_{b,d-\nu}^{-1}(\tensor{f}^{d}(j_1,\hdots,j_\nu,\cdot)) : (j_1,\hdots,j_\nu)\in\Ib^\nu\},
$$
where $\tensor{f}^{d}(j_1,\hdots,j_\nu,\cdot) \in \mathbf{V}_{b,d-\nu}$ is a partial evaluation of $\tensor{f}^{d}$ along the first $\nu$ dimensions.  We note that
\begin{align*} 
    T_{b,d-\nu}^{-1}(\tensor{f}^{d}(j_1,\ldots,j_\nu,\cdot)) &= ((id_{\{1,\hdots,\nu\}} \otimes T_{b,d-\nu}^{-1})\tensor{f}^{d})(j_1,\ldots,j_\nu,\cdot) \\&= (T_{b,\nu}\circ T_{b,d}^{-1} \tensor{f}^{d})(j_1,\ldots,j_\nu,\cdot) = \tensor{f}^{\nu}(j_1,\ldots,j_\nu,\cdot),
\end{align*}
where the second equality results from \Cref{tbdbard}. The result then follows from  \Cref{ranks-partial-evaluations} again. 
\end{proof}

  \begin{proof}[Proof of \Cref{bijection-measurable}]
  \label{proof:bijection}
Subsets of the form $J \times A$, with $A$ a Borel set of $ [0,1)$ and  $J = \times_{k=1}^d J_k$ with 
$J_k \subset I_b$, $1\le k \le d$, form a generating system of the Borel $\sigma$-algebra of $I_b^d \times [0,1)$. The image of such a  set $J \times A$ through $t_{b,d}$ is $\bigcup_{j\in J} A_{j}$, where $ A_{j_1,\hdots,j_d} = b^{-d} ( j + A)$ with $j=\sum_{k=1}^{d} j_k b^{d-k}$. Then 
$$
\lambda(t_{b,d}(J\times A)) = \lambda (\bigcup_{j\in J} A_{j}) = \#J b^{-d}  \lambda(A) =   \#J_1 \hdots \#J_d b^{-d}  \lambda(A)=\mu_b(J_1)\hdots \mu_b(J_d) \lambda(A) = \mu_{b,d}(J\times A).
$$
Then, we conclude on $T_{b,d}$ by noting that it is a linear bijection (\Cref{Tbd-bijection}) which preserves measurability. 
\end{proof}

\begin{lemma}\label{lemma:reasonablecross}
Let $S$ be a closed subspace of $\Lp$, $1\le  p\le \infty$. 
The norm $\Vert \cdot \Vert_p$ is a reasonable crossnorm on $(\ell^p(\Ib)^{\otimes d}
 \otimes S$.
\end{lemma}
\begin{proof}
Let $v_k\in \ell^p(\Ib)$, $1\le k \le d$, and $g\in S$. For $ p <\infty$, we have 
$$
\Vert v_1\otimes \hdots \otimes v_d \otimes g \Vert_p^p =  \sum_{i_1\in \Ib} \hdots \sum_{i_d\in \Ib} \vert v_1(i_1)\vert^p \hdots  \vert v_d(i_d)\vert^p b^{-d} \int_{0}^1 \vert g(y) \vert^p dy = \Vert v_1 \Vert_{\ell^p}^p \hdots \Vert v_d \Vert_{\ell^p}^p \Vert g \Vert_p^p,
$$
and for $p=\infty$,
$$\Vert v_1\otimes \hdots \otimes v_d \otimes g \Vert_\infty = \max_{i_1 \in \Ib}  \vert v_1(i_1) \vert \hdots \max_{i_d\in \Ib}  \vert v_d(i_d) \vert \esssup_{y} \vert g(y) \vert =\Vert v_1 \Vert_{\ell^\infty} \hdots \Vert v_d \Vert_{\ell^\infty} \Vert g \Vert_\infty,$$
which proves that $\Vert \cdot\Vert_p$ is a crossnorm. 
Then, consider the dual norm $\Vert \varphi \Vert_p^* = \sup_{\Vert \tensor{f} \Vert \le 1} \vert \varphi(\tensor{f}) \vert $ over the algebraic tensor space $(\ell^p(\Ib)^*)^{\otimes d}  \otimes S^*$, where $V^*$ stands for the continuous dual of a space $V$. 
For $(v,\psi) \in \ell^p(\Ib)\times \ell^p(\Ib)^*$, we consider the duality pairing $\psi(v) = b^{-1}\sum_{k=0}^{b-1} \psi_{k} v_{k}$, such that $ \ell^p(\Ib)^* = \ell^q(\Ib)$ with $1/p + 1/q=1$.
 Consider $\phi \in S^*$ and $\varphi_\nu \in \ell^p(\Ib)^*$, $1\le \nu\le d$. To prove that $\Vert \cdot \Vert_p$ is a reasonable crossnorm, we have to prove that 
$$\Vert \varphi_{1} \otimes \hdots \otimes \varphi_{d}  \otimes \phi \Vert_{p}^* \le \Vert \varphi_1 \Vert_{\ell^q} \hdots \Vert \varphi_d \Vert_{\ell^q}  \Vert \phi \Vert_{p}^*  ,$$
with $ \Vert \phi \Vert_{p}^* = \sup_{f \in S, \Vert f\Vert_p\le 1}  \phi(f)$. 
 Let $\varphi = \varphi_{1} \otimes \hdots \otimes \varphi_{d} \in (\ell^p(\Ib)^*)^{\otimes d} = \ell^q(\Ib^d)$.
For $j\in \Ib^d$, we let $\delta_j = \delta_{j_1}\otimes \hdots \otimes \delta_{j_d} \in \ell^p(\Ib^d)$. 
 Any $\tensor{f} \in V_{b,d,S}$ admits a representation $\tensor{f} = \sum_{j\in \Ib^d} \delta_j \otimes g_j $ where 
 $g_j = \tensor{f}(j_1,\hdots,j_d,\cdot) \in \Lp,$ and  
$$
\vert ( \varphi_{1} \otimes \hdots \otimes \varphi_{d}  \otimes \phi)(\tensor{f}) \vert = \vert \varphi( \sum_{j\in \Ib^d} \delta_{j}   \phi(  g_{j}) \vert = \vert \varphi(\mathbf{v}) \vert
$$
where $\mathbf{v} \in \ell^p(\Ib^d)$ is a tensor with entries $\mathbf{v}(j) = \phi(  g_{j})$. Also,
$$
 \vert \varphi(\mathbf{v}) \vert \le \Vert \varphi \Vert_{\ell^p}^*  \Vert \mathbf{v} \Vert_{\ell^p}  \le \Vert \varphi \Vert_{\ell^q}  \Vert \phi \Vert_{p}^* \Vert \mathbf{w} \Vert_{\ell^p},
$$
where $\mathbf{w} \in \ell^p(\Ib^d)$ is a tensor with entries $\mathbf{w}(j) = \Vert   g_{j}\Vert_{p} = \Vert   \tensor{f}(j_1,\hdots,j_d,\cdot)\Vert_{p} .$ From \Cref{thm:tensorizationmap}, we have $\Vert \mathbf{w} \Vert_{\ell^\infty} = \max_{j \in \Ib^d} \Vert \tensor{f}(j_1,\hdots,j_d,\cdot) \Vert_\infty = \Vert \tensor{f} \Vert_\infty,$ and 
 for $p<\infty$
$$
\Vert \mathbf{w} \Vert_{\ell^p}^p = b^{-d} \sum_{j\in \Ib^d} \vert \mathbf{w}(j) \vert^p = b^{-d} \sum_{j\in \Ib^d}  \Vert  \tensor{f}(j_1,\hdots,j_d,\cdot)\Vert_{p}^p  =  \Vert \tensor{f} \Vert_p^p.
$$
Therefore, 
$
\vert ( \varphi_{1} \otimes \hdots \otimes \varphi_{d}  \otimes \phi)(\tensor{f}) \vert \le \Vert \varphi \Vert_{\ell^q}  \Vert \phi \Vert_{p}^*\Vert \tensor{f} \Vert_p.
$
We conclude by noting that $\Vert \cdot \Vert_{\ell^q}$ is a crossnorm on $\ell^q(\Ib^d) = \ell^q(\Ib)^{\otimes d}$, so that 
 $\Vert \varphi \Vert_{\ell^q(\Ib^d)} = \Vert \varphi_1 \Vert_{\ell^q} \hdots \Vert \varphi_d \Vert_{\ell^q} .$
\end{proof}

\begin{proof}[Proof of \Cref{b-adic-stability}]
\label{proof:badicstab}
By definition, the result is true for $d=1.$ The result is then proved by induction. Assume that for all $f\in S$, $f(b^{-d}(\cdot + k)) \in S$ for all $k\in \{0,\hdots,b^d-1\}$. Then for $f\in S$, consider 
 the function $f(b^{-d-1}(\cdot + k))$ with $k \in \{0,\hdots,b^{d+1}-1\}$. We can write $k = bk''+k'$ for some $k' \in \{0,\hdots,b-1\}$ and $k'' \in \{0,\hdots,b^d-1\}$. Then for any $x\in [0,1)$,
$$f(b^{-d-1}(x + k)) = f(b^{-d}(b^{-1}(x + k' ) + k''))  = g(b^{-1}(x + k'))$$
for some $g\in S$, and $g(b^{-1}(x + k')) = h(x) $ for some $h\in S$. Therefore $f(b^{-d-1}(\cdot + k)) = h(\cdot) \in S$, which ends the proof.
\end{proof}

\begin{proof}[Proof of \Cref{inclusion_Vbd}]
\label{proof:inclusion}
For $f \in S$, we have $(T_{b,1} f) (i_1,\cdot)= f(b^{-1}( \cdot+i_1)).$ Then from  \Cref{b-adic-stability}, we have $(T_{b,1} f) (i_1,\cdot) \in S$, which implies $f \in V_{b,1,S}.$ Now assume $f \in V_{b,d,S}$ for $d\in \N$, i.e. $T_{b,d}f = \tensor{f}^{d} \in \mathbf{V}_{b,d,S}$. Then $\tensor{f}^d(i_1,\hdots,i_{d},\cdot) \in S$ and from \Cref{b-adic-stability}, $\tensor{f}^d(i_1,\hdots,i_{d}, b^{-1}(i_{d+1} +\cdot)) \in S$. Then using \Cref{tbdbard}, we have that 
$\tensor{f}^d(i_1,\hdots,i_{d}, b^{-1}(i_{d+1} +\cdot)) = f\circ t_{b,d}(i_1,\hdots,i_{d}, t_{b,1}(i_{d+1} ,\cdot)) = f\circ t_{b,d+1}(i_1,\hdots,i_{d+1},\cdot)$, which implies that $(T_{b,d+1}f)(i_1,\hdots,i_{d+1},\cdot) \in S$, and therefore $f\in V_{b,d+1,S}.$
\end{proof}

\begin{proof}[Proof of \Cref{lemma:vbvectorspace}]
\label{proof:vectorspace}
Since $0 \in V_{b,d,S}$ for any $d$, we have $0\in V_{b,S}$.  For
$f_1,f_2 \in V_{b,S}$, there exists $d_1,d_2 \in \N$ such that $f_1\in\Vbd[b][d_1]{S}$ and
$f_2\in\Vbd[b][d_2]{S}$. Letting $d=\max\{d_1,d_2\}$, we have from \Cref{inclusion_Vbd} that 
$f_1,f_2\in V_{b,d,S}$, and therefore $cf_1+f_2\in\Vbd{S} \subset V_{b,S}$ for all $c \in \R$, which ends the proof.
\end{proof}

\begin{proof}[Proof of \Cref{thm:dense}]
\label{proof:dense}
The set of simple functions over $[0,1)$ is dense in $\Lp([0,1))$ for $1\le p<\infty$ (see, e.g., \cite[Lemma 4.2.1]{bogachev2007measure}). Then, it remains to prove that $\Vb{S}$ is dense in the set of simple functions over $[0,1)$. Consider a simple function $f = \sum_{i=0}^{n-1} a_i \indicator{[x_i,x_{i+1})} \neq 0$, with $0=x_0<x_1<\hdots<x_n=1,$ and $\Vert f \Vert_p^p = \sum_{i=0}^{n-1} \vert a_i \vert^p (x_{i+1}-x_i)$. Let $x_i^{d} = b^{-d} \lfloor b^d x_i \rfloor$, $0\le i \le n$, and consider the function $f_{d} = \sum_{i=0}^{n-1} a_i \indicator{[x_i^{d},x_{i+1}^{d})}$ which is such that  $f_{d} \in \Vbd{S}.$ Then, noting that $x_0^{d} = x_0 =0$ and $x_{n}^{d}=x_n=1$, it holds 
\begin{align*}
f - f_{d} &= \sum_{i=0}^{n-1} a_i(\indicator{[x_i,x_{i+1})} - \indicator{[x_i^{d},x_{i+1}^{d})}) =  \sum_{i=0}^{n-1} a_i(\indicator{[x_{i+1}^{d},x_{i+1})} -\indicator{[x_i^{d},x_{i})}  ) = \sum_{i=0}^{n-2} (a_i - a_{i+1})\indicator{[x_{i+1}^{d},x_{i+1})}. 
\end{align*}
Then, noting that $0\le  x_i - x_i^{d} \le b^{-d}$ for all $0 < i < n$, we have 
$$\Vert f-f_{d} \Vert_p^p = \sum_{i=0}^{n-2} \vert a_i - a_{i+1} \vert^p (x_{i+1} - x_{i+1}^{d}) \le 2^p b^{-d} \sum_{i=0}^{n-1} \vert a_i\vert^p =2^p b^{-d} \Vert f \Vert^p_p \big( \min_{0\le i\le n-1} (x_{i+1}-x_i)\big)^{-1}, $$
so that $\Vert f-f_{d} \Vert_p \to 0$ as $d \to \infty$, which ends the proof.
\end{proof}

\begin{proof}[Proof of \Cref{rank-bound-VbdS-closed-dilation}]
\label{proof:rankbound} 
(i) If $f\in S,$ from \Cref{inclusion_Vbd}, we have $f \in V_{b,\nu,S}$ for any $\nu$, so that $r_{\nu,\nu}(f) \le \dim S$. Then, using \Cref{link-ranks-d-nu}, we have $r_{\nu,d}(f) = r_{\nu,\nu}(f) \le \dim S$. The other bound $r_{\nu,d}(f)\le b^\nu$ results from \Cref{rank-bound-VbdS}. \\
(ii) The fact that  $f \in V_{b,\bar d,S}$ follows from \Cref{inclusion_Vbd}. 
Then, from \Cref{link-ranks-d-nu}, we have that $r_{\nu,\bar d}(f) = r_{\nu,\nu}(f)$ for all $1 \le \nu \le \bar d$.  
For $1 \le \nu \le d$, \Cref{link-ranks-d-nu} also implies $r_{\nu,\nu}(f) = r_{\nu,d}(f)$ and we obtain the desired inequality from  \Cref{rank-bound-VbdS}. For $\nu> d$, we note that 
$r_{\nu,\nu}(f) = \dim \Umin{\{\nu+1\}}(T_{b,\nu} f)$. From \Cref{inclusion_Vbd}, we know that $T_{b,\nu} f \in V_{b,\nu,S}$ for $\nu\ge d$, so that $\Umin{\{\nu+1\}}(T_{b,\nu} f) \subset S$ and $r_{\nu,\nu}(f) \le \dim(S)$. The other bound $r_{\nu,\bar d}(f)\le b^\nu$ results from \Cref{rank-bound-VbdS}.
 \end{proof}

\begin{proof}[Proof of \Cref{local-projection-structure}]
\label{proof:localproj}
Let $\{\phi_l\}_{1\le l \le \dim S}$ be a basis of $S$, such that for $g\in\Lp$, $\mathcal{I}_S (g) = \sum_{l=1}^{\dim S} \phi_l \sigma_l(g),$ with $\sigma_l$ a linear map from $\Lp$ to $\R$.  For $f \in \Lp$, and $x\in [b^{-d} j, b^{-d}(j+1))$,  
$$
\mathcal{I}_{b,d,S} f (x)= \sum_{l=1}^{\dim S} \phi_l(b^{d} x - j) \sigma_l(f(b^{-d} (j+\cdot )). 
$$
We have $f(b^{-d} (j+\cdot ) = \boldsymbol{f}(j_1,\hdots,j_d,\cdot)$, with $j= \sum_{k=1}^d b^{d-k} j_k$ and $\boldsymbol{f} = T_{b,d}f$, so that 
$$
T_{b,d} (\mathcal{I}_{b,d,S} f)(j_1,\hdots,j_d,y) =  \sum_{l=1}^{\dim S} \phi_l(y) \sigma_l(\boldsymbol{f}(j_1,\hdots,j_d,\cdot)).
$$ 
For $\boldsymbol{f} = \varphi_1 \otimes \hdots \varphi_d \otimes g$, using the linearity of $\sigma_l$, we then have 
$$
T_{b,d} ( \mathcal{I}_{b,d,S} (T_{b,d}^{-1} \boldsymbol{f} ))(j_1,\hdots,j_d,y) = \varphi_1(j_1) \hdots  \varphi_d(j_d)
  (\sum_{l=1}^{\dim S} \phi_l(y) \sigma_l(g)) = \varphi_1(j_1) \hdots  \varphi_d(j_d)  \mathcal{I}_S(g)(y),
$$
 which proves \eqref{tensor-structure-local-projection}.
\end{proof}

\section{Proofs for \Cref{sec:tensorapp}}

\begin{proof}[Proof of \Cref{prop:maxrank}]
\label{proof:maxrank}
    (ii). 
Let $\varphi_A,\;\varphi_B\in\tool$ with
    $\varphi_A\in\Vbd[@][d_A]{S}$, $\varphi_B\in\Vbd[@][d_B]{S}$
    and w.l.o.g.\ $d_B\geq d_A$. 
    Set $r_A:=\rmax(\varphi_A)$ and $r_B:=\rmax(\varphi_B)$.
         Then,
    \begin{align*}
        \cost(\varphi_A+\varphi_B)&\leq bd_A(\max(r_A, \dim S)+r_B)^2+(\max(r_A, \dim S)+r_B)\dim S\\
        &\leq
        2bd_Ar_B^2+4bd_Ar_A^2+4bd_A(\dim S)^2+r_A\dim S+r_B\dim S+(\dim S)^2\\
        &\leq [4+4(\dim S)^2+\dim S]n+4n^2
        \leq [8+4(\dim S)^2+\dim S]n^2.
    \end{align*}
    
    (i). Let $W_0$ denote the principal Branch of the Lambert $W$ function.
    Take $n\in\N$ large enough such that
    \begin{align*}
        d_A:=\left\lfloor\frac{1}{\ln(b)}W_0\left[
        n\frac{\ln(b)}{{2}\max\set{b,\dim S}}\right]\right\rfloor\geq 2,\quad
        d_A:=\left\lfloor\frac{n-\dim S}{b}\right\rfloor\geq 2
    \end{align*}
    Pick a full-rank function $\varphi_A\in\Vbd[@][d_A]{S}$ such that
    $r^2_A:=\rmax^2(\varphi_A)= b^{2\left\lfloor\frac{d_A}{2}\right\rfloor}$.
    Then,
    \begin{align*}
        \cost(\varphi_A)\leq
        bd_Ab^{d_A}+b^{\frac{d_A}{2}}\dim S\leq
        {2} \max\set{b,\dim S}d_Ab^{d_A} \le n,
    \end{align*}
    by the choice of $d_A$ and the properties of the Lambert $W$ function.
    
    Pick any $\varphi_B\in\Vbd[@][d_B]{S}$ with $r_B:=\rmax(\varphi_B)=1$ and $d_B=d_A$, so that $\cost(\varphi_B)=bd_B+\dim S \le n$.
    Then, $\varphi_A,\;\varphi_B\in\tool$.
    On the other hand, $r_A\geq r_B$ and
    from \cite{Hoorfar2008} we can estimate the Lambert $W$ function from below
    as
    \begin{align*}
        W_0\left[
        n\frac{\ln(b)}{{2}\max\set{b,\dim S}}\right]\geq
        \ln\left[
        n\frac{\ln(b)}{{2}\max\set{b,\dim S}}\right]-
        \ln\ln\left[
        n\frac{\ln(b)}{{2}\max\set{b,\dim S}}\right]
    \end{align*}
  Then 
      \begin{align*}
        \cost(\varphi_A+\varphi_B)&\geq bd_Ar_A^2+r_A\dim S\\
        &\geq b\left(\frac{n-\dim S}{b}-1\right)
        \left(b^{\frac{1}{\ln(b)}W_0\left[
        n\frac{\ln(b)}{{2}\max\set{b,\dim S}}\right]-1}\right)\\
        &=\left(\frac{n-\dim S}{b}-1\right)
        \left(n\frac{\ln(b)}{{2}\max\set{b,\dim S}}\right)
        \left[\ln\left(n\frac{\ln(b)}{{2}\max\set{b,\dim S}}\right)\right]^{-1},
    \end{align*}
    The leading term in the latter expression is
    \begin{align*}
        \frac{\ln(b)}{{2b \max\set{b,\dim S}}} n^2\left[\ln\left(n\frac{\ln(b)}{{2}\max\set{b,\dim S}}\right)\right]^{-1}.
    \end{align*}
    This cannot be bounded by $cn$ for any $c>0$ and thus
     \ref{prop:rmaxi} follows.
\end{proof}

\begin{proof}[Proof of \Cref{lemma:p4-N}]
\label{proof:p4N}
   Let $\varphi_A,\,\varphi_B\in\tool^{\mathcal{N}}$ with
    $d_A:=d(\varphi_A)$, $d_B:=d(\varphi_B)$, $\bs r^A := \bs r^A(\varphi_A)$, $\bs r^B := \bs r^B(\varphi_B)$ and
    w.l.o.g.\ $d_A\leq d_B$.
     Then  using \Cref{rank-bound-VbdS-closed-dilation},
      \begin{align*}
    \cost_{\mathcal{N}}(\varphi_A+\varphi_B)&\leq \sum_{\nu=1}^{d_B} (r^A_\nu + r^B_\nu) 
    \le \sum_{\nu=1}^{d_A} r_\nu^A + (d_B-d_A) \dim S + \sum_{\nu=1}^{d_B} r^B_\nu \\
    &\le   \cost_{\mathcal{N}}(\varphi_A) + \cost_{\mathcal{N}}(\varphi_B) (1+ \dim S) \le (2+\dim S) n.
    \end{align*}	
\end{proof}

\begin{proof}[Proof of \Cref{lemma:p4-C}]
\label{proof:p4C}
   Let $\varphi_A,\,\varphi_B\in\tool^{\mathcal{C}}$ with
    $d_A:=d(\varphi_A)$, $d_B:=d(\varphi_B)$, $\bs r^A := \bs r^A(\varphi_A)$, $\bs r^B := \bs r^B(\varphi_B)$ and
    w.l.o.g.\ $d_A\leq d_B$.
  Then
      \begin{align*}
 &   \cost_{\mathcal{C}}(\varphi_A+\varphi_B)\leq b(r_1^A+r_1^B)+
       \sum_{k=2}^{d_B}b(r^A_{k-1}+r^B_{k-1})(r_k^A+r_k^B) + (r_{d_B}^A+r_{d_B}^B) \dim S\\
       = & \underbrace{br_1^A+\sum_{k=2}^{d_A} b r^A_{k-1}r_k^A + r_{d_A}^A \dim S}_{N_1}  + \underbrace{ br_1^B+\sum_{k=2}^{d_B} b r^B_{k-1}r_k^B+b r^B_{d_B}\dim S}_{N_2}  +  \underbrace{ \sum_{k=1}^{d_A} br^A_{k-1} r_k^B + br^B_{k-1} r_k^A}_{N_3} \\
       & + \underbrace{\sum_{k=d_A+1}^{d_B} b r^A_{k-1}r_k^A}_{N_4} + \underbrace{(r_{d_B}^A - r_{d_A}^A) \dim S}_{N_5}   + \underbrace{ \sum_{k=d_A+1}^{d_B} b r^A_{k-1} r_k^B + b r^B_{k-1} r_k^A}_{N_6} .
    \end{align*}	
  Since $\varphi_A,\,\varphi_B\in\tool^{\mathcal{C}}$, we have $N_1 = \cost_{\mathcal{C}}(\varphi_A) \le n$ and $N_2 = \cost_{\mathcal{C}}(\varphi_B)\le n$. Then, using 	\Cref{lemma:rankgrowth}, we have 
 \begin{align*}
  N_3 \le&  b \Big(\sum_{k=2}^{d_A}(r_{k-1}^A)^2\Big)^{1/2}
        \Big(\sum_{k=2}^{d_A}(r_k^B)^2\Big)^{1/2}
        +
        b\Big(\sum_{k=2}^{d_A}(r_{k-1}^B)^2\Big)^{1/2}
        \Big(\sum_{k=2}^{d_A}(r_k^A)^2\Big)^{1/2}\\
        \le& b\Big(\sum_{k=2}^{d_A} b r_{k-1}^A r_k^A\Big)^{1/2}
        \Big(\sum_{k=2}^{d_A} b r_{k-1}^B r_k^B)^2\Big)^{1/2}
        +
        b\Big(\sum_{k=2}^{d_A} b r_{k-1}^Br_{k}^B\Big)^{1/2}
        \Big(\sum_{k=2}^{d_A} b  r_{k-1}^A r_k^A\Big)^{1/2}\\
        \le& 2b \cost_{\mathcal{C}}(\varphi_A)^{1/2} \cost_{\mathcal{C}}(\varphi_B)^{1/2} \le 2bn.
\end{align*}
If $d_A=d_B$, we have $N_4 = N_5 = N_6=0$. If $d_A<d_B$,  
   using \Cref{rank-bound-VbdS-closed-dilation}, we have 
   \begin{align*}
   &N_4 \le (\dim S)^2 b (d_B - d_A) \le (\dim S)^2 \cost_{\mathcal{C}}(\varphi_B) \le n (\dim S)^2 ,\\
   &N_5 \le (\dim S)^2 \le (\dim S) \cost_{\mathcal{C}}(\varphi_A) \le n \dim S,\\
   &N_6 \le  (\dim S) \big( \sum_{k=d_A+1}^{d_B} b r_k^B +  b r^B_{k-1} \big) \le 2 (\dim S) \cost_{\mathcal{C}}(\varphi_B) \le 2n \dim S.
   \end{align*}
    Thus, putting all together
    \begin{align*}
        \cost_{\mathcal{C}}(\varphi_A+\varphi_B)\leq [(\dim S)^2+3\dim S+2b+2]n,
    \end{align*}
    and \ref{P4} is satisfied with
    $c:=(\dim S)^2+3\dim S+2b+2$.
\end{proof}

\begin{proof}[Proof of \Cref{ext-sparse}]
\label{proof:extsparse}
We have the representation
\begin{align*}
&T_{b,d}(\varphi)(i_1,\hdots,i_d,y) = \sum_{k_1=1}^{r_1}\cdots\sum_{k_d=1}^{r_d}\sum_{q=1}^{\dim S}v_1^{k_1}(i_1)   \cdots v_d^{k_{d-1},k_d}(i_d) v^{k_d,q}_{d+1} \varphi_{q}(y).\end{align*}
    Then, from \Cref{tbdbard}, 
 \begin{align*}
&T_{b,\bar d}(\varphi)(i_1,\hdots,i_{\bar d},y)  \\
&= \sum_{k_1=1}^{r_1}\cdots\sum_{k_d=1}^{r_d}\sum_{q=1}^{\dim S}v_1^{k_1}(i_1) 
    \cdots v_d^{k_{d-1},k_d}(i_d) v^{k_d,q}_{d+1} T_{b,\bar d-d}(\varphi_{q})(i_{d+1},\hdots,i_{\bar d},y)\\
    &= \sum_{k_1=1}^{r_1}\cdots\sum_{k_d=1}^{r_d}\sum_{q=1}^{\dim S}\sum_{j_{d+1}=1}^{b} v_1^{k_1}(i_1)  
    \cdots v_d^{k_{d-1},k_d}(i_d) \underbrace{v^{k_d,q}_{d+1} \delta_{j_{d+1}}(i_{d+1})}_{\bar v^{k_d,(q,j_{d+1})}_{d+1}(i_{d+1})} T_{b,\bar d-d}(\varphi_{q})(j_{d+1},i_{d+2},\hdots,i_{\bar d},y),
    \end{align*}
    where $\bar v_{d+1} \in \R^{\alert{b}\times r_d \times (b\dim S)}$. 
Since $S$ is closed under $b$-adic dilation, we know from \Cref{rank-bound-VbdS-closed-dilation}  that  $r_{\nu}( \varphi_{q}) \le \dim S$ for all $\nu \in \N$. Let $l=\bar d-d$ and first assume $l\ge 2$.  Then, 
$T_{b,\bar d-d}(\varphi_{q})$ admits a representation
 \begin{align*}&
T_{b,\bar d-d}(\varphi_{q})(j_{d+1},i_{d+2},\hdots,i_{\bar d},y) \\
&= \sum_{\alpha_{1}=1}^{\dim S} \hdots  
\sum_{\alpha_{l}=1}^{\dim S}  \sum_{p=1}^{\dim S}  
\alert{w}^{q,\alpha_{1}}_1(j_{d+1}) \alert{w}^{q,\alpha_{1},\alpha_{2}}_{2}(i_{d+2})  \hdots \alert{w}^{q,\alpha_{l-1},\alpha_{l}}_{l}(i_{\bar d}) 
\alert{w}^{q,\alpha_{l},p}_{l+1} \varphi_{p}(y) 
\\
&= \sum_{\alpha_{2}=1}^{\dim S} \hdots  
\sum_{\alpha_{l}=1}^{\dim S}  \sum_{p=1}^{\dim S}
 \alert{w}^{q,\alpha_{2}}_{1,2}(j_{d+1}, i_{d+2})  \hdots \alert{w}^{q,\alpha_{l-1},\alpha_{l}}_{l}(i_{\bar d}) 
\alert{w}^{q,\alpha_{l},p}_{l+1} \varphi_{p}(y) 
   \end{align*}
   with $\alert{w}^{q,\alpha_{2}}_{1,2}(j_{d+1}, i_{d+2}) = \sum_{\alpha_1=1}^{\dim S} \alert{w}^{q,\alpha_{1}}_1(j_{d+1}) \alert{w}^{q,\alpha_{1},\alpha_{2}}_{2}(i_{d+2})$. Then, 
 \begin{align*}
 &T_{b,\bar d-d}(\varphi_{q})(j_{d+1},i_{d+2},\hdots,i_{\bar d},y) \\
& = \sum_{\alpha_{2},q_2=1}^{\dim S} \hdots  
\sum_{\alpha_{l},q_{l}=1}^{\dim S}  \sum_{p=1}^{\dim S}
\underbrace{\delta_{q,q_2}\alert{w}^{q,\alpha_{2}}_{1,2}(j_{d+1}, i_{d+2})}_{\bar v^{(q,j_{d+1}),(q_2,\alpha_2)}_{d+2}(i_{d+2})}   \hdots \underbrace{ \delta_{q_{l-1},q_{l}}\alert{w}^{q_{l-1},\alpha_{l-1},\alpha_{l}}_{l}(i_{\bar d})}_{\bar v^{(q_{l-1},\alpha_{l-1}),(q_l,\alpha_l)}_{\bar d}(i_{\bar d})} 
\underbrace{\alert{w}^{q_{l},\alpha_{l},p}_{l+1}}_{\bar v_{\bar d+1}^{(q_l,\alpha_l),p}} \varphi_{p}(y) 
    \end{align*}
    with $\bar v_{d+2} \in \R^{b\times (b\dim S) \times (\dim S)^2}$, $\bar v_\nu \in \R^{b\times (\dim S)^2\times (\dim S)^2}$   for $d+3\le \nu\le \bar d$,  and $\bar v_{\bar d+1} \in \R^{(\dim S)^2\times \dim S}$. Then, we have 
$\varphi \in \mathcal{R}_{b,\bar d,S,\overline{\bs r}}(\overline{\mathbf{v}})$ with $\overline{\mathbf{v} }= (\bar v_1,\hdots,v_{d},\bar v_{d+1}, \hdots,\bar v_{\bar d+1}) $, with $\bar v_\nu$ defined above for $\nu>d,$ and $\bar r_\nu = r_\nu$ for $\nu \le d$, $\bar r_{d+1} = b \dim S$, and $\bar r_{\nu} = (\dim S)^2$ for $d+1<\nu \le \bar d$. 
 From the definition of $\bar v_\nu$, we easily deduce that $\Vert \bar v_{d+1} \Vert_{\ell_0} =\alert{b}\Vert v_{d+1} \Vert_{\ell_0}$, $\Vert \bar v_{d+2} \Vert_{\ell^0}\le b^2 (\dim S)^2$, $\Vert \bar v_\nu \Vert_{\ell_0} \le
 \alert{b}(\dim S)^3$ for $d+3\le \nu\le \bar d$,  and $\Vert \bar v_{\bar d+1} \Vert_{\ell_0} \le (\dim S)^3$. Then, for $ l =\bar d - d \ge 2$, 
 we obtain 
 \begin{align*}
    \cost_{\mathcal{S}}(\overline{\mathbf{v}})& = \sum_{\nu=1}^{d+1} \Vert \bar v_\nu \Vert_{\ell_0} +  \sum_{\nu=d+2}^{\bar d+1} \Vert \bar v_\nu \Vert_{\ell_0} \le  \alert{b}\cost_{\mathcal{S}}({\mathbf{v}}) + b^2 (\dim S)^2 +  \alert{b}(\dim S)^3 (\bar d - d - \alert{2})+\alert{(\dim S)^3}.
  \end{align*}
For $l =1$, we have a representation
 \begin{align*}
 T_{b,\bar d-d}(\varphi_{q})(j_{d+1},y)  =  \sum_{p=1}^{\dim S}  \bar v_{d+2}^{(q,j_{d+1}),p} \varphi_p(y)
    \end{align*}
    with $\bar v_{d+2} \in \R^{\alert{(b\dim S)\times\dim S}}$ such that $\bar v_{d+2}^{(q,j_{d+1}),p} =  \sum_{\alpha_1=1}^{\dim S}\alert{w}^{q,\alpha_{1}}_1(j_{d+1})
\varphi^{q,\alpha_{1},p}_{2} $. Then, for $l=1$, $\varphi \in \mathcal{R}_{b,\bar d,S,\overline{\bs r}}(\overline{\mathbf{v}})$ with $\overline{\mathbf{v} }= (\bar v_1,\hdots,v_{d},\bar v_{d+1},\bar v_{\bar d+2}) $, and 
 \begin{align*}
    \cost_{\mathcal{S}}(\overline{\mathbf{v}})& \le  \alert{b}\cost_{\mathcal{S}}({\mathbf{v}}) + \alert{b(\dim S)^2} = \alert{b}\cost_{\mathcal{S}}({\mathbf{v}}) + \alert{b(\dim S)^2 (\bar d-d)}.
  \end{align*}
  For any $l\ge 1$, we then deduce 
   \begin{align*}
    \cost_{\mathcal{S}}(\overline{\mathbf{v}})& \le  \alert{b} \cost_{\mathcal{S}}({\mathbf{v}}) + \alert{b^2(\dim S)^3(\bar d - d)}.
  \end{align*}
\end{proof}

\begin{proof}[Proof of \Cref{sum-functions-sparse}]
\label{proof:sumsparse}
$\varphi_A$ and $\varphi_B$ admit representations
\begin{align*}
&T_{b,d}(\varphi_C)(i_1,\hdots,i_d,y) = \sum_{k_1=1}^{r_1^C}\cdots\sum_{k_d=1}^{r_d^C}\sum_{q=1}^{\dim S} v_{1}^{C,k_1}(i_1)   \cdots v_d^{C,k_{d-1},k_d}(i_d) v^{C,k_d,q}_{d+1} \varphi_{q}(y),\end{align*}
with $C=A$ or $B$. Then, $\varphi_A+\varphi_B$ admit the representation 
\begin{align*}
&T_{b,d}(\varphi_A + \varphi_B)(i_1,\hdots,i_d,y) = \sum_{k_1=1}^{r_1^A+r_1^B}\cdots\sum_{k_d=1}^{r_d^A+r^B_d}\sum_{q=1}^{\dim S} v_{1}^{k_1}(i_1)   \cdots v_d^{k_{d-1},k_d}(i_d) v^{k_d,q}_{d+1} \varphi_{q}(y),\end{align*}
with $v_1^{k_1} = v_1^{A,k_1}$ if $1\le k_1\le r_1^A$ and $v_1^{k_1} = v_1^{B,k_1}$ if $r_1^A < k_1\le r_1^A +  r_1^B$, $$v_\nu^{k_{\nu-1},k_\nu} =  \begin{cases} 
v^{A,k_{\nu-1},k_\nu} & \text{if  $1\le k_{\nu-1} , k_\nu \le r_1^A$}\\ 
v^{B,k_{\nu-1},k_\nu} & \text{if  $r_1^A < k_{\nu-1} , k_\nu \le r_1^A + r_1^B$}\\
0 & \text{elsewhere},
\end{cases} $$
and $v_{d+1}^{k_d,q} = v_{d+1}^{A,k_d,q}$ if $1\le k_d\le r_1^A$ and $v_{d+1}^{k_d,q} = v_{d+1}^{B,k_d,q}$ if $r_1^A < k_d \le r_1^A +  r_1^B$. From the above, we deduce that $\Vert v^C \Vert_{\ell_0} \le\Vert v^A \Vert_{\ell_0}  + \Vert v^B \Vert_{\ell_0}$, so that 
$$
 \cost_{\mathcal{S}}(\mathbf{v}) = \sum_{\nu=1}^{d+1} \Vert v^C \Vert_{\ell_0}  \le \cost_{\mathcal{S}}(\mathbf{v}) \le \cost_{\mathcal{S}}(\mathbf{v}_A) + \cost_{\mathcal{S}}(\mathbf{v}_B).
$$
\end{proof}

\begin{proof}[Proof of \Cref{lemma:p6}]
\label{proof:p6}
    The norm defined in \Cref{thm:tensorizationmap}
    is a reasonable crossnorm (see \Cref{lemma:reasonablecross}) and
    thus, in particular, not weaker than the injective norm on $\Vbd{S}$. Thus,
    by \cite[Lemma 8.6]{Hackbusch2012}
    $\TT{\Vbd{S}}{\bs r}$ for $\bs r\in\N^d$ is a weakly closed subset
    of $\Lp$.
    Moreover, the set $\tool$, with either $\tool=\tool^{\mathcal{N}}$ or $\tool^{\mathcal{C}}$, is a finite union of the sets $\TT{\Vbd{S}}{\bs r}$
    for different $d\in\N$ and $\bs r\in\N^d$.
    Since finite unions of closed sets (in the weak topology) are closed,
    it follows that $\tool$ is weakly closed in $\Lp$, and a fortiori, $\tool$ is also closed in the strong topology.
    Since $\Lp$ is reflexive for $1<p<\infty$ and $\tool$ is weakly closed,
    $\tool$ is proximinal in $\Lp$
    (see \cite[Theorem 4.28]{Hackbusch2012}).

Now consider that $S$ is finite-dimensional. There exists $d$ such that $\tool \subset \Vbd{S}$ and $\Vbd{S}$ is finite-dimensional. Since $\tool$ is a closed subset of a finite-dimensional space $\Vbd{S}$, it is proximinal in $\Lp$ for any $1\le p\le \infty$.
\end{proof}

\begin{proof}[Proof of \Cref{comparing-complexities}]
\label{proof:cfcomplex}
Consider a function $0\neq \varphi \in V_{b,S}$ and let $d=d(\varphi)$ and $\bs r = \bs r(\varphi)$.
We have $$\cost_{\mathcal{N}}({\varphi}) = \sum_{\nu=1}^{d} r_\nu  \le br_1 + \sum_{r_\nu=2}^{d} b r_{\nu-1} r_{\nu} + b \dim S = \cost_{\mathcal{C}}({\varphi}) ,$$ which implies $\Phi_n^{\mathcal{C}}  \subset  \Phi_n^{\mathcal{N}}$. Also 
\begin{align*}
\cost_{\mathcal{C}}(\varphi)& \le  br_1 + b  (\sum_{\nu=1}^{d-1} r_{\nu}^2)^{1/2}(\sum_{\nu=2}^{d} r_{\nu}^2)^{1/2}+ b \dim S\le  br_1 + b ( \sum_{\nu=1}^{d-1} r_{\nu})(\sum_{\nu=2}^{d} r_{\nu})+ b \dim S\\
&\le b (\sum_{\nu=1}^{d} r_{\nu})^2 + b \dim S = b \cost_{\mathcal{N}}({\varphi})^2 + b \dim S,
\end{align*}
which yields $\Phi_n^{\mathcal{N}} \subset \Phi_{b\dim S + b n^2}^{\mathcal{C}}$.
Also, we clearly have $
\cost_{\mathcal{S}}(\varphi) \le \cost_{\mathcal{C}}({\varphi})$, which implies $\Phi_n^{\mathcal{C}} \subset \Phi_n^{\mathcal{S}} $.
Now consider any tensor network $\mathbf{v} \in \mathcal{P}_{b,d,S,\bs r}$ such that $\varphi=\mathcal{R}_{b,d,S,\bs r}(\mathbf{v})$, with $d(\varphi) \le d$ and $\bs r(\varphi) \le \bs r$.
We have that $r_1(\varphi) \le \dim \{ v_1^{k_1}(\cdot) \in \R^{b} : 1\le k_1\le r_1\} \le \Vert v_1 \Vert_{\ell_0}$ and 
for $2\le \nu \le d $, $r_\nu(\varphi) \le \dim \{v_\nu^{\cdot,k_\nu}(\cdot) \in \R^{b\times r_{\nu-1}} : 1\le k_\nu \le r_\nu\} \le \Vert v_\nu \Vert_{\ell^0}$. Therefore $$\cost_{\mathcal{N}}(\varphi) = \sum_{\nu=1}^d r_\nu(\varphi) \le \sum_{\nu=1}^d \Vert v_\nu \Vert_{\ell^0} \le \cost_{\mathcal{S}}(\mathbf{v}).$$ The inequality being true for any tensor network $\mathbf{v}$ such that $\varphi=\mathcal{R}_{b,d,S,\bs r}(\mathbf{v})$, we deduce $\cost_{\mathcal{N}}(\varphi) \le \cost_{\mathcal{S}}(\varphi)$, which yields $ \Phi_n^{\mathcal{S}} \subset \Phi_n^{\mathcal{N}}$.
\end{proof}

\begin{proof}[Proof of \Cref{lemma:cpnop4}]
\label{proof:cpnop4}
(i). Consider $\varphi_A,\varphi_B \in \Phi^{\mathcal{R}}_n,$ and let $d_A = d(\varphi_A)$, $d_B = d(\varphi_B)$, $r_A = r(\varphi_A)$ and $r_B = r(\varphi_B). $ Assume w.l.o.g. that $d_A\le d_B.$ The function $\varphi_A$ admits a representation 
$$
T_{b,d_A}\varphi_A (i_1,\hdots,i_{d_A},y)= \sum_{k=1}^{r_A} w_{1}^{A,k}(i_1)\hdots w_{d}^{A,k}(i_d) w_{d+1}^{A,k}(y),
$$
and 
$$
T_{b,d_B}\varphi_A (i_1,\hdots,i_{d_B},y)= \sum_{k=1}^{r_A} w_{1}^{A,k}(i_1)\hdots w_{d}^{A,k}(i_d) T_{b,d_B}(w_{d+1}^{A,k})(y).
$$
From the assumption on $S$, we have $T_{b,d_B}(w_{d+1}^{A,k})$ of rank $1$, so that $r(T_{b,d_B}\varphi_A) \le r_A$. We easily deduce that $r(\varphi_A + \varphi_B) \le r_A + r_B$ and $\cost_{\mathcal{R}}(\varphi_A + \varphi_B) \le b d_B (r_A+r_B) + (r_A+r_B) b\dim S \le  2 n + b r_A (d_B - d_A) \le 2n + n^2 \le 3n^2.$ \\
\alert{(ii). The proof idea is analogous to \Cref{prop:maxrank}: we take
a rank-one tensor $\varphi_B\in\tool^{\mc R}$ such that $d_B\sim n$
and a full-rank tensor $\varphi_A\in\tool^{\mc R}$
with $d_A<d_B$ such that
$r_A\sim b^{d_A}\sim n$. Then, as in
\Cref{prop:maxrank}, $\cost_{\mc R}(\varphi_A+\varphi_B)\sim n^2$.}
\end{proof}

\begin{proof}[Proof of \Cref{lemma:canonical}]
\label{proof:canonical}
Let $\varphi \in \Phi^{\mathcal{R}}_n$, $d=d(\varphi)$, $r=r(\varphi)$.
The function $\varphi$ admits a representation 
$$
T_{b,d}\varphi (i_1,\hdots,i_{d},y)= \sum_{k=1}^{r} \sum_{q=1}^{\dim S}w_{1}^{k}(i_1)\hdots w_{d}^{k}(i_d) w_{d+1}^{q,k} \varphi_q(y).
$$
Letting $v_1 = w_1$,  $v_{d+1} =w_{d+1}$ and $v_\nu \in \R^{b\times r \times r} $ such that $v_\nu^{k_{\nu-1},k_\nu} = \delta_{k_{\nu-1},k_\nu} w_\nu^{k_\nu}$ for $2\le \nu \le d$, and letting $\bs r = (r,\hdots,r) \in \N^d$, we have 
$$
T_{b,d}\varphi (i_1,\hdots,i_{d},y) = \sum_{k_1=1}^{r} \hdots \sum_{k_d=1}^r \sum_{q=1}^{\dim S}v_{1}^{k}(i_1)\hdots w_{d}^{k_{d-1},k_d}(i_d) w_{d+1}^{q,k_d} \varphi_q(y),
$$
which proves that $T_{b,d}  \varphi \in \Phi_{b,d,S,\bs r}$ with 
$$
\cost_{\mathcal{S}}(\varphi) = \sum_{\nu=1}^{d+1} \Vert v_\nu \Vert_{\ell_0} =    \sum_{\nu=1}^{d+1} \Vert w_\nu \Vert_{\ell_0} \le brd + r \dim S = \cost_{\mathcal{R}}(\varphi) \le n,
$$
that is $\varphi \in \Phi^{\mathcal{S}}_n.$
\end{proof}

\end{document}